\documentclass[final]{siamltex}
\usepackage {epsfig,amsmath,euscript}
\usepackage[latin1]{inputenc}
\usepackage[english]{babel}
\usepackage{color}
\usepackage{amsmath}
\usepackage{amssymb}
\usepackage{ae}
\usepackage{subfigure}
\usepackage[T1]{fontenc}
\usepackage{graphicx}
\usepackage{epstopdf}
\usepackage{color}
\definecolor{rred}{rgb}{0.7,0.0,0.2}
\definecolor{bblue}{rgb}{0.2,0.0,0.7}

\newcommand{\secref}[1]{\ref{sec:#1}}

\newcommand{\seclab}[1]{\label{sec:#1}}
\newcommand{\eqlab}[1]{\label{eq:#1}}
\renewcommand{\eqref}[1]{(\ref{eq:#1})}
\newcommand{\eqsref}[2]{(\ref{eq:#1}) and~(\ref{eq:#2})}

\newcommand{\figref}[1]{Fig.~\ref{fig:#1}}
\newcommand{\figlab}[1]{\label{fig:#1}}
\newcommand{\propref}[1]{Proposition~\ref{proposition:#1}}
\newcommand{\proplab}[1]{\label{proposition:#1}}
\newcommand{\lemmaref}[1]{Lemma~\ref{lemma:#1}}
\newcommand{\lemmalab}[1]{\label{lemma:#1}}
\newcommand{\remref}[1]{Remark~\ref{remark:#1}}
\newcommand{\remlab}[1]{\label{remark:#1}}
\newcommand{\thmref}[1]{Theorem~\ref{theorem:#1}}
\newcommand{\thmlab}[1]{\label{theorem:#1}}
\newcommand{\tablab}[1]{\label{tab:#1}}
\newcommand{\tabref}[1]{Table~\ref{tab:#1}}
\newcommand{\defnlab}[1]{\label{defn:#1}}

\usepackage{tikz}
\newcommand*\circled[1]{\tikz[baseline=(char.base)]{
            \node[shape=circle,draw,inner sep=2pt] (char) {#1};}}
\title{On the use of blow up to study regularizations of singularities of piecewise smooth dynamical systems in $\mathbb{R}^3$}

\author{K. Uldall Kristiansen and S. J. Hogan\thanks{K. Uldall Kristiansen: Department of Applied Mathematics and Computer Science, Technical University of Denmark, 2800 Kgs. Lyngby, DK. S. J. Hogan: Department of Engineering Mathematics, University of Brisol, Bristol BS8 1UB, United Kingdom. S.J. Hogan wishes to thank both Danmarks Nationalbank and the Otto M{\o}nsteds Fond for support. In addition, he is extremely grateful to Morten Br{\o}ns for hosting a very productive sabbatical at DTU, Lyngby, Denmark from January to September 2014. }}
\begin{document}
\maketitle

%

\begin{abstract}

In this paper we use the blow up method of Dumortier and Roussarie \cite{dumortier_1991,dumortier_1993,dumortier_1996}, in the formulation due to Krupa and Szmolyan \cite{krupa_extending_2001}, to study the regularization of singularities of piecewise smooth dynamical systems \cite{filippov1988differential} in $\mathbb R^3$. Using the regularization method of Sotomayor and Teixeira \cite{Sotomayor96}, first we demonstrate the power of our approach by considering the case of a fold line. We quickly recover a main result of Bonet and Seara \cite{reves_regularization_2014} in a simple manner. Then, for the two-fold singularity, we show that the regularized system only fully retains the features of the singular canards in the piecewise smooth system in the cases when the sliding region does not include a full sector of singular canards. In particular, we show that every locally unique primary singular canard persists the regularizing perturbation. For the case of a sector of primary singular canards, we show that the regularized system contains a canard, provided a certain non-resonance condition holds. Finally, we provide numerical evidence for the existence of secondary canards near resonance. 

\end{abstract}

%



\pagestyle{myheadings}
\thispagestyle{plain}

\section{Introduction}

A piecewise smooth dynamical system \cite{filippov1988differential, MakarenkovLamb12} consists of a finite set of ordinary differential equations
\begin{equation}
\dot{\mathbf{x}}=f_i(\mathbf{x}), \quad \mathbf{x}\in R_i\subset \mathbb{R}^n
\eqlab{pwsdef}
\end{equation}
where the smooth vector fields $f_i$, defined on disjoint open regions $R_i$, are smoothly extendable to the closure of $R_i$. The regions $R_i$ are separated by an $(n-1)$-dimensional set $\Sigma$ called the \textit{switching boundary}, which consists of finitely many smooth manifolds intersecting transversely.  The union of $\Sigma$ and all $R_i$ covers the whole state space $D \subseteq \mathbb{R}^n$. {In this paper, we consider $n=3$. There are almost no results for piecewise smooth dynamical systems with $n>3$. We will consider $n=2$ in \cite{krkri_2014}.}

The study of piecewise smooth dynamical systems is important for a number of reasons. Firstly, they challenge the classical notion of solution in at least two distinct ways. When the normal components of the vector fields either side of $\Sigma$ are in the \textit{same} direction, the gradient of a trajectory is discontinuous, leading to Carath\'eodory solutions \cite{cortez2008}. In this case, the dynamics is described as \textit{crossing} or \textit{sewing}. But when the normal components of the vector fields either side of $\Sigma$ are in the \textit{opposite} direction, a vector field \textit{on} $\Sigma$ needs to be defined. The precise choice is not unique and crucially depends on the nature of the problem under consideration. One possibility is the use of differential inclusions. Another choice is to adopt the Filippov convention \cite{filippov1988differential}, where a \textit{sliding} vector field is defined on $\Sigma$. In this case, the dynamics is described as \textit{sliding}. For both Carath\'eodory and Filippov solutions, trajectories and separatrices, as well as notions of topological equivalence and bifurcation, all need revision and extension \cite{filippov1988differential}. 

Secondly, piecewise smooth dynamical systems are of great significance in applications \cite{Bernardo08}, ranging from problems in mechanics (friction, impact) and biology (genetic regulatory networks) to variable structure systems in control engineering where the idea of sliding mode control \cite{Utkin77} has been widely adopted.

It is therefore perhaps not surprising that the theory of the dynamics of piecewise smooth systems is fragmented and, certainly at the time of writing, lacks the coherence and clarity of understanding that characterizes a lot of the work done on smooth systems \cite{Guckenheimer97}. Nevertheless there are many areas of piecewise smooth systems where progress has been made. In particular, it is now known \cite{jeffrey_geometry_2011} that there are eight distinct codimension-1 \textit{sliding bifurcations} in $\mathbb R^3$. A sliding bifurcation occurs where the relative direction of the normal component of vector fields either side of $\Sigma$ is reversed under parameter variation. Crossing and sliding interchange, leading to fundamentally different dynamics. 

The approach used by \cite{jeffrey_geometry_2011} was geometric, exploiting known results from singularity theory. The fundamental objects involved in these sliding bifurcations are: the fold, the cusp and the two-fold. It was subsequently shown \cite{desroches_canards_2011} that these two-folds contain dynamics that is very similar to canards \cite{Benoit81}. An in-depth study \cite{Desroches13} of a planar slow-fast piecewise linear caricature of the van der Pol system showed that the ensuing limit cycles, dubbed \textit{quasi-canards}, have a great deal in common with van der Pol canards, in particular explosive growth under parameter variation. Another piecewise linear system was considered in \cite{Prohens13}, where the connection to canards was also made. It should be emphasized that in all three studies \cite{Desroches13, desroches_canards_2011, Prohens13}, the so-called critical manifold of the slow-fast system is piecewise smooth and the presence of folded points is not necessary to ensure the existence of canards. In fact, the canards themselves are piecewise smooth, involving sliding in  \cite{desroches_canards_2011} and crossing in  \cite{Desroches13,Prohens13}. 

The connection between the dynamics in smooth systems and in piecewise smooth systems is also of great interest, since behaviour in one type of system is often assumed to be ``close" to that in the other. For example, piecewise smooth functions \cite{Carmona08} are often used as caricatures of nonlinear functions \cite{Michelson86}. A nonlinear function is replaced by piecewise linear approximations and a set of simpler linear problems is solved instead \cite{casey06}. More recently, in a reversal of this trend, in many practical applications it is unfeasible to deal with a large number of switches between different dynamics and so piecewise smooth systems have been replaced by smooth counterparts \cite{ISI:000174961100007,ISI:000246954500016,ISI:000246954500026}. 

But is it true that the behaviour in a piecewise smooth system is ``close" to that in a corresponding smooth system, and if so, how close? An early theoretical result in this area for the case $n=2$ is due to Sotomayor and Teixeira \cite{Sotomayor96}. In the case of sliding using the Filippov convention, they proposed a regularization which involved replacing the switching manifold $\Sigma$ with a boundary layer of width $2\varepsilon$. Outside this layer, the regularization agreed with the piecewise smooth vector fields either side of $\Sigma$. Inside this layer, a suitable monotonic function was chosen, such that the regularization was differentiable everywhere. The case $n=3$ was considered by \cite{Llibre97} and the general case in \cite{llibre_sliding_2008}.  In \cite{Llibre07}, it was shown that the regularization process developed in \cite{Sotomayor96} produces a singular problem for which the discontinuous set is a critical manifold and that, for a class of vector fields, the sliding vector field coincides with the reduced problem of the corresponding singular problem. Other regularizations may possess some or all of these properties but, to date, no results have appeared in the literature.

But what about bifurcations? Are the bifurcations in piecewise smooth systems ``close" to bifurcations in a corresponding smooth system? Bonet and Seara \cite{reves_regularization_2014} considered the case of the fold, which is responsible for the sliding bifurcation known as \textit{grazing-sliding} (as well as other piecewise smooth phenomena). Here they encountered a fundamental difficulty in answering these questions. The fold gives rise to a nonhyperbolic point in the critical manifold. But Fenichel theory \cite{fen1, fen2, fen3, jones_1995} requires hyperbolicity. So Bonet and Seara \cite{reves_regularization_2014} needed to extend Fenichel theory close to the fold point, which they achieved using asymptotic methods. They showed that topological features of the piecewise smooth bifurcations were preserved under regularization. 

{A widely used approach in geometric singular perturbation theory to deal with loss of hyperbolicity is the \textit{blow up method}. Originally due to Dumortier and Roussarie \cite{dumortier_1991,dumortier_1993,dumortier_1996}, the method has been developed subsequently by Krupa and Szmolyan \cite{krupa_extending_2001}. It involves the application of a map that blows up nonhyperbolic points to a higher dimensional cylinder, then uses rescalings of the resulting vector field to desingularize the problem and obtain non-trivial dynamics on the cylinder (see the survey article by Alvarez \textit{et al.} \cite{Alvarez11} for more information on blow up methods). Krupa and Szmolyan applied their blow up method in planar slow-fast systems in \cite{krupa_extending_2001,krupa_extending_2001_nonlinearity} and were able to extend slow manifolds beyond fold, transcritical and pitchfork singularities. In \cite{krupa_extending_2001} the authors also described the geometry of planar canards. Since then the method has been applied to many problems. For example, it has been successfully applied in biochemistry \cite{kosiuk_scaling_2011,2014arXiv1403.5658K} and to study canards in $\mathbb R^3$ \cite{krupa_local_2010, szmolyan_canards_2001}.} Note that the term \textit{blow up} is also used in other areas of mathematics and can have different meanings, depending on the context. In particular, it is used in \cite{Llibre07,llibre_sliding_2008,Llibre97} to describe sliding motion, using the singular limit of a slow-fast system, but these references do not consider potential nonhyperbolic points. 

%
%


{In this paper, our main aim is to apply the blow up method \cite{dumortier_1991,dumortier_1993,dumortier_1996}, as formulated by Krupa and Szmolyan \cite{krupa_extending_2001}, to examine how (singular) canards associated with a two-fold behave under the regularization proposed by Sotomayor and Teixeira \cite{Sotomayor96}.} En route, we will recover a main result of Bonet and Seara \cite{reves_regularization_2014} in a simpler way. 

Our paper is structured as follows. Following preliminaries and the problem statement in section \secref{prelim}, we present the normalized equations of motion near a two-fold singularity of a piecewise smooth system in section \secref{normalizedEqs}. We then present, in section \secref{slidingVectorField}, the sliding vector field in this case and show how \textit{singular canards} occur naturally in the piecewise smooth system. In section \secref{regularize}, we present the regularization of our piecewise smooth system, following Sotomayor and Teixeira \cite{Sotomayor96}. In section \secref{foldLines}, we apply the blow up method to the case of a fold (line). In this way, we demonstrate the power of this approach whilst recovering results of Bonet and Seara \cite{reves_regularization_2014}. The main section of the paper is section \secref{twoFold}, where we apply the blow up method to the case of a two-fold, to show how canards seen in the piecewise smooth problem can persist under regularization. In \secref{numerics}, we present some numerical experiments showing the presence of secondary canards. Finally, we discuss future work in section \secref{future} and conclude in section \secref{conclusion}.

\section{Preliminaries and problem statement}
\seclab{prelim}
In this section, we introduce our notation and set up the problem. Let $\textbf x=(x,y,z)\in \mathbb R^3$ and consider an open set $\mathcal U$ and a smooth function $f=f(\textbf x)$ having $0$ as a regular value. Then $\Sigma\subset \mathcal U$ defined by $\Sigma=f^{-1}(0)$ is a smooth $2D$ manifold. The manifold $\Sigma$ will be our \textit{switching manifold}. It separates the set $\Sigma_+ =\{(x,y,z)\in \mathcal U\vert f(x,y,z)>0\}$ from the set $\Sigma_-=\{(x,y,z)\vert f(x,y,z)<0\}$. We introduce local coordinates so that $f(x,y,z)=y$ and $\Sigma=\{(x,y,z)\in \mathcal U\vert y=0\}$. 

We consider two smooth vector-fields $X^+$ and $X^-$ that are smooth on $\overline{\Sigma}_+$ and $\overline{\Sigma}_-$, respectively, and define the piecewise smooth vector-field $X=(X^-,X^+)$ by
\begin{eqnarray*}
  X(\textbf x) 
  =\left\{ \begin{array}{cc}
                                   X^-(\textbf x)& \text{for}\quad  \textbf x\in \Sigma_-\\
                                   X^+(\textbf x)& \text{for}\quad \textbf x\in \Sigma_+
                                  \end{array}\right.
\end{eqnarray*}
Then, as mentioned above, $\Sigma$ is divided into two types of region: crossing and sliding:
\begin{itemize}
 \item $\Sigma_{cr}\subset \Sigma$ is the \textit{crossing region} where $(X^+ f(x,0,z))(X^-f(x,0,z)) =X_2^+(x,0,z) X_2^-(x,0,z) >0$.
 \item $\Sigma_{sl}\subset \Sigma$ is the \textit{sliding region} where $(X^+ f(x,0,z))(X^-f(x,0,z)) =X_2^+(x,0,z) X_2^-(x,0,z)<0$.
\end{itemize}
Here $X^{\pm} f=\nabla f\cdot X^{\pm}$ denotes the Lie-derivative of $f$ along $X^{\pm}$. Since $f(x,y,z)=y$ in our coordinates we have simply that $X^{\pm} f =X_2^{\pm}$. On $\Sigma_{sl}$ we follow the Filippov convention \cite{filippov1988differential} and define the sliding vector-field as a convex combination of $X^+$ and $X^-$
\begin{eqnarray}
 X_{sl}(\textbf x) = \sigma X^+(\textbf x)+ (1-\sigma) X^-(\textbf x),\eqlab{XSliding}
\end{eqnarray}
where $\sigma \in (0,1)$ is defined so that the vector-field $X_{sl}(\textbf x)$ is tangent to $\Sigma_{sl}$:
\begin{eqnarray*}
 \sigma = \frac{X^-f(x,0,z)}{X^-f(x,0,z)-X^+f(x,0,z)}.\eqlab{lambdaSliding}
\end{eqnarray*}

From the above, it is clear that, in general, trajectories in $\Sigma_{\pm}$ can reach $\Sigma$ in finite time (backward or forward). Hence, since $X_{sl}(\textbf x)$ can have equilibria (usually called \textit{pseudoequilibira}, or sometimes \textit{quasiequilibira}), it is possible for trajectories to reach these equilibria in finite time, unlike in smooth systems. An orbit of a piecewise smooth system can be made up of a concatenation of arcs from $\Sigma$ and $\Sigma_{\pm}$.

The boundaries of $\Sigma_{sl}$ and $\Sigma_{cr}$ where $X^+ f = X_2^+=0$ or $X^-f = X_2^-=0$ are singularities called \textit{tangencies}. In what follows, we define two different types of generic tangencies: the fold and the two-fold. We will consider the cusp singularity in future work.
\begin{definition}
A point $q\in \Sigma$ is a fold singularity if 
\begin{eqnarray}
X^{+} f(q)=0,\quad X^+(X^+f)(q)\ne 0,\quad \text{and}\quad X^-f(q)\ne 0,\eqlab{foldPlus}
\end{eqnarray}
or 
\begin{eqnarray}
X^{-} f(q)=0,\quad X^-(X^-f)(q)\ne 0,\quad \text{and}\quad X^+f(q)\ne 0.\eqlab{foldNegative}
\end{eqnarray}
A point $p\in \Sigma$ is a two-fold singularity if both $X^{+}f(p)=0$ and $X^-f(p)=0$, as well as $X^+(X^+f)(p)\ne 0$ and $X^-(X^-f)(p)\ne 0$ and if 
the vectors $X^+(p)$ and $X^-(p)$ are not parallel. 
\end{definition}

We have
\begin{proposition}
The two-fold singularity is the transversal intersection of two lines $l^+$ and $l^-$ of fold singularities satisfying \eqref{foldPlus} and \eqref{foldNegative} respectively. 
\end{proposition}
\begin{proof}
This is a relatively simple application of the implicit function theorem. 

\end{proof}
%
%
Following this proposition it is therefore possible to introduce a new smooth set of coordinates, which we continue to denote by $\textbf{x}=(x,y,z)$, so that  $l^+$ and $l^-$ become the $x$ and $z$-axis respectively, namely $$l^+=\{\textbf{x}\in \mathcal U\vert y=0=z\}\quad \text{and} \quad l^-=\{\textbf{x}\in \mathcal U\vert x=0=y\},$$ possibly restricting $\mathcal U$ further. The two-fold singularity $p$ is then at the origin: $$p=(0,0,0).$$ We shall also continue to denote the new vector-field by $X=(X^-,X^+)$. Conditions \eqsref{foldPlus}{foldNegative} also imply that for this new vector-field the following inequalities hold
\begin{eqnarray*}
 X_3^+\vert_{l^+} \ne 0,\,X_1^-\vert_{l^-} \ne 0.
\end{eqnarray*}
 In particular:
 \begin{align}
  X_3^+(p) \ne 0,\,X_1^-(p)\ne 0.\eqlab{X1PositiveX3NegativeNEQ0}
 \end{align}

 For a fold, it is important to distinguish between the \textit{visible} and \textit{invisible} cases.
\begin{definition}\cite[Definition 2.1]{jeffrey_geometry_2011}
A fold singularity $q$ with $X^{+} f(q)=0$ or $X^{-} f(q)=0$ is \textnormal{visible} if
 \begin{eqnarray*}
\mbox{{$X^{+ }(X^{+ } f)(q) >0\quad \text{or} \quad X^{-}(X^{- } f)(q) <0,\quad \text{respectively}$}},
 \end{eqnarray*}
and \textnormal{invisible} if
\begin{eqnarray*}
\mbox{{$X^{+ }(X^{+ } f)(q) <0\quad \text{or}\quad X^{-}(X^{- } f)(q) >0,\quad \text{respectively}.$}}
\end{eqnarray*}
\end{definition}
\begin{definition} \cite[Definition 2.3]{jeffrey_geometry_2011}
 Similarly we say that the two-fold singularity $p$ is 
 \begin{itemize}
  \item \textnormal{visible} if the fold lines $l^+$ and $l^-$ are both visible;
  \item \textnormal{invisible-visible} if $l^+$ ($l^-$) is visible and $l^-$ ($l^+)$ is invisible;
  \item \textnormal{invisible} if $l^+$ and $l^-$ are both invisible. 
 \end{itemize}
\end{definition}

\section{Normalized equations of motion near a two-fold sliding bifurcation}\seclab{normalizedEqs}
In this section, we derive a normal form for the equations of motion near a two-fold singularity, in such a way that the sliding and crossing regions remain fixed under parameter variation. 

By Taylor expanding $X^{\pm}$ about the origin $p$ we consider the following systems:
\begin{eqnarray*}
 \dot x &=&X^+_1(p)+O(x+y+z),\\
 \dot y&=&\partial_y X^+_2(p)y+\partial_z X^+_2(p)z +R^+(x,y,z),\\
 \dot z&=&X^+_3(p)+O(x+y+z),
\end{eqnarray*}
for $y>0$ and
\begin{eqnarray*}
 \dot x &=&X^-_1(p)+O(x+y+z),\\
 \dot y&=&\partial_y X^-_2(p)y+\partial_z X^+_2(p)x +R^-(x,y,z),\\
 \dot z&=&X^-_3(p)+O(x+y+z),
\end{eqnarray*}
for $y<0$, with the quadratic $R^{\pm}=O(2)$ satisfying
\begin{eqnarray*}
 R^+(x,0,0)\equiv 0,\quad R^-(0,0,z)\equiv 0.
\end{eqnarray*}
Following \eqref{X1PositiveX3NegativeNEQ0} we introduce $(\tilde x,\tilde z)$  where
\begin{eqnarray*}
 x&=& \frac{\tilde x}{X_1^-(p)},\\
 z&=& \frac{\tilde z}{X_3^+(p)},
\end{eqnarray*}
which potentially reverses the orientation. Dropping the tildes, this gives the following equations:
\begin{eqnarray*}
 \dot x &=&c+O(x+y+z),\\
 \dot y&=&ay+bz +\mathcal O((y+z)(x+y+z)),\nonumber\\
 \dot z&=&1+O(x+y+z),\nonumber
\end{eqnarray*}
for $y>0$ and
\begin{eqnarray*}
 \dot x &=&1+O(x+y+z),\\
 \dot y&=&\alpha y-\beta  x +\mathcal O((x+y)(x+y+z)),\nonumber\\
 \dot z&=&\gamma +O(x+y+z),\nonumber
\end{eqnarray*}
for $y<0$.
The constants $b$ and $\beta$ are both non-zero (since the lines $y=0=z$ and $x=0=y$ are fold singularities). The signs of $b$ and $\beta$ determine the visibility of the folds but they also determine regions of sliding and crossing. In order to simplify the sequel, we introduce the following scalings to ensure that each region retains its original sliding or crossing characteristics under parameter variation. Hence
\begin{eqnarray*}
 x&=& \text{sign}(\beta) \tilde x,\\
 z&=& \text{sign}(b) \tilde z,\\
c&=&\beta^{-1}\tilde c,\\
 \gamma &=&b^{-1}\tilde \gamma,
\end{eqnarray*}
to obtain the system described in the following proposition (again, dropping the tildes):
\begin{proposition}\proplab{visibleInvisble}
 Consider a piecewise smooth system $X=(X^-,X^+)$ with a two-fold singularity $p$ where $X^+(p)$ and $X^-(p)$ are independent. Then within a sufficiently small neighborhood of $p$ the system can be transformed into the following \textnormal{normal form}:
\begin{eqnarray}
 \dot x &=&\vert \beta\vert^{-1} c+O(x+y+z),\eqlab{XPositive}\\
 \dot y&=&ay+\vert b\vert z +\mathcal O((y+z)(x+y+z)),\nonumber\\
 \dot z&=&\text{sign}(b)+O(x+y+z),\nonumber
\end{eqnarray}
for $y>0$ and
\begin{eqnarray}
 \dot x &=&\text{sign}(\beta)+O(x+y+z),\eqlab{XNegative}\\
 \dot y&=&\alpha y-\vert \beta\vert x +\mathcal O((x+y)(x+y+z)),\nonumber\\
 \dot z&=&\vert b \vert^{-1}\gamma +O(x+y+z),\nonumber
\end{eqnarray}
for $y<0$ and $c-\gamma \ge 0$. By possibly changing the direction of time, we can take $c+\gamma\ge 0$. The two-fold singularity is placed at the origin and it is the transverse intersection of two lines of fold singularities 
\begin{align}
l^+:&\,y=0=z,\eqlab{foldLines1}\\
l^-:&\,x=0=y,\nonumber
\end{align}
corresponding to tangency from above and below, respectively. Furthermore:
\begin{itemize}
 \item $l^+$ is visible (invisible) for $b>0$ ($b<0$);
 \item $l^-$ is visible (invisible) for $\beta >0$ ($\beta<0$);
\end{itemize}
and hence the two-fold is 
\begin{itemize}
 \item visible if $b>0$ and $\beta>0$;
 \item visible-invisible if $b>0$ and $\beta<0$ (or $b<0$ and $\beta>0$);
 \item invisible if $b<0$ and $\beta<0$. 
\end{itemize}
Moreover, 
\begin{eqnarray}
 \Sigma_{sl}:\,y=0,xz>0,\eqlab{SigmaSliding}\\
 \Sigma_{cr}:\,y=0,xz<0.\nonumber
\end{eqnarray}
$\Sigma_{sl} = \Sigma_{sl}^-\cup \Sigma_{sl}^+$ where $\Sigma_{sl}^-:\,y=0,x<0,z<0$ and $\Sigma_{sl}^+:\,y=0,x>0,z>0$ are stable and unstable sliding regions, respectively. Similarly, $\Sigma_{cr} = \Sigma_{cr}^-\cup \Sigma_{cr}^+$ where $\Sigma_{cr}^-:\,y=0,x>0,z<0$ and $\Sigma_{cr}^+:\,y=0,x<0,z>0$ are regions with crossing downwards and upwards respectively.

\end{proposition}

\begin{proof}
To take $c+\gamma\ge 0$, we simply multiply the vector-field by $\text{sign}(c+\gamma)$, potentially, this changes the direction of time. To ensure $c-\gamma \ge 0$ we transform $(x,y,z)\mapsto (z,-y,x)$, reverse $b$ and $\beta$ and change the signs of $a$ and $\alpha$. 

 The visibility of $l^+$ is determined by the sign of 
 \begin{eqnarray*}
  X^+(X^+f)(x,0,0) =\vert b\vert \text{sign}(b)+\mathcal O(x)= b+\mathcal O(x).
 \end{eqnarray*}
The statement about the visibility of $l^+$ therefore follows by taking $x$ sufficiently small. Similarly, the statement about the visibility of $l^-$ follows from
 \begin{eqnarray*}
  X^-(X^-f)(0,0,z) =-\vert \beta \vert \text{sign}(\beta)+\mathcal O(z)=-\beta+\mathcal O(z),
 \end{eqnarray*}
  taking $z$ sufficiently small.

The region $\Sigma_{sl}$ \eqref{SigmaSliding} is a sliding region because
\begin{eqnarray*}
 X_2^+(x,0,z)X_2^-(x,0,z) = (\vert b\vert +\mathcal O(x+z))(-\vert \beta \vert +\mathcal O(x+z))xz <0.
\end{eqnarray*}
$\Sigma_{sl}^-$ is stable sliding because $X_2^+(x,0,z) = (\vert b\vert +\mathcal O(x+z))z<0$ there. Similarly, $\Sigma_{cr}^-$ is a region with crossing downwards.
\end{proof}

We illustrate the different two-fold singularities and the division of $\Sigma$ into sliding and crossing regions in \figref{sliding}. We emphasize again that our normal form is such that these regions remain fixed for all parameter values.
\begin{figure}[h!] 
\begin{center}
\subfigure[Visible two-fold: $b>0$ and $\beta>0$]{\includegraphics[width=.49\textwidth]{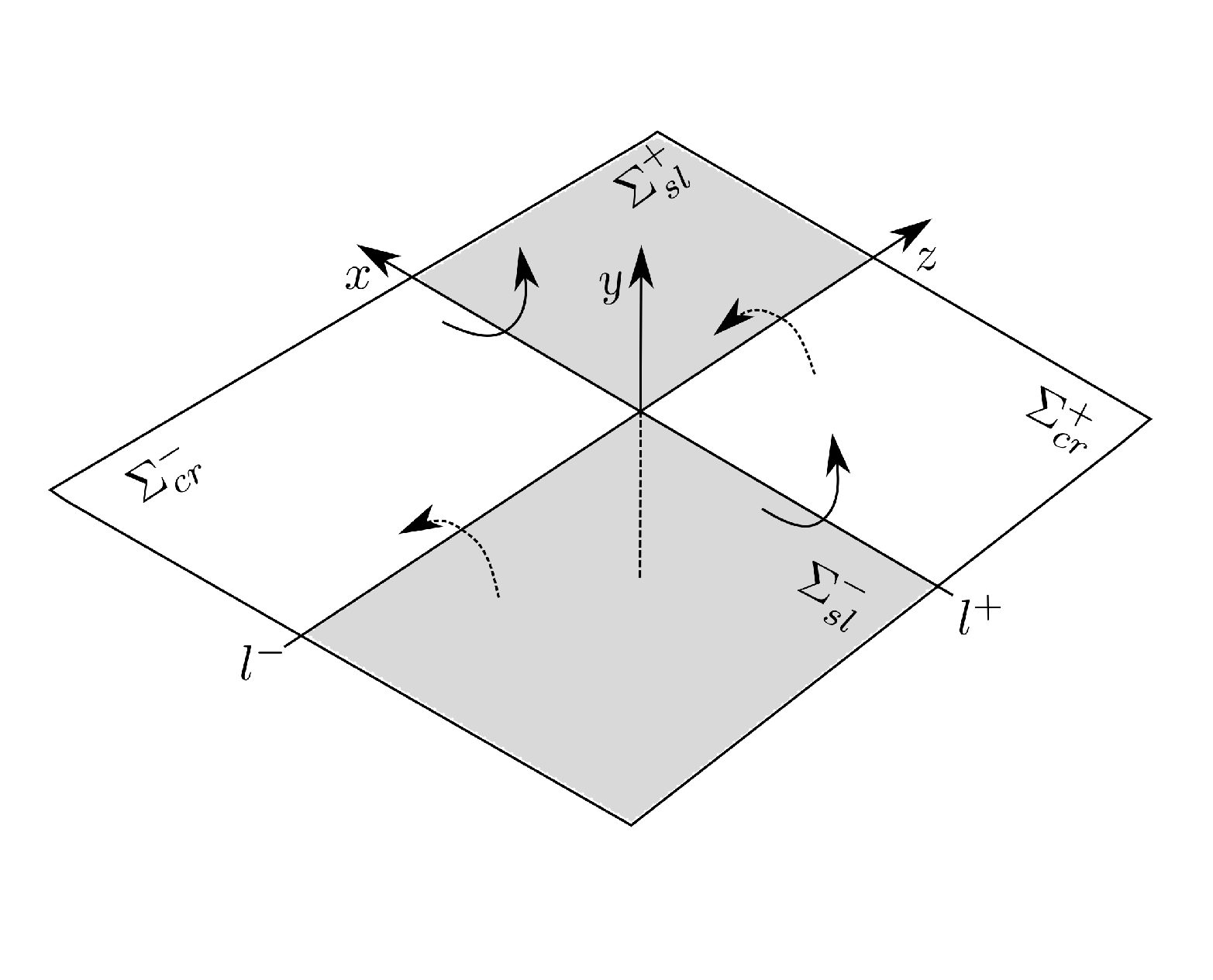}}
\subfigure[Invisible two-fold: $b<0$ and $\beta<0$]{\includegraphics[width=.49\textwidth]{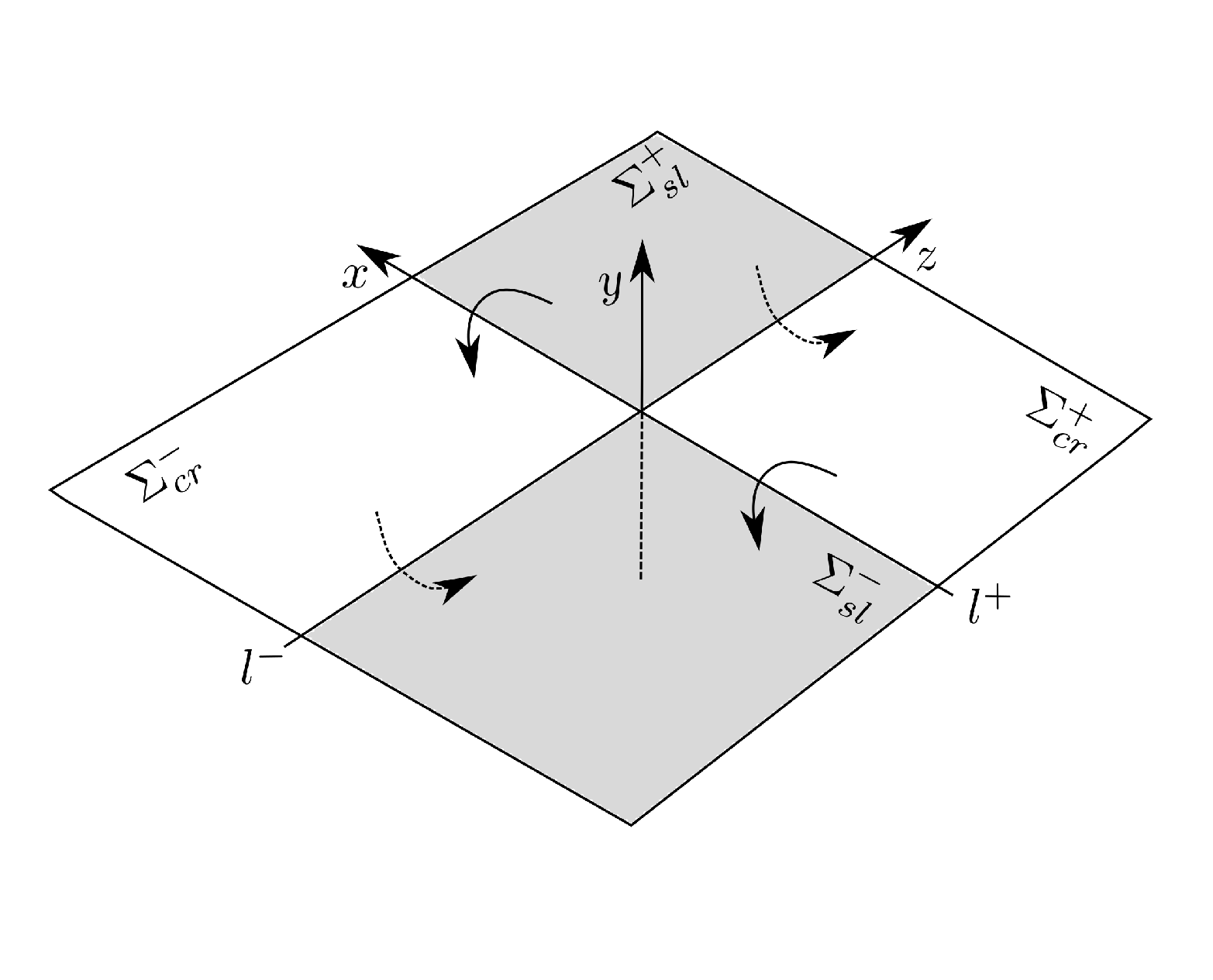}}
\subfigure[Visible-invisible two-fold: $b>0$ and $\beta<0$]{\includegraphics[width=.49\textwidth]{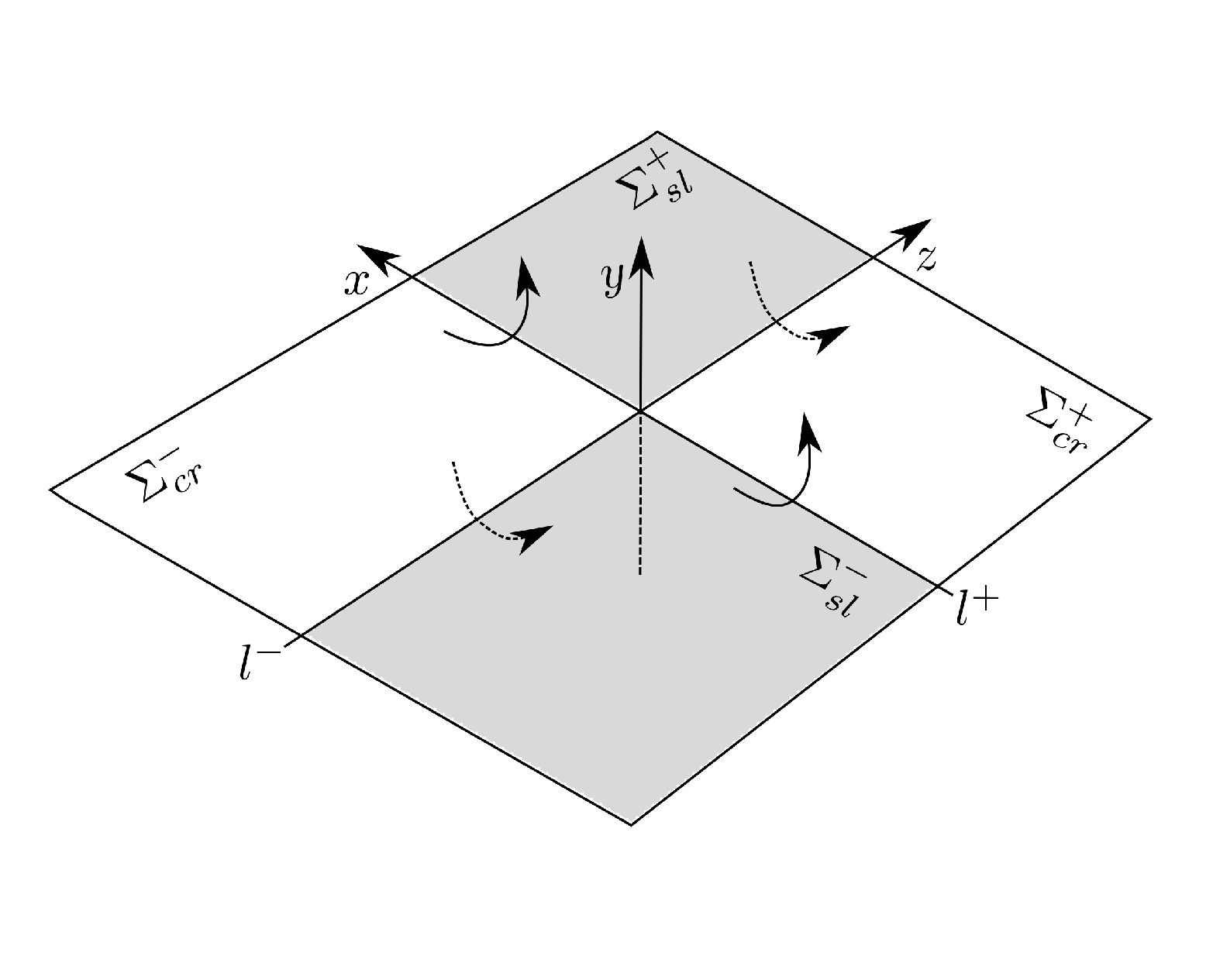}}
\end{center}
 \caption{Illustration of the different types of two-fold singularities. The dotted trajectories indicate tangencies of trajectories of $X^-$ with the fold line $l^-$. Similarly, the full trajectories illustrate the tangencies of $X^+$ with the fold line $l^+$. The case $b<0$ and $\beta>0$ is symmetrically related to the case shown in (c). }
\figlab{sliding}
\end{figure}

 \section{The sliding vector-field and singular canards}\seclab{slidingVectorField}
The sliding vector-field within $\Sigma_{sl}$ is given by
\begin{eqnarray}
 \dot x&=&\sigma X_1^+(x,0,z)+(1-\sigma )X_1^-(x,0,z),\eqlab{slidingEqns}\\
 \dot y&=&0,\nonumber\\
 \dot z&=&\sigma X_3^+(x,0,z)+(1-\sigma )X_3^-(x,0,z),\nonumber
\end{eqnarray}
with
\begin{eqnarray}
 \sigma &=& \frac{X^-f(x,0,z)}{X^-f(x,0,z)-X^+f(x,0,z)}=\frac{-\vert \beta \vert x + \mathcal O(z(x+z))}{-\vert \beta \vert x-\vert b\vert z+\mathcal O((x+z)^2)}.\eqlab{lambda}
\end{eqnarray}
The denominator
\begin{eqnarray*}
X^-f(x,0,z)-X^+f(x,0,z)&=&-\vert \beta \vert x-\vert b\vert z+\mathcal O((x+z)^2)
\end{eqnarray*}
is positive within $\Sigma_{sl}^-$, provided $x$ and $z$ are sufficiently small. It is negative within $\Sigma_{sl}^+$ and only vanishes within $\overline{\Sigma}_{sl}$ on the two-fold singularity $p$. We can therefore multiply the sliding vector-field by $\vert X^-f(x,0,z)-X^+f(x,0,z)\vert$, corresponding to a time re-parametrisation, within $\Sigma_{sl}$. The sliding vector-field within $\Sigma_{sl}^{\mp}$ is then given by
 \begin{eqnarray}
  \dot x&=&\mp c x\mp \vert b\vert \text{sign}(\beta) z+\mathcal O((x+z)^2)\eqlab{slidingVectorField}\\
  \dot y&=&0,\nonumber\\
  \dot z&=&\mp \vert \beta\vert \text{sign}(b) x\mp \gamma z+\mathcal O((x+z)^2).\nonumber
 \end{eqnarray}
 Similar equations were derived in \cite[p. 279]{filippov1988differential}.
%
 As we will see, trajectories can go from the stable sliding region $\Sigma_{sl}^-$ to the unstable sliding region $\Sigma_{sl}^+$, or vice versa. Hence these trajectories resemble canards in slow-fast systems \cite{Benoit81}. However to emphasize that we are dealing with an underlying piecewise smooth system, which is not slow-fast, we introduce the concept of \textit{singular canards} in the following definition:
 \begin{definition}
 A singular canard is a trajectory of the sliding equations \eqref{slidingVectorField} which is contained within $\Sigma_{sl}$ having a smooth continuation through the two-fold singularity $p$.
 Following \cite{szmolyan_canards_2001}, we say that the singular canard is a \textit{primary} singular canard if it goes from the \textit{stable sliding region} $\Sigma_{sl}^-$ to the unstable sliding region $\Sigma_{sl}^+$ in forward time. If it goes from the \textit{unstable sliding region} $\Sigma_{sl}^+$ to the stable sliding region $\Sigma_{sl}^-$ then we say it is a \textit{faux} singular canard.
\end{definition}

We will elaborate on the relation with the canard literature in section \secref{comparison}.
To describe singular canards, we take the equations on $\Sigma_{sl}^-$ \eqref{slidingVectorField}:
\begin{eqnarray}
  \dot x&=&- c x- \vert b\vert \text{sign}(\beta) z+\mathcal O((x+z)^2)\eqlab{slidingVectorFieldNewTime}\\
  \dot z&=&-\vert \beta \vert \text{sign}(b) x- \gamma z+\mathcal O((x+z)^2),\nonumber
 \end{eqnarray}
and note that the orbits of \eqref{slidingVectorFieldNewTime} agree with those of \eqref{slidingVectorField} - one just has to reverse the direction of time within $\Sigma_{sl}^+$ for them to agree as trajectories. Then the two-fold singularity $p$ appears as an equilibrium of the equations \eqref{slidingVectorFieldNewTime}.
By identifying trajectories  of \eqref{slidingVectorFieldNewTime} with those of \eqref{slidingVectorField}, we obtain the following lemma:
 \begin{lemma}
 A primary/faux singular canard agrees as a set with a trajectory of \eqref{slidingVectorFieldNewTime} which is forward/backwards asymptotic to the two-fold singularity $p$. 
\end{lemma}

To study singular canards we therefore consider \eqref{slidingVectorFieldNewTime} and its linearization about $(x,z)=(0,0)$.
 \begin{proposition}\proplab{EqSlidingVectorField}
 The two-fold singularity $p$ is an equilibrium of \eqref{slidingVectorFieldNewTime} with associated eigenvalues:
 \begin{align}
  \lambda_{\pm} = -\frac12 (c+\gamma)\pm \frac12 \sqrt{(c-\gamma)^2 + 4b\beta}.\eqlab{lambdapm}
 \end{align}
The eigenvectors corresponding to $\lambda_{\pm}$ are spanned by
\begin{align}
v_{\pm} = \begin{pmatrix}1 \\
                                                                 -\chi_{\pm}
                                                               \end{pmatrix} \eqlab{vpm}
\end{align}
 where
                                                               \begin{align}
\chi_{\pm} = \frac{\text{sign}(\beta)}{2\vert b\vert}\left({c-\gamma}\pm \sqrt{(c-\gamma)^2+4b\beta}\right).\eqlab{chipm}
                                                               \end{align}
A singular canard corresponds to a trajectory tangent at the origin to an eigenvector with $\chi_{\pm}<0$. It is a primary/faux singular canard if the associated eigenvalue $\lambda_{\pm}$ is negative/positive. 
For  $(c-\gamma)^2+4b\beta>0$ we have the following three cases (S), (SN), and (N):  
 \begin{itemize}
  \item[(S)] For $0\le c+\gamma < \sqrt{(c-\gamma)^2+4b\beta}$ the equilibrium is a saddle with $\lambda_{-}<0<\lambda_+$.
  \item[(SN)] For $c+\gamma = \sqrt{(c-\gamma)^2+4b\beta}$ the equilibrium is a saddle-node with $\lambda_-<\lambda_+=0$.
   \item[(N)] For $c+\gamma > \sqrt{(c-\gamma)^2+4b\beta}$ the equilibrium is a stable node with $\lambda_{-}<\lambda_+ <0$. 
 \end{itemize}
 Given that $c-\gamma\ge 0$, the signs of $\chi_{\pm}$ are described by the following:
  \begin{itemize}
\item (visible) If $b>0$ and $\beta>0$ then $\chi_-<0<\chi_+$.
\item (visible-invisible) If $b>0$ and $\beta<0$ then $\chi_+<\chi_{-}<0$. If $b<0$ or $\beta>0$ then $0<\chi_-<\chi_{+}$. 
  \item (invisible) If $b<0$ and $\beta<0$ then $\chi_+<0<\chi_-$.
    \end{itemize}
 \end{proposition}
 

\begin{proof}
 The coefficient matrix
 \begin{eqnarray}
 A= \begin{pmatrix}
   -c & -\vert b\vert \text{sign}(\beta)\\
   -\vert \beta \vert \text{sign}(b) &-\gamma
  \end{pmatrix}, \eqlab{AMatrix}
 \end{eqnarray}
has eigenvalues $\lambda_{\pm}$ \eqref{lambdapm}.
The curve $\sqrt{(c-\gamma)^2+4b\beta}= c+\gamma\ge 0$ separates parameter regions where the origin is a stable node from regions where it is a saddle (see \figref{ParameterEqType1}). 

Inserting \eqref{vpm} into the equation for the eigenvectors gives the following equation for $\chi_{\pm}$:
\begin{align}
-c+\vert b\vert \text{sign}(\beta)\chi_{\pm} = \lambda_{\pm},\eqlab{eqnchipm}
\end{align}
and therefore
\begin{align*}
\chi_{\pm} = \text{sign}(\beta)\vert b\vert^{-1} (c-\lambda_{\pm}).
\end{align*}
from which \eqref{chipm} follows. 
To verify the signs of $\chi_{\pm}$ in the proposition we use \eqref{chipm} to conclude that
\begin{align*}
 \chi_{\mp}\text{sign}(\beta) \gtrless 0,
\end{align*}
in the visible ($\beta>0$) and invisible ($\beta<0$) cases. In the visible-invisible case we have that
\begin{align*}
 \chi_{\pm}\text{sign}(\beta)> 0.
\end{align*}
since $c-\gamma>0$. 


\end{proof}
\begin{remark}\remlab{excludingCases}
 We exclude the case $(c-\gamma)^2 +4b\beta< 0$ for the visible-invisible two-fold, since this does not give rise to singular canards. 
{The case $(c-\gamma)^2+4b\beta=0$ is degenerate, and we shall not consider it further in this paper. }

{We shall also not consider the degenerate (SN) case further. It is similar to the folded saddle-node bifurcation in smooth slow-fast dynamical systems \cite{krupa_local_2010}. It seems likely that the techniques used there could be employed to study the regularization of this degeneracy.}
\end{remark}

\begin{figure}[h!] 
\begin{center}
{\includegraphics[width=.9\textwidth]{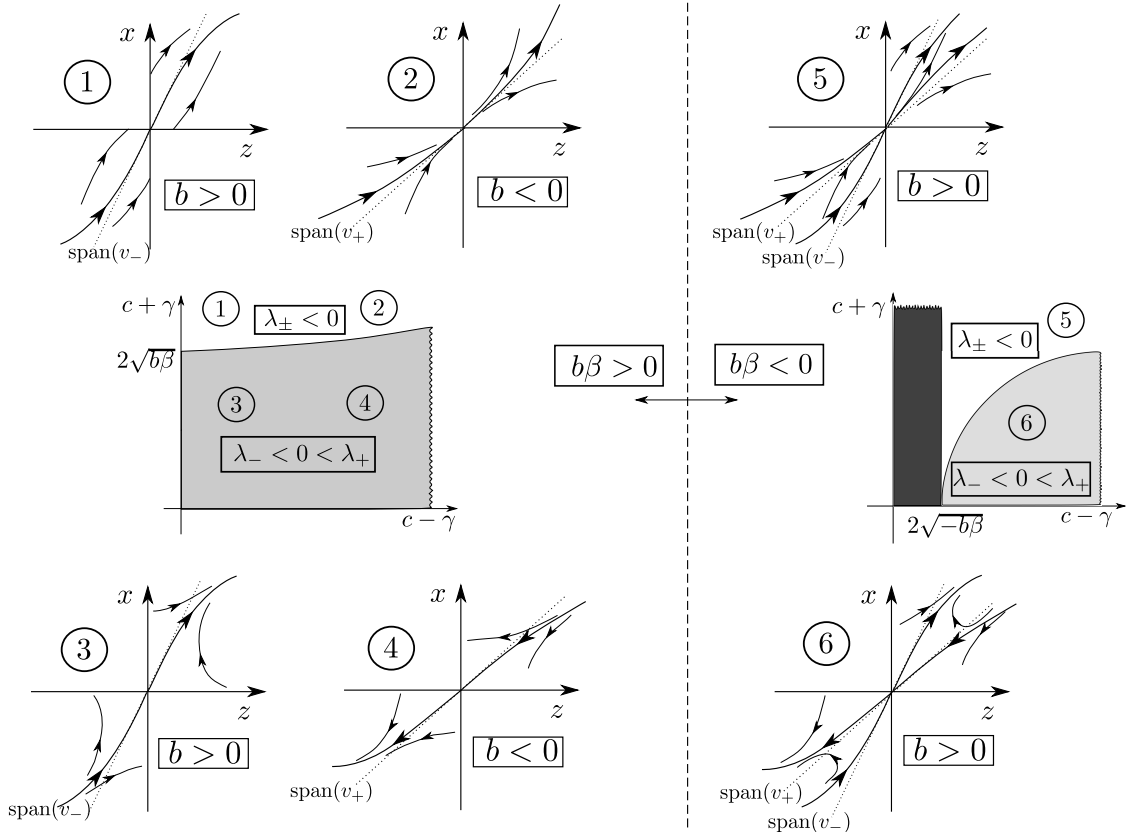}}
\end{center}
 \caption{The different types of two-fold singularity: $b\beta>0$ (to the left of the vertical dotted line) and $b\beta <0$ (to the right). In the embedded parameter diagrams, the two-fold singularity is a saddle in the grey regions, a stable node in the white regions and a focus in the black region (where does not give rise to canards). The phase portraits are described in \propref{singularCanards}. Similar diagrams are presented in \cite[Figs. 111 - 117]{filippov1988differential} and also in \cite{teixeira_bifurcations_1993}, but without the connection to canards.}
\figlab{ParameterEqType1}
\end{figure}
\begin{remark}
The constants $\chi_{\pm}$ also satisfy the following equation 
\begin{align}
\vert b\vert \text{sign}(\beta)\chi_{\pm}^2-(c-\gamma) \chi_{\pm}-\vert \beta \vert \text{sign}(b)=0,\eqlab{chieqn}
\end{align}
which shall appear later on.
\end{remark}

For case (N), it is clear from \propref{EqSlidingVectorField} that singular canards can {be either unique or non-unique}. 
In the classical canard literature \cite{brons-krupa-wechselberger2006:mixed-mode-oscil,szmolyan_canards_2001}, there is a notion of strong and weak canards,  which will also be useful here. 
\begin{definition}
 In case (N), if the singular canard (\propref{EqSlidingVectorField}) corresponds to a \textit{strong} (\textit{weak}) eigendirection
 then we call it a \textit{strong} (\textit{weak}) singular canard.  
 \end{definition}
 
 The following elementary lemma (compare with \figref{ParameterEqType1}) then describes the uniqueness of the singular canards:
\begin{lemma}\lemmalab{uniqueness}
\begin{itemize}
\item Case (S): All singular canards are unique. 
 \item Case (N):  Strong singular canards are unique whereas weak singular canards are non-unique. 
 \end{itemize}
\end{lemma}

We then conclude:
\begin{proposition}\proplab{singularCanards}
Assume that $b\beta\ne 0$ and $(c+\gamma)^2+4b\beta >0$. We then have the following six different hyperbolic cases:
\begin{itemize}
 \item Case (N) ($\lambda_-<\lambda_+<0$): 
 \begin{itemize}
 \item In the visible case, there exists one strong singular canard. It is a primary singular canard and it coincides as set with the unique trajectory that is tangent to $v_-$ ($\chi_-<0$) at the origin of \eqref{slidingVectorFieldNewTime}. 
 (Region \textnormal{\circled{1}} in \figref{ParameterEqType1}.)
 \item In the invisible case, $\Sigma_{sl}$ is filled with primary singular canards. They are all tangent to $v_+$ ($\chi_+<0$) at the origin of \eqref{slidingVectorFieldNewTime}.  (Region \textnormal{\circled{2}} in \figref{ParameterEqType1}.)
 \item In the visible-invisible case, there exists a sector of primary singular canards if $b>0$. The sector is restricted by the $x$-axis and the strong singular canard that is tangent to $v_{-}$ at the origin. All other trajectories are weak singular canards that are primary and tangent to $v_+$ at the origin.  (Region \textnormal{\circled{5}} in \figref{ParameterEqType1}.) If $b<0$ then there are no singular canards.
 \end{itemize}
 \item Case (S) ($\lambda_-<0<\lambda_+$): 
 \begin{itemize}
 \item In the visible case, there exists one singular canard. It is a primary singular canard and it coincides as a set with the stable manifold ($\lambda_-<0$) that is tangent to $v_-$ ($\chi_-<0$) at the origin of \eqref{slidingVectorFieldNewTime}. (Region \textnormal{\circled{3}} in \figref{ParameterEqType1}.)
 \item In the invisible case, there exists one singular canard. It is a faux singular canard and it coincides as a set with the unstable manifold ($\lambda_+>0$) that is tangent to $v_+$ ($\chi_+<0$) at the origin of \eqref{slidingVectorFieldNewTime}. (Region \textnormal{\circled{4}} in \figref{ParameterEqType1}.) 
 \item In the visible-invisible case, there exist one primary singular canard and one faux singular canard if $b>0$. They correspond as sets to the stable and unstable manifolds and are tangent to $v_-$ and $v_+$ respectively at the origin of \eqref{slidingVectorFieldNewTime}. (Region \textnormal{\circled{6}} in \figref{ParameterEqType1}.) If $b<0$ then there are no singular canards.
 \end{itemize}
\end{itemize}
\end{proposition}

\begin{proof}
For there to exist a singular canard associated with the eigendirections described in \propref{EqSlidingVectorField} the span of $v_{\pm}$ must be contained within $\Sigma_{sl}\cup \{p\}$.  For case (S), the statements in the proposition then follow from \propref{EqSlidingVectorField} and the stable manifold theorem. For case (N), in the visible case, only the span of $v_{-}$ is contained within $\Sigma_{sl}\cup \{p\}$. The vector $v_{-}$ corresponds to the strong eigendirection and there is therefore a unique trajectory tangent to $v_-$ at the origin. For case (N), in the invisible case, it is only the span of $v_{+}$ which is contained within $\Sigma_{sl}\cup \{p\}$ and the vector $v_{+}$ corresponds to the weak eigendirection. Therefore all orbits within $\Sigma_{sl}$ are tangent to $v_+$ at the origin and they are all primary singular canards. 
%

For both cases (N) and (S) in the visible-invisible case with $b\beta<0$, we have that both $v_{\pm}$ are contained within $\Sigma_{sl}\cup \{p\}$ if $b>0$. If $b<0$ then neither of the vectors are contained within $\Sigma_{sl}\cup \{p\}$. Repeating the arguments for the visible and invisible cases above completes the proof of the statements about the visible-invisible case in both cases (N) and (S).
\end{proof}

\section{Regularization}\seclab{regularize}

It is natural to ask how the canards seen in the two-fold singularity $p=(0,0,0)$ can survive regularization. There is a number of ways that the original piecewise smooth system vector field $X=(X^-,X^+)$ can be regularized. We follow the approach of Sotomayor and Teixeira \cite{Sotomayor96}. We define a $C^k$-function ($1\le k\le \infty$) $\phi(y)$ which satisfies:
\begin{eqnarray}
 \phi(y)=\left\{\begin{array}{cc}
                 1 & \text{for}\quad y\ge 1,\\
                 \in (-1,1)& \text{for}\quad y\in (-1,1),\\
                 -1 & \text{for}\quad y\le -1,\\
                \end{array}\right.\eqlab{phiFunc}
 \end{eqnarray}
 where 
 \begin{align}
 \phi'(y)>0 \quad \text{within}\quad y\in (-1,1).\eqlab{phiProperties}
 \end{align}
 The regularized vector-field $X_\epsilon ({\bf x})$ is then given by
\begin{eqnarray}
 X_\epsilon ({\bf x}) =
 \frac12 X^+({\bf x})(1+\phi(\epsilon^{-1} y)) +\frac12 X^-({\bf x})(1-\phi(\epsilon^{-1}y)).\eqlab{Xeps}
\end{eqnarray}
Note that $X_\epsilon({\bf x})=X^{\pm}({\bf x})$ for $y\gtrless \pm \epsilon$. The region $y\in (-\epsilon,\epsilon)$ is the \textit{region of regularization}. By re-scaling 
\begin{align}
\tilde y=\epsilon^{-1} y\eqlab{yTilde}
\end{align}
the region of regularization becomes $\tilde y \in (-1,1)$. We now drop the tilde and re-scale time according to $\tau = 2 \epsilon^{-1}t$ to obtain the following set of equations {from \eqref{Xeps}}:
\begin{eqnarray}
x' &=&\epsilon \left(X_1^+(x,0,z)(1+\phi(y))+X_1^-(x,0,z)(1-\phi(y))+O(\epsilon)\right)\eqlab{slowFastEquations1}\\
&=&\epsilon \left((\vert \beta\vert^{-1} c+\mathcal O(x+z))(1+\phi(y))+(\text{sign}(\beta)+\mathcal O(x+z))(1-\phi(y))+O(\epsilon)\right),\nonumber\\
y'&=&X_2^+(x,0,z) (1+\phi(y)) +X_2^-(x,0,z)(1-\phi(y)) +O(\epsilon)\nonumber\\
&=&(\vert b\vert+\mathcal O(x+z)) {z} (1+\phi(y)) +(-\vert \beta\vert +\mathcal O(x+z))x(1-\phi(y)) +O(\epsilon),\nonumber\\
 z'&=&\epsilon \left(X_3^+(x,0,z) (1+\phi(y))+X_3^-(x,0,z)(1-\phi(y))+O(\epsilon)\right)\nonumber\\
 &=&\epsilon \left((\text{sign}(b)+\mathcal O(x+z)) (1+\phi(y))+(\vert b\vert^{-1} \gamma+\mathcal O(x+z))(1-\phi(y))+O(\epsilon)\right),\nonumber
\end{eqnarray}
with $()'=\frac{d}{d\tau}$. This is a slow-fast system. {The $y$ variable is fast with $\mathcal O(1)$ velocities whereas $x$ and $z$ are slow variables with $\mathcal O(\epsilon)$ velocities. Time $\tau$ is the \textit{fast time} and time $t$ is the \textit{slow time}. In this paper, we apply and extend Fenichel's theory of singular perturbations to study these regularized equations \eqref{slowFastEquations1}. Fenichel's theory allows us to go from a description of the singular limit $\epsilon=0$ to a description for $\epsilon>0$. This approach has the advantage that it is geometric. So, for example, we are able to solve persistence problems by invoking transversality. In this paper, it will also be important to identify the singular limit $\epsilon=0$ with the original piecewise smooth system. }

The key to the subsequent analysis in this paper is the following result ({a similar result is given in \cite[Theorem 1.1]{llibre_sliding_2008}, in a slightly different form}):
\begin{theorem}\thmlab{criticalManifold}
 There exists a critical manifold 
 \begin{align}
 S_0:\quad &y=h_0(x,z),\,\quad \text{for}\quad (x,z)\ne 0,\eqlab{h0Function}\\
 &\mbox{{$y\in [-1,1],\quad \,\,\,\,\, \text{for}\quad (x,z)=0,\nonumber$}}
 \end{align}
 of \eqref{slowFastEquations1} for $\epsilon=0$. On the critical manifold the motion of the slow variables $x$ and $z$ is described by reduced equations which coincide with the sliding equations \eqref{slidingEqns}. 
\end{theorem}
\begin{proof}
The critical manifold $S_0$ is the set of equilibria of \eqref{slowFastEquations1}$\vert_{\epsilon=0}$. {For $x=0,\,z=0$ we obtain $y\in [-1,1]$. Otherwise the set of equilibria can be described by the following equation}
\begin{align*}
 \frac{1-\phi(y)}{1+\phi(y)} = -\frac{X_2^+(x,0,z)}{X_2^-(x,0,z)},
\end{align*}
and so
\begin{eqnarray}
 \phi(y) =-\frac{X_2^+(x,0,z)+X_2^-(x,0,z)}{X_2^+(x,0,z)-X_2^-(x,0,z)}=-\frac{(\vert b\vert+\mathcal O(x+z)) {z} +(-\vert \beta\vert +\mathcal O(x+z))x}{(\vert b\vert +\mathcal O(x+z))z-(-\vert \beta \vert +\mathcal O(x+z))x}. \eqlab{slowManifold}
\end{eqnarray}
Hence
\begin{align}
 1+\phi(y) &= 2\sigma,\eqlab{phiyAndLambda}\\
 1-\phi(y) &=2(1-\sigma),\nonumber
\end{align}
where $\sigma$ is given by \eqref{lambda}. {When $(x,0,z)\in \Sigma_{sl}$ then $\phi(y)\in (-1,1)$ according to \eqref{slowManifold}. Now since $\phi'(y)\ne 0$ for $\phi(y)\in (-1,1)$ the equation \eqref{slowManifold} can be solved for $y$ giving rise to the function $h_0=h_0(x,z)$ in \eqref{h0Function}. The function $h_0$ has a smooth extension onto $(x,0,z)\in \overline{\Sigma}_{sl}\backslash \{p\}$.}

For the second part of the theorem, we insert \eqref{phiyAndLambda} into \eqref{slowFastEquations1}, undo the rescaling of time and return to the original slow time in \eqref{Xeps}, and then set $\epsilon=0$. Then we obtain the reduced problem for $x$ and $z$:
\begin{align}
 \dot x &=\sigma X_1^+(x,0,z)+(1-\sigma) X_1^-(x,0,z),\eqlab{layerEqns}\\
 \dot z&=\sigma X_3^+(x,0,z)+(1-\sigma) X_3^-(x,0,z).\nonumber
\end{align}
These equations coincide with the sliding equations in \eqref{slidingEqns}.
\end{proof}
\begin{remark}\remlab{S0vsSigmaSl}
 If we return to our original $y$ variable using \eqref{yTilde} then the critical manifold $S_0$ becomes a graph $y=\epsilon h(x,z)$ within $(x,z)\ne 0$. For $\epsilon=0$ the whole of $S_0$ therefore collapses to $\overline{\Sigma}_{sl}:\,y=0$.
\end{remark}
\subsection{Fenichel's theory and the singular canards revisited}
Since the reduced equations \eqref{layerEqns} coincide with the sliding equations \eqref{slidingEqns}, the analysis in section \secref{slidingVectorField} can be carried over to the singular limit of the regularized system \eqref{Xeps}. In particular, we obtain singular canards on the critical manifold $S_0$ (which collapses to $\Sigma_{sl}$ according to \remref{S0vsSigmaSl}). The main focus of the paper is to investigate the fate of the singular canards for small $\epsilon$.

 {Points of $S_0$ are equilibria of the system \eqref{slowFastEquations1} for $\epsilon=0$. The property of normally hyperbolicity, on which Fenichel's geometric singular perturbation theory rests, relates to properties of these equilibria. }

\begin{definition}\defnlab{hyperbolicityDefinition}
{The critical manifold $S_0$ is \textnormal{normally hyperbolic} at $\textbf{x}\in S_0$ if the linearization about the equilibrium $\textbf{x}$ of \eqref{slowFastEquations1}$\vert_{\epsilon=0}$ has as many eigenvalues with zero real part as there are slow variables.  }
\end{definition}

{Fenichel's theory \cite{fen1,fen2,fen3,jones_1995} shows that the compact normally hyperbolic parts of $S_0$ perturb to invariant slow manifolds, diffeomorphic and $\mathcal O(\epsilon)$-close to the critical manifolds, for $\epsilon$ sufficiently small. 
By linearizing about an equilibrium $\textbf{x}=(x,y,z)\in S_0$ for our limiting system \eqref{slowFastEquations1}$\vert_{\epsilon=0}$ we obtain the following condition for normal hyperbolicity
\begin{eqnarray}
 \phi'(y)(\vert b\vert z+\vert \beta \vert x+\mathcal O((x+z)^2))\ne 0.\eqlab{hyperbolicity}
\end{eqnarray}
Where this condition is violated, the linearization of an equilibrium of \eqref{slowFastEquations1}$\vert_{\epsilon=0}$ will have three zero eigenvalues which therefore exceeds the number of slow variables. 
}

Setting $x=0$ in \eqref{slowManifold} gives $S_0\cap \{x=0,z\ne 0\}$ as
\begin{eqnarray}
 \phi(y) &=& -1,\eqlab{phiyEqNeg1}
\end{eqnarray}
for $z\ne 0$.
On the other hand, setting $z=0$, gives $S_0\cap \{x\ne 0,z=0\}$ as
\begin{eqnarray}
 \phi(y) &=& 1,\eqlab{phiyEqPos1}
\end{eqnarray}
for $x\ne 0$. 
Since $\phi$ is at least $C^1$ we conclude that $\phi'(y)=0$ on the lines 
\begin{align}
\tilde l^-=\{(x,y,z)\vert x=0,\,y=-1,\,z\ne 0\},\quad \tilde l^+ =\{(x,y,z)\vert x\ne 0,y=1,\,z=0\}.\eqlab{lines}
\end{align}
Using \eqref{hyperbolicity} we therefore conclude the following:
\begin{proposition}\proplab{nonHyperbolicLines}
 The fold lines $l^-$ and $l^+$ in the original piecewise smooth system give rise to lines $\tilde l^-$ and $\tilde l^+$ \eqref{lines}, respectively, of non-hyperbolic points on the critical manifold $S_0$. The two-fold singularity $p=(0,0,0)$ in the original piecewise smooth system gives rise to a non-hyperbolic line 
 \begin{align}
 \tilde p:\,y\in [-1,1],\,x=0,\,z=0,\,\epsilon=0,\eqlab{tildep}
 \end{align}
 in the extended phase space $(x,y,z,\epsilon)$.
\end{proposition}

Therefore Fenichel's theory cannot be used to explain how we leave the sliding region, in particular the fate of singular canards for $\epsilon\ne 0$. 

Fenichel's theory applies away from the fold lines $\tilde l^-$ and $\tilde l^+$ of \eqref{lines} and the line $\tilde p$ \eqref{tildep}. In particular we know that the perturbed invariant manifolds inherit the stability of $S_0$, which is determined by the sign of \eqref{hyperbolicity}. Our critical manifold $S_0$ divides into an attracting part: $S_{a}:\,x<0,z<0$, a repelling part $S_r:\,x>0,z>0$, the two fold lines \eqref{lines} and the line $\tilde p$ \eqref{tildep} so that: 
$$\mbox{{$S_0=S_a\cup S_r\cup \tilde l^- \cup \tilde l^+\cup \tilde p$}}.$$ 
In the following proposition, we collect the results of the application of Fenichel's theory \cite{fen1,fen2,fen3,jones_1995}:
\begin{proposition}\proplab{fenichelProposition}
 Let $\mathcal U^-\subset \{(x,z)\}$ and $\mathcal U^+\subset \{(x,z)\}$ be compact regions completely contained within the fourth ($x<0,z<0$) and first quadrants ($x>0,z>0$), respectively. Then the critical manifolds 
 \begin{align*}
 S_a\vert_{\mathcal U^-}:&\quad y=h_0(x,z),\,(x,z)\in \mathcal U^-,\\
 S_r\vert_{\mathcal U^+}:&\quad y=h_0(x,z),\,(x,z)\in \mathcal U^+,
 \end{align*}
 perturb to invariant slow manifolds 
 \begin{align*}
 S_{a,\epsilon}:&\quad y=h_\epsilon(x,z,\epsilon)=h_0(x,z)+\mathcal O(\epsilon),\,(x,z)\in \mathcal U^-,\\
 S_{r,\epsilon}:&\quad y=h_\epsilon(x,z,\epsilon)=h_0(x,z)+\mathcal O(\epsilon),\,(x,z)\in \mathcal U^+
 \end{align*}
 for $\epsilon\le \epsilon_0(\mathcal U^-,\mathcal U^+)$ sufficiently small. In general, $S_{a,\epsilon}$ and $S_{r,\epsilon}$ are non-unique and they are all $\mathcal O(e^{-c/\epsilon})$-close for some $c>0$ independent of $\epsilon$. The flow on $S_{a,\epsilon}$ and $S_{r,\epsilon}$ is $\epsilon$-close to the flow of the reduced problem \eqref{layerEqns}.
\end{proposition}

%
%
We illustrate the application of Fenichel's theory in \figref{SaSr2}. 




\begin{figure}[h!] 
\begin{center}
\subfigure[$\epsilon=0$]{\includegraphics[width=.495\textwidth]{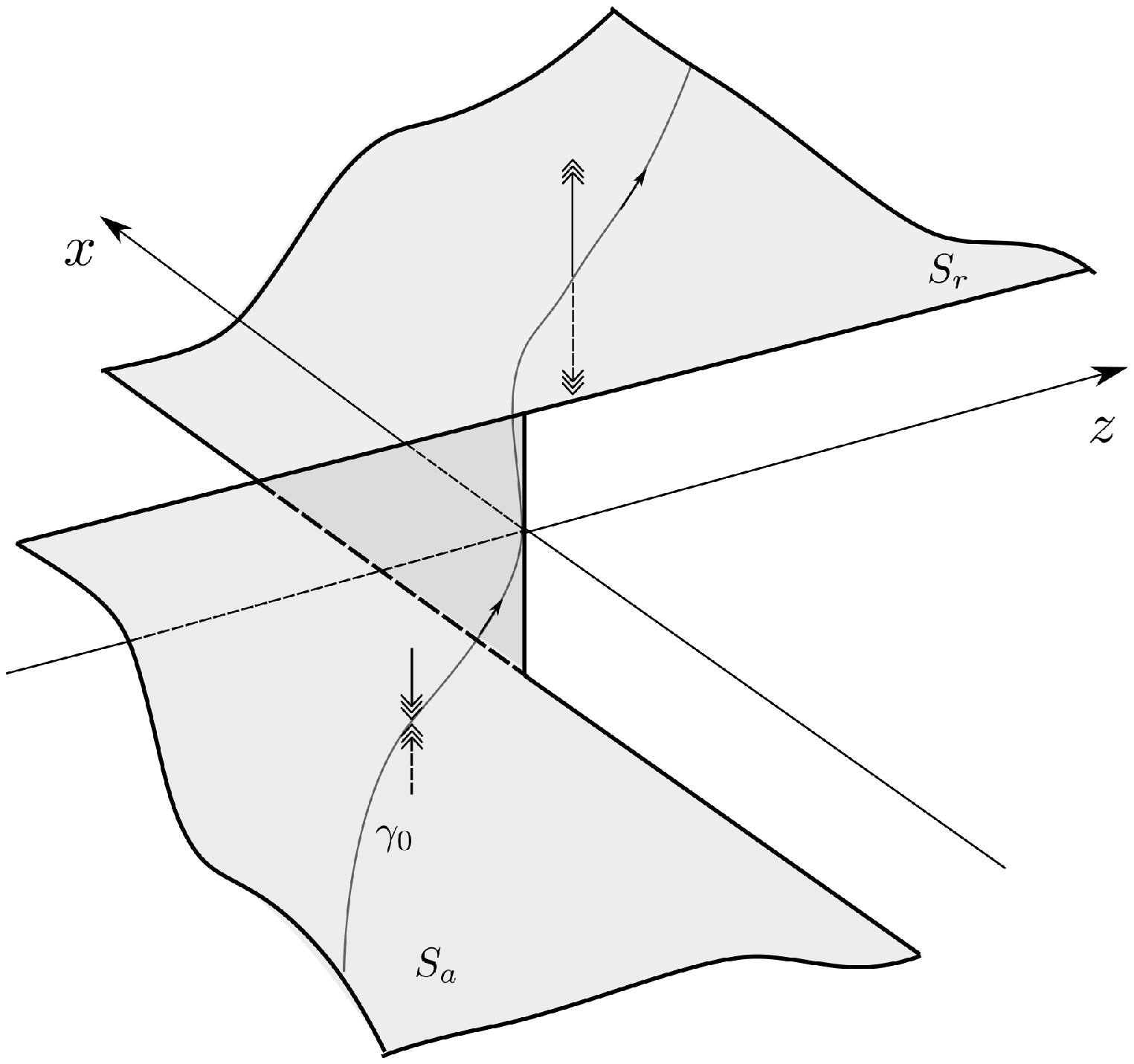}}
\subfigure[$\epsilon\ne 0$]{\includegraphics[width=.495\textwidth]{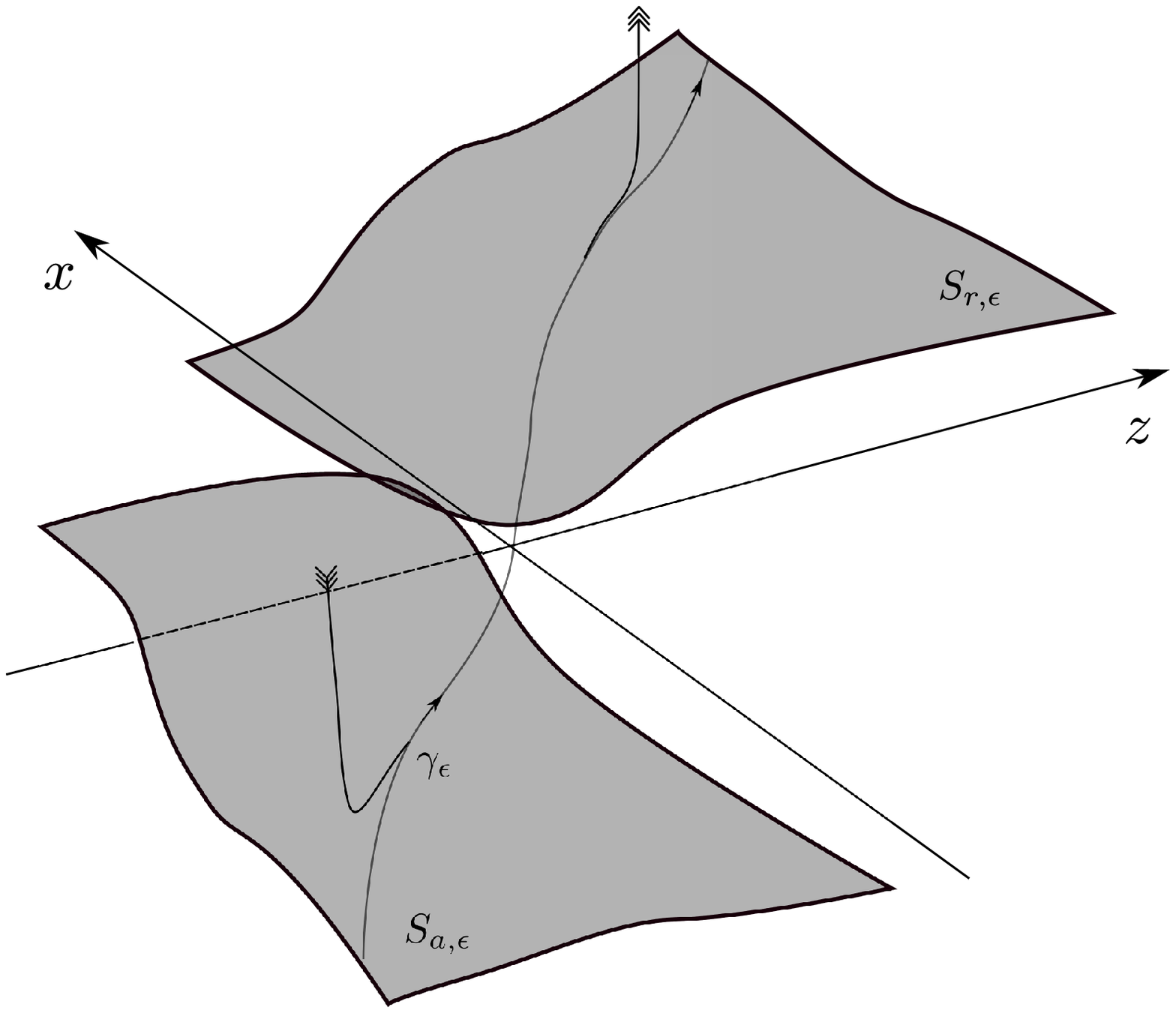}}
\end{center}
 \caption{(a) The normally hyperbolic critical manifolds $S_a$ and $S_r$ and (b) their perturbed versions $S_{a,\epsilon}$ and $S_{r,\epsilon}$. The manifolds $S_{a,\epsilon}$ and $S_{r,\epsilon}$ do not extend as far as $S_{a}$ and $S_{r}$, since Fenichel's theory only applies away from the fold lines. We have illustrated in (a) a singular canard $\gamma_0$ and in (b) a perturbed version $\gamma_\epsilon$ with its ends in $S_{a,\epsilon}$ and $S_{r,\epsilon}$. }
\figlab{SaSr2}
\end{figure}

\subsection{The connection between the results of section \secref{slidingVectorField} and the definition of canards for the regularized system}\seclab{connection}
From \thmref{criticalManifold} it follows that the analysis in section \secref{slidingVectorField} on singular canards can be directly translated into what we continue to call singular canards on the critical manifold for the limiting regularized system. A \textit{canard} for the regularized system \eqref{slowFastEquations1} is a trajectory that has its ends contained within the Fenichel slow manifold $S_{a,\epsilon}$ and $S_{r,\epsilon}$. Canards will be obtained as a perturbation of a singular canard and when this is possible we say that a singular canard \textit{persists} the regularization. Using the blow up method in the formulation of Krupa and Szmolyan \cite{krupa_extending_2001} we are able to continue $S_{a,\epsilon}$ and $S_{r,\epsilon}$ up close to $\tilde p$ by following the singular canards. This is not covered by \propref{fenichelProposition}. Canards will therefore be obtained as transverse intersections of $S_{a,\epsilon}$ and $S_{r,\epsilon}$. We illustrate a perturbed canard $\gamma_\epsilon$ with ends in $S_{a,\epsilon}$ and $S_{r,\epsilon}$ in \figref{SaSr2}. As $S_{a,\epsilon}$ and $S_{r,\epsilon}$ are in general not unique, canards will in general also be non-unique but they will be $\mathcal O(e^{-c/\epsilon})$-close. A canard is also called a primary canard if it goes from $S_{a,\epsilon}$ to $S_{r,\epsilon}$ in forward time. If it goes from $S_{r,\epsilon}$ to $S_{a,\epsilon}$ in forward time, then it is a faux canard.

\subsection{Remark on nomenclature on canards in slow-fast systems in $\mathbb R^3$}\seclab{comparison}
In \cite{brons-krupa-wechselberger2006:mixed-mode-oscil,szmolyan_canards_2001,wechselberger_existence_2005} the authors consider slow-fast systems in $\mathbb R^3$. Their reduced equations (e.g. \cite[Eq. (12)]{szmolyan_canards_2001}) are comparable to our equations \eqref{layerEqns} upon de-singularization. In these (and other) references, cases corresponding to our cases (S), (N) and (SN) also appear and are there referred to as \textit{folded saddle}, \textit{folded node} and \textit{folded saddle-node}, respectively. Here the adjective \textit{folded} is used to highlight the fact that the equilibrium appears on a fold line of the critical manifold. Other authors \cite[Fig. 17]{desroches_canards_2011} have suggested that this nomenclature be adopted for the canards that appear in the two-fold singularity in piecewise smooth systems. However this is misleading for two reasons. First, there is no underlying geometry in the piecewise smooth system that suggests the use of the word \textit{folded}. Second, even in the singular limit, the geometry of the critical manifold is different from that in slow-fast systems \cite{brons-krupa-wechselberger2006:mixed-mode-oscil,szmolyan_canards_2001,wechselberger_existence_2005} (for example, compare \cite[Fig. 2]{szmolyan_canards_2001} with \figref{SaSr2}). Of course, the critical manifold in these references also splits into an attracting critical manifold and a repelling one. But the closure of these manifolds coincides along a fold line. In our case, the closure of $S_a$ and $S_r$ only coincides in the line $\tilde p$ \eqref{tildep}, which when undoing the scaling \eqref{yTilde} and setting $\epsilon=0$, collapses to the two-fold $p$. The two-fold $p$ is the \textit{intersection} of the transverse fold lines $l^{\pm}$ \eqref{foldLines1}. 


 \section{Fold lines}\seclab{foldLines}
In this section we describe the blow up method in the formulation of Krupa and Szmolyan \cite{krupa_extending_2001} and demonstrate its application by considering the slow manifold near the fold line $\tilde l^-:\,x=0,\,y=-1$ but away from the two-fold singularity (the analysis of the fold line $\tilde l^+$ is identical). There are no canards in this section. We focus on the attracting region $S_a$ (corresponding to $\Sigma_{sl}^-$) by taking $z\le -c^{-1}$, $c$ sufficiently large but independent of $\epsilon$ (the case of $z\ge c^{-1}$ can be handled similarly). The case of the visible fold has been covered in \cite{reves_regularization_2014} in a $2D$ system. The purpose of this section is to demonstrate the use of the blow up method by extending the results of \cite{reves_regularization_2014} to our $3D$ system, carefully highlighting the dependency of the third variable, and to describe the invisible fold. In the next section, we apply the blow up method to the two-fold singularity. 
 
In \cite{reves_regularization_2014}, a two-dimensional system $(x,y)$ is considered with a switching manifold at $y=0$ containing a fold point at $(x,y)=0$. A regularizing function $\phi$ is used, which is $C^{k-1}$-smooth but no smoother ($2\le k<\infty$). A careful and lengthy asymptotic analysis is employed to conclude that the $1D$ slow manifold can be continued as an attracting invariant manifold close to the fold point. In the case of a visible fold, it is shown that the slow manifold leaves the slow-fast region at $y=-1$ when $x=\epsilon^{k/(2k-1)}\eta+\mathcal O(\epsilon^{3k/(2k-1)})$, where $\eta$ a non-zero constant that depends on $k$ and $\phi^{(k)}(-1)$ (see Theorem 2.2 and Section 3 of \cite{reves_regularization_2014}). 

In this section, we will use the blow up method to recover the results of \cite{reves_regularization_2014} more quickly and with a geometric insight that is necessarily absent from the asymptotic approach. The extension to the invisible fold case is performed at the same time. Another advantage of this method is that the steps used in this section are very similar to the ones that we need to take to study the two-fold singularity in the next section. 


Suppose, as in \cite{reves_regularization_2014,Sotomayor96}, that $\phi$ is a $C^\infty$-function on $\mathbb R\backslash \{-1,1\}$ but only $C^{k-1}$-smooth, $2\le k <\infty$, without being contained in $C^{n}$ for any $n\ge k$, on the whole of $\mathbb R$. Hence
\begin{align*}
 \phi^{(n)}(\pm 1)=0,\,\quad \text{for}\quad 0\le n\le k-1
\end{align*}
while at least one of the limits $\lim_{y\rightarrow  \pm 1^{\mp}} \phi^{(k)}(y)$ is different from $0$. 
Since we focus on $\tilde l^-$, we suppose that $\lim_{y\rightarrow -1^{+}} \phi^{(k)}(y)\ne 0$. By shifting $y$ by $1$:
\begin{align}
 y\mapsto \tilde y=1+y,\eqlab{yTranslation}
\end{align}
and dropping the tilde on $y$, we write $\phi$ as 
\begin{align}
\phi(y)=-1+\phi^{[k]} y^k+\mathcal O(y^{k+1}), \quad \phi^{[k]}\equiv \frac{\phi^{(k)}(-1)}{k!}>0.\eqlab{phik}
\end{align}
Using \eqref{slowFastEquations1} we obtain the following equations of motion
\begin{align}
 \dot x &= 2\epsilon \text{sign}(\beta) (1+\mathcal O(\epsilon+x+y^k+z)),\eqlab{EqFoldLines}\\
 \dot y&=(\vert b\vert + \mathcal O(x+z))z\phi^{[k]} y^k (1+\mathcal O(y))+2(-\vert \beta \vert+\mathcal O(x+z))x(1+\mathcal O(y^k))+\mathcal O(\epsilon),\nonumber\\
 \dot z &=2\epsilon \vert b\vert^{-1} \gamma (1+\mathcal O(\epsilon+x+y^k+z)),\nonumber\\
 \dot{\epsilon}&=0.\nonumber
\end{align}
\begin{remark}\remlab{Cinfty}
 {As in \cite{reves_regularization_2014}, it is not possible for us to handle the $C^\infty$-case, since $k$ in \eqref{phik} is finite. The existence of a finite $k$ (and a non-zero $\phi^{[k]}$) is crucial to our approach since we can then find a proper scaling for the blow up method. The $C^\infty$-case must therefore be handled by other methods. 
 
 The regularization of Sotomayor and Teixeira  \cite{Sotomayor96}  does not support analytic regularizations, since they necessarily deform the vector-fields outside $\vert y\vert<\epsilon$. }
\end{remark}

\subsection{Blow up transformation}
We now consider the transformed line $\tilde l^-:\,(x,y,z,\epsilon)=(0,0,z,0)$ in the extended phase space. It is non-hyperbolic (see \propref{nonHyperbolicLines}). The blow up method \cite{krupa_extending_2001} introduces a \textit{quasi-homogeneous blow up}, given by 
\begin{align}
\Lambda:\,B\equiv \mathbb R\times \overline{\mathbb R}_+ \times S^2&\rightarrow (x,y,z,\epsilon)\in \mathbb R^2\times \mathbb R_-\times \overline{\mathbb R}_+,\eqlab{qhblowup}\\
(z,r,(\overline x,\overline y,\overline \epsilon))&\mapsto  (x,y,z,\epsilon)=(r^{a_1} \overline x,r^{a_2} \overline y,z,r^{a_3} \overline \epsilon).\nonumber
\end{align}
The number $r$ is the \textit{exceptional divisor} such that when $r=0$ the blown-up coordinates collapse to the non-hyperbolic line. Applying $\Lambda$ therefore has the effect of blowing up the non-hyperbolic line to a cylinder $(z,(\overline x,\overline y,\overline \epsilon))\in \mathbb R \times S^2$. The weights $(a_1,a_2,a_3)$ are chosen so that the blown-up vector field has the exceptional divisor as a common factor. With a time-rescaling it is then possible to remove this common factor and de-trivialize the vector-field on the cylinder $(z,(\overline x,\overline y,\overline \epsilon))\in \mathbb R \times S^2$. {We will determine our weights below.}


Calculations are performed using local coordinates, although as noted in \cite{krupa_extending_2001}, it is almost essential not to use spherical coordinates. The correct choice of local coordinates is based on \textit{directional charts} $\kappa_i$, which in our case will correspond to $4$-dimensional spaces fixed at $\overline x=-1$ and $\overline \epsilon=1$, respectively \cite[Definition 3.1]{szmolyan_canards_2001}. On each chart $\kappa_i$, the vector-field is described using local coordinates that are introduced via a local blow up map $\mu_i$, which will be a \textit{directional blow up} in the direction corresponding to the chart $\kappa_i$, so that the quasi-homogeneous blow up defined in \eqref{qhblowup} becomes a composition: $\Lambda = \mu_i\circ \kappa_i$ \cite{krupa_extending_2001,szmolyan_canards_2001}. In our case, we need to consider only two charts:
\begin{align}
 \kappa_1:\, \overline x=-1.\quad &(z,r,(\overline x,\overline y,\overline \epsilon))\in \{B\,\vert \,\overline x<0\} \mapsto (r_1,y_1,z,\epsilon_1) \in \mathbb R^4,\eqlab{kappa1}\\
 r_1 &=r\left(-\overline{x}\right)^{1/a_1},\,y_1 =(\overline x)^{-a_2/a_1} \bar y,\,\epsilon_1 = (-\overline x)^{-a_3/a_1}\overline \epsilon,\nonumber
 \end{align}
 and
 \begin{align}
 \kappa_2:\, \overline \epsilon=1.\quad &(r,z,(\overline x,\overline y,\overline \epsilon))\in \{B\,\vert \,\overline \epsilon >0\} \mapsto (x_2,y_2,z,r_2)\in \mathbb R^4,\eqlab{kappa2}\\
 x_2 &=(\overline \epsilon)^{-a_1/a_3}\overline x,\,r_2 =r(\overline \epsilon)^{1/a_3},\,y_2 =(\overline \epsilon)^{-a_2/a_3} \bar y,\nonumber
\end{align}
and the local blow up transformations
\begin{align}
 \mu_1&:\,(r_1,y_1,z,\epsilon_1)\mapsto (x,y,z,\epsilon)=(-r_1^{a_1},r_1^{a_2} y_1,z,r_1^{a_3} \epsilon_1),\eqlab{mu1FoldLines}\\
 \mu_2&:\,(x_2,y_2,z,r_2)\mapsto (x,y,z,\epsilon)=(r_2^{a_1}x_2,r_2^{a_2} y_2,z,r_2^{a_3}).\eqlab{mu2FoldLines}
\end{align}
where $(a_1,a_2,a_3)$ are given in \eqref{weights} below. A good way of thinking about obtaining $\mu_1$ and $\mu_2$, is to insert $\overline x=-1$ from $\kappa_1$ and $\overline \epsilon=1$ from $\kappa_2$, respectively, into \eqref{qhblowup}. This then fixes the details of the maps $\kappa_1$ and $\kappa_2$ by the conditions $\Lambda=\mu_1\circ \kappa_1$ and $\Lambda=\mu_2\circ \kappa_2$, respectively. The details of $\kappa_1$ and $\kappa_2$ are not used in the sequel. For this reason the details of the charts are usually omitted by authors and instead simply referred to as $\kappa_1:\, \overline{x}=-1$ and $\kappa_2:\,\overline{\epsilon}=1$.  {Points at infinity in chart $\kappa_2$ correspond to the equator $\overline \epsilon=0$ of the blow up sphere. The purpose of chart $\kappa_1:\,\overline{x}=-1$ is to cover the part of the equator with $\overline x<0$.}

\subsection{On determining the weights $(a_1,a_2,a_3)$}
{The chart $\kappa_2$ is called the \textit{scaling chart} \cite{krupa_extending_2001} and 
corresponds cf. \eqref{mu2FoldLines} to an $\epsilon$-dependent scaling of the variables $(x,y,z)$. The variable $r_2$ is just a parameter since it is a function of $\epsilon$ only. The plane $r_2=0$ in this chart corresponds to the singular line $\tilde l^-$ in the original variables. If we re-scale the old vector-field, we find
\begin{align*}
 \dot x_2 &= 2r_2^{a_3-a_1}\text{sign}(\beta) (1+\mathcal O(z+r_2^{a_1}+r_2^{a_2k}+r_2^{a_3})),\\
 \dot y_2&=r_2^{a_2(k-1)}(\vert b\vert + \mathcal O(z+r_2^{a_1}))z\phi^{[k]} y_2^k (1+\mathcal O(r_2^{a_2}))-2r_2^{a_1-a_2}(\vert \beta \vert+\mathcal O(z+r_2^{a_1}))x_2(1+\mathcal O(r_2^{ka_2}))+\mathcal O(r_2^{a_3}),\nonumber\\
 \dot z &=2r_2^{a_3} \vert b\vert^{-1} \gamma (1+\mathcal O(z+r_2^{a_1}+r_2^{a_2k}+r_2^{a_3})),\nonumber\\
 \dot r_2&=0.\nonumber
\end{align*}
Note that this vector field vanishes when $r_2=0$. However, if we equate powers of $r_2$ in the first two equations, we can de-trivialize the dynamics for $r_2=0$ by dividing the vector-field by the resulting common factor. Since $r_2$ is constant, the division of the right hand sides corresponds to multiplying time by the same factor.
The new vector-field can then be analysed in the limit $r_2=0$ and it is possible to apply regular perturbation theory to capture the interesting case $r_2>0$. 

Equating the powers as described gives the following two linear homogeneous equations 
\begin{align*}
 a_3-a_1=a_2(k-1),\quad a_2(k-1) = a_1-a_2,
\end{align*}
for the three unknown weights $(a_1,a_2,a_3)$. The complete solution is spanned by
\begin{align}
a_1 = 2k,\,a_2 = 2,\,a_3= {2(2k-1)},\eqlab{weights}
\end{align}
with $k\in \mathbb N$. All non-zero solutions give rise to the same equations\footnote{The only difference between the solutions is the relation between $r_2$ and $\epsilon$} for $r_2=0$, which can be written as
\begin{align}
 \dot x_2 &= 2r_2^{2(k-1)}\text{sign}(\beta) (1+\mathcal O(z+r_2^{2k})),\eqlab{EqnsChartK2FoldLinesBeta}\\
 \dot y_2&=r_2^{2(k-1)}(\vert b\vert + \mathcal O(z))z\phi^{[k]} y_2^k -2r_2^{2(k-1)}(\vert \beta \vert+\mathcal O(z))x_2+\mathcal O(r_2^2),\nonumber\\
 \dot z &=2r_2^{2(2k-1)} \vert b\vert^{-1} \gamma (1+\mathcal O(z+r_2^{2k})),\nonumber\\
 \dot r_2&=0.\nonumber
\end{align}
Now we divide the right hand sides by the common factor $r_2^{a_3-a_1}=r_2^{2(k-1)}$ to obtain 
\begin{align}
 \dot x_2 &= 2\text{sign}(\beta) (1+\mathcal O(z+r_2^{2k})),\eqlab{EqnsChartK2FoldLines}\\
 \dot y_2&=(\vert b\vert + \mathcal O(z))z\phi^{[k]} y_2^k -2(\vert \beta \vert+\mathcal O(z))x_2+\mathcal O(r_2^2),\nonumber\\
 \dot z &=2r_2^{2k} \vert b\vert^{-1} \gamma (1+\mathcal O(z+r_2^{2k})),\nonumber\\
 \dot r_2&=0.\nonumber
\end{align}
Equations \eqref{EqnsChartK2FoldLines} describe the dynamics in chart $\kappa_2$}. They will be discussed in greater detail in section \secref{chartK2FoldLines} below.

The change of coordinates between the two charts $\kappa_1$ and $\kappa_2$ is determined by $\kappa_{12}$:
\begin{align}
\kappa_{12}:(x_2,y_2,z,r_2)\mapsto (r_1,y_1,z,\epsilon_1)= (r_2 (-x_2)^{-1/a_1},(-x_2)^{-a_2/a_1}y_2,z,(-x_2)^{-a_3/a_1}),\eqlab{kappa12FoldLines}
 \end{align}
defined for $x_2<0$, and its inverse $\kappa_{21}$:
\begin{align*}
 \kappa_{21}:(r_1,y_1,z,\epsilon_1)\mapsto (x_2,y_2,z,r_2)= (-\epsilon_1^{-a_1/a_3},\epsilon_1^{-a_2/a_3}y_1,z,r_1\epsilon_1^{a_1/a_3}).
\end{align*}
As in \cite{krupa_extending_2001} we will denote invariant objects in the blow up space $B$, defined in \eqref{qhblowup}, using bars, e.g. $\overline M_a, \overline M_r$ where subscript $a,r$ refers to attractive and repelling respectively. In addition, these objects will be given a second subscript corresponding to charts. So for example the manifold $\overline M_a$ will be denoted by $M_{a,i}$ in the chart $\kappa_i$.

%
%

\subsection{Dynamics in chart $\kappa_2$}\seclab{chartK2FoldLines}

We obtained the equations in this chart in \eqref{EqnsChartK2FoldLines} above. These equations with $r_2=0$ were used in \cite{reves_regularization_2014} to construct the inner solution of an asymptotic expansion. 
After the desingularization (the process of going from \eqref{EqnsChartK2FoldLinesBeta} to \eqref{EqnsChartK2FoldLines} by division of $r_2^{2(k-1)}$ and then setting $r_2=0$), the space $r_2=0$ carries non-trivial dynamics. Using Proposition 3.10 of \cite{reves_regularization_2014}, we deduce the existence of a family of trajectories $\gamma_2(z)$ (parametrized by $z$) within $r_2=0$, where
\begin{align}
y_2 = \left(  \frac{2(\vert \beta \vert+\mathcal O(z)) x_2}{\phi^{[k]}(\vert b\vert + \mathcal O(z))z } \right)^{1/k}+\mathcal O(1).\eqlab{gamma2z}
\end{align}
for $x_2\ll 0$.

In the visible case ($\beta>0$), each solution with $z<0$ is obtained as an overflowing center manifold and $\gamma_2(z)$ is therefore unique.
Each trajectory of $\gamma_2(z)$ intersects $y_2=0$ at a point $(x_2,y_2,z,r_2) = (\eta,0,z,0)$ where $\eta=\eta(z,k,\phi^{[k]})\ne 0$. By performing the scaling
\begin{align*}
x_2&= 2^{1/(2k-1)} \left(\vert b\vert^{-1/(2k-1)}\vert \beta\vert^{-(k-1)/(2k-1)}+\mathcal O(z)\right) (-\phi^{[k]} z)^{-1/(2k-1)} \tilde x_2,\\
 y_2&= 2^{2/(2k-1)} \left(\vert b\vert^{-2/(2k-1)} \vert\beta \vert^{1/(2k-1)}+\mathcal O(z)\right) (-\phi^{[k]} z)^{-2/(2k-1)} \tilde y_2
\end{align*}
and introducing new scaled variables $(\tilde x_2,\tilde y_2)$, we obtain the equations 
\begin{align*}
\dot{\tilde x}_2 &=1,\\
\dot{\tilde y}_2 &=-\tilde y_2^k-\tilde x_2,
\end{align*}
(with respect to a new time) considered in \cite[Proposition 3.10]{reves_regularization_2014}. 
The advantage of this scaling is that \eqref{gamma2z} in these coordinates now intersects $\tilde y_2=0$ in $(\tilde x_2,\tilde y_2,z,r_2)=(\tilde \eta,0,z,0)$ where the corresponding $\tilde \eta=\tilde \eta(k)$ only depends upon $k$. Therefore, we can conclude that $\eta$ takes the form:
\begin{align}
  \eta(z,k,\phi^{[k]}) &= 2^{1/(2k-1)} \left(\vert b\vert^{-1/(2k-1)}\vert \beta\vert^{-(k-1)/(2k-1)}+\mathcal O(z)\right) (-\phi^{[k]} z)^{-1/(2k-1)} \tilde \eta(k)\nonumber\\
  &=\mathcal O((-\phi^{[k]} z)^{-1/(2k-1)}).\eqlab{etaExpr}
\end{align}
Consider the sections 
\begin{align}
\Lambda_2^{in}:\,x_2=-\rho^{-1};\quad \Lambda_2^{out}:\,y_2=0, \eqlab{Sigma2inFoldLines}
\end{align}
with $\rho$ small. Let $$\Pi_2:\,\Lambda_2^{in}\rightarrow \Lambda_2^{out},$$ be the transition from $\Lambda_2^{in}$ to $\Sigma_1^{in}$ induced by the forward flow of \eqref{EqnsChartK2FoldLines} in a neighborhood of $\gamma_{2}(z)$. By the previous arguments we have that $\Pi_2(\gamma_2 \cap \Lambda_2^{in})=(\eta,0,z,r_2)$. By regular perturbation theory $\Pi_2$ maps a neighborhood of $\gamma_2 \cap \Lambda_2^{in}$ diffeomorphically onto a neighborhood of $\gamma_2 \cap \Lambda_2^{out}=(\eta,0,z,0)$. 

In the invisible case, $\gamma_2(z)$ is obtained as a non-unique attracting center manifold and a full neighborhood of $(x_2,y_2,z,r_2)=(0,0,z,0)$ contracts towards $\gamma_2(z)$. 
%
\subsection{Dynamics in chart $\kappa_1$}
The advantage of chart $\kappa_1$ is that it enables us to follow $S_{a,\epsilon}$ close to the fold-line and connect with the analysis in the scaling chart $\kappa_2$. In the language of asymptotic methods one can say that the chart $\kappa_1$ enables us to match an inner solution from the scaling chart with an outer solution obtained by Fenichel's theory. 
Inserting \eqref{mu1FoldLines} into \eqref{EqFoldLines} gives the following equations
\begin{align}
 \dot r_1 &=-k^{-1}\text{sign}(\beta) r_1 \epsilon_1,\eqlab{chartK1EquationsFoldLines}\\
 \dot y_1 &= 2k^{-1} \text{sign}(\beta) \epsilon_1 y_1 + (\vert b\vert+\mathcal O(z))z \phi^{[k]} y_1^k (1+\mathcal O(r_1^{a_2}))+2\vert \beta \vert +\mathcal O(z+r_1^2),\nonumber\\
 \dot z &= 2 r_1^{2k} \epsilon_1 \vert b\vert^{-1} \gamma(1+\mathcal O(z+r_1^2)),\nonumber\\
 \dot \epsilon_1 &= a_3 \text{sign}(\beta) \epsilon_1^2,\nonumber
\end{align}
where we have divided by the common factors $r_1^{2(k-1)}$ and $1+\mathcal O(r_1^2)$. There exist two invariant planes, $r_1=0$ and $\epsilon_1=0$, of equations \eqref{chartK1EquationsFoldLines} which intersect in an invariant line:
\begin{align*}
L_{a,1}:\,(r_1,y_1,z,\epsilon_1) = (0,\left(  \frac{2(\vert \beta \vert+\mathcal O(z))}{\phi^{[k]}(\vert b\vert + \mathcal O(z))(-z) } \right)^{1/k},z,0),\,z<0.
\end{align*}
The linearization about each $p\in L_{a,1}$ has three zero eigenvalues and one negative eigenvalue. The $r_1,z$ and $\epsilon_1$-directions are neutral, and the $y_1$-direction is contractive.
Within $r_1=0$ and for $\epsilon_1$ sufficiently small, there exists a center manifold $C_{a,1}$ at $p\in L_{a,1}$ given by
\begin{align}
 C_{a,1}:\,(r_1,y_1,z,\epsilon_1) = (0,\left(\frac{2(\vert \beta \vert+\mathcal O(z))}{\phi^{[k]}(\vert b\vert + \mathcal O(z))(-z) } \right)^{1/k} + \mathcal O(\epsilon_1),z,\epsilon_1). \eqlab{CMa1}
\end{align}
Within $C_{a,1}$, $\dot{z}=0$. The center manifold $C_{a,1}$ is unique as a center manifold of $L_{a,1}$ if $\dot{\epsilon}_1 >0$ and non-unique if $\dot{\epsilon}_1<0$ (see \figref{centerManifoldCa2FoldLine}). Following \eqref{chartK1EquationsFoldLines} $C_{a,1}$ is therefore unique if the fold-line is visible ($\beta>0$) and non-unique if the fold-line is invisible ($\beta<0$). The family of trajectories $\gamma_2(z)$ from the chart $\kappa_2$ can be transformed into the chart $\kappa_1$, using $\kappa_{12}$ \eqref{kappa12FoldLines}:
\begin{align*}
\gamma_1(z)=\kappa_{12}(\gamma_2(z)):\,(r_1,y_1,z,\epsilon_1) = (0,(-x_2)^{-a_2/a_1}y_2,z,(-x_2)^{-a_3/a_1}),
\end{align*}
with $y_2=y_2(x_2,z)$ in \eqref{gamma2z}. 
Therefore the family of trajectories $\gamma_1(z)$ is contained within $r_1=0$. Each trajectory in this family is backward (forward) asymptotic to a point $(0,\left(  \frac{2(\vert \beta \vert+\mathcal O(z))}{\phi^{[k]}(\vert b\vert + \mathcal O(z))(-z) } \right)^{1/k},z,0)$ on the invariant line $L_{a,1}$ for $\beta>0$ ($\beta<0$) (cf. \eqref{EqnsChartK2FoldLines}).  

\begin{figure}[h!] 
\begin{center}
\subfigure[Visible fold-line: unique $C_{a,1}$]{\includegraphics[width=.45\textwidth]{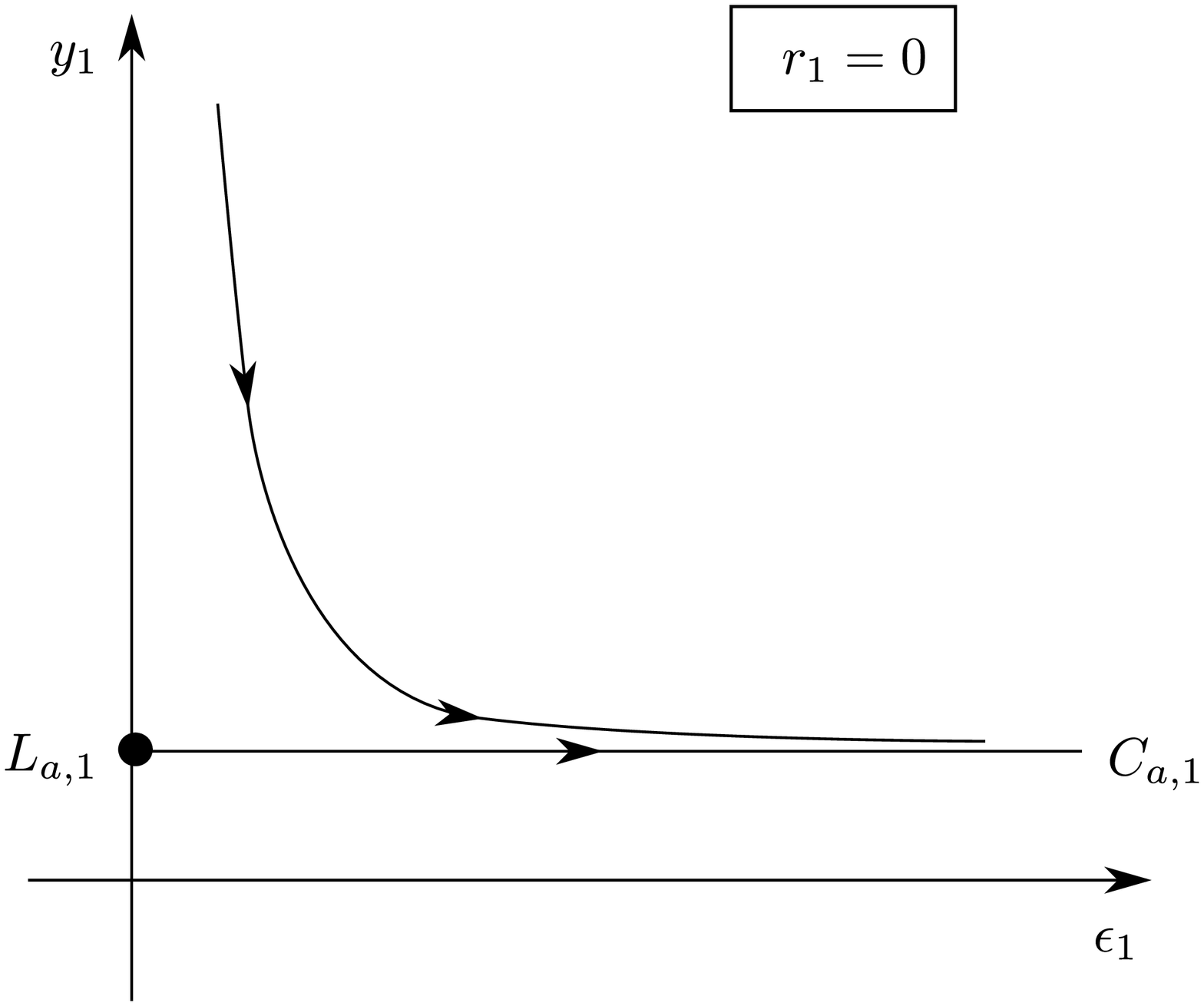}}
\subfigure[Invisble fold-line: non-unique $C_{a,1}$]{\includegraphics[width=.45\textwidth]{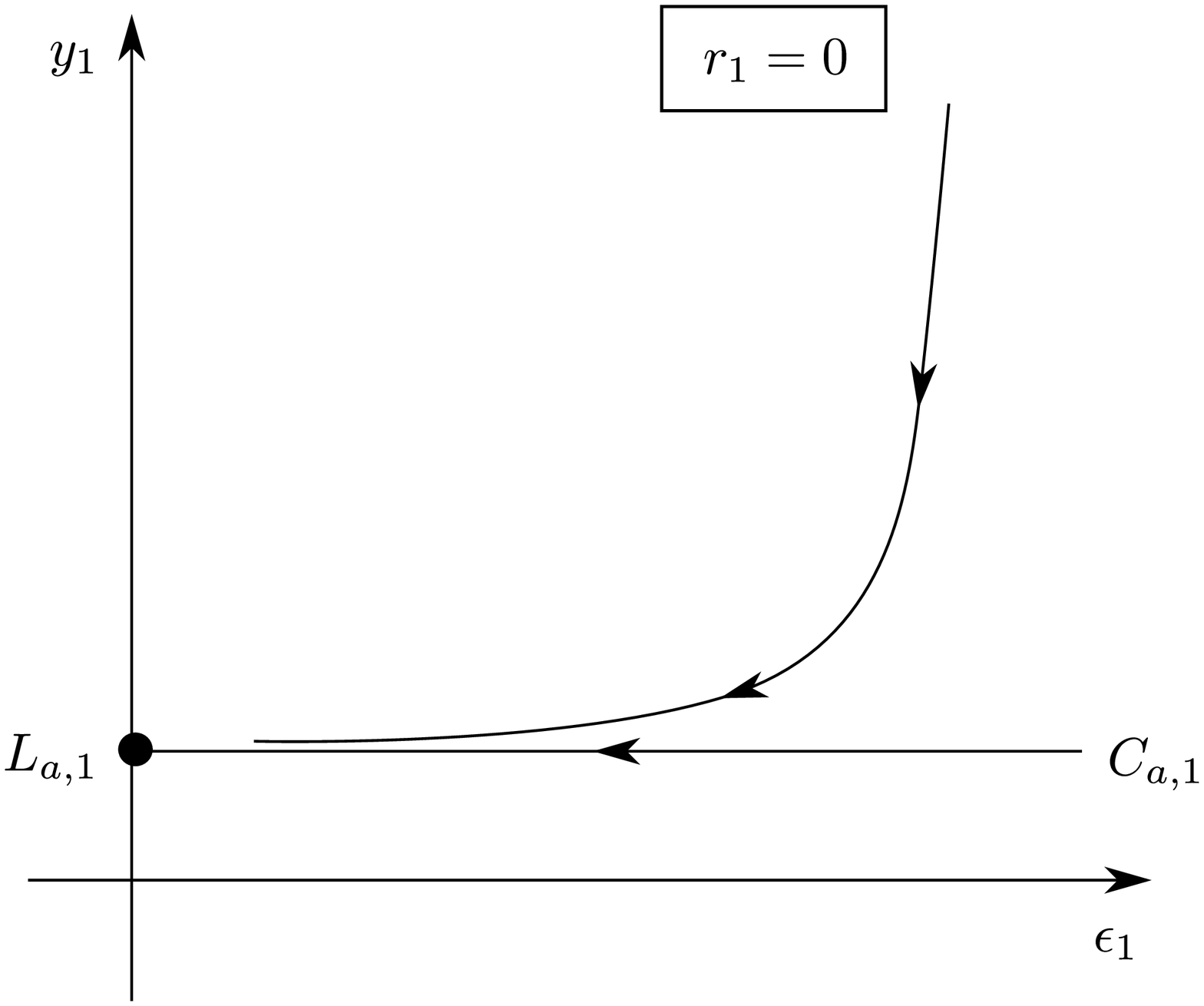}}
\end{center}
\caption{Illustration of the attracting center manifold $C_{a,1}$ \eqref{CMa1} within $r_1=0$ using a projection onto the $(\epsilon_1, y_1)$-plane. (a) The motion on $C_{a,1}$ is away from $L_{a,1}$ (which is a point in this projection) and $C_{a,1}$ is unique. (b) The motion on $C_{a,1}$ is towards $L_{a,1}$ and $C_{a,1}$ is non-unique.  }
\figlab{centerManifoldCa2FoldLine}
\end{figure}

Within $\epsilon_1=0$ we find a manifold $S_{a,1}$ of equilibria given by
\begin{align*}
 S_{a,1}:\,(r_1,y_1,z,\epsilon_1) = (r_1,\left(\frac{2(\vert \beta \vert+\mathcal O(z))}{\phi^{[k]}(\vert b\vert + \mathcal O(z))(-z) } \right)^{1/k} + \mathcal O(r_1),z,0).
\end{align*}
This is the critical manifold $S_{a}$ written in chart $\kappa_1$. The two invariant manifolds $S_{a,1}$ and $C_{a,1}$ are both contained within a $3D$ center manifold $M_{a,1}$ of $L_{a,1}$ given by
\begin{align*}
 M_{a,1}:\,(r_1,y_1,z,\epsilon_1) = (r_1,\left(\frac{2(\vert \beta \vert+\mathcal O(z))}{\phi^{[k]}(\vert b\vert + \mathcal O(z))(-z) } \right)^{1/k} + \mathcal O(r_1+\epsilon_1),z,\epsilon_1).
\end{align*}
The center manifold $M_{a,1}$ is foliated by $\epsilon=r_1^{a_3}\epsilon_1=\text{const}.$ We denote such an invariant foliation by $M_{a,1}(\epsilon)$ in chart $\kappa_1$, and note that it corresponds to the slow manifold $S_{a,\epsilon}$ where this is defined by Fenichel's theory. The manifold $C_{a,1}(\epsilon)$ intersects $\Lambda_1^{in}$.The slow flow on $M_{a,1}$ is determined by
\begin{align*}
 \dot r_1 &=-k^{-1}\text{sign}(\beta) r_1,\\
 \dot z &= 2 r_1^{2k} \vert b\vert^{-1} \gamma(1+\mathcal O(r_1^2)),\\
 \dot \epsilon_1 &= a_3 \text{sign}(\beta) \epsilon_1,
\end{align*}
where we have divided out the common factor $\epsilon_1$. In \figref{centerManifoldFoldLines}, we illustrate the dynamics, using a projection onto $z=\text{const}.<0$.

\begin{figure}[h!] 
\begin{center}
\subfigure[Visible fold-line]{\includegraphics[width=.45\textwidth]{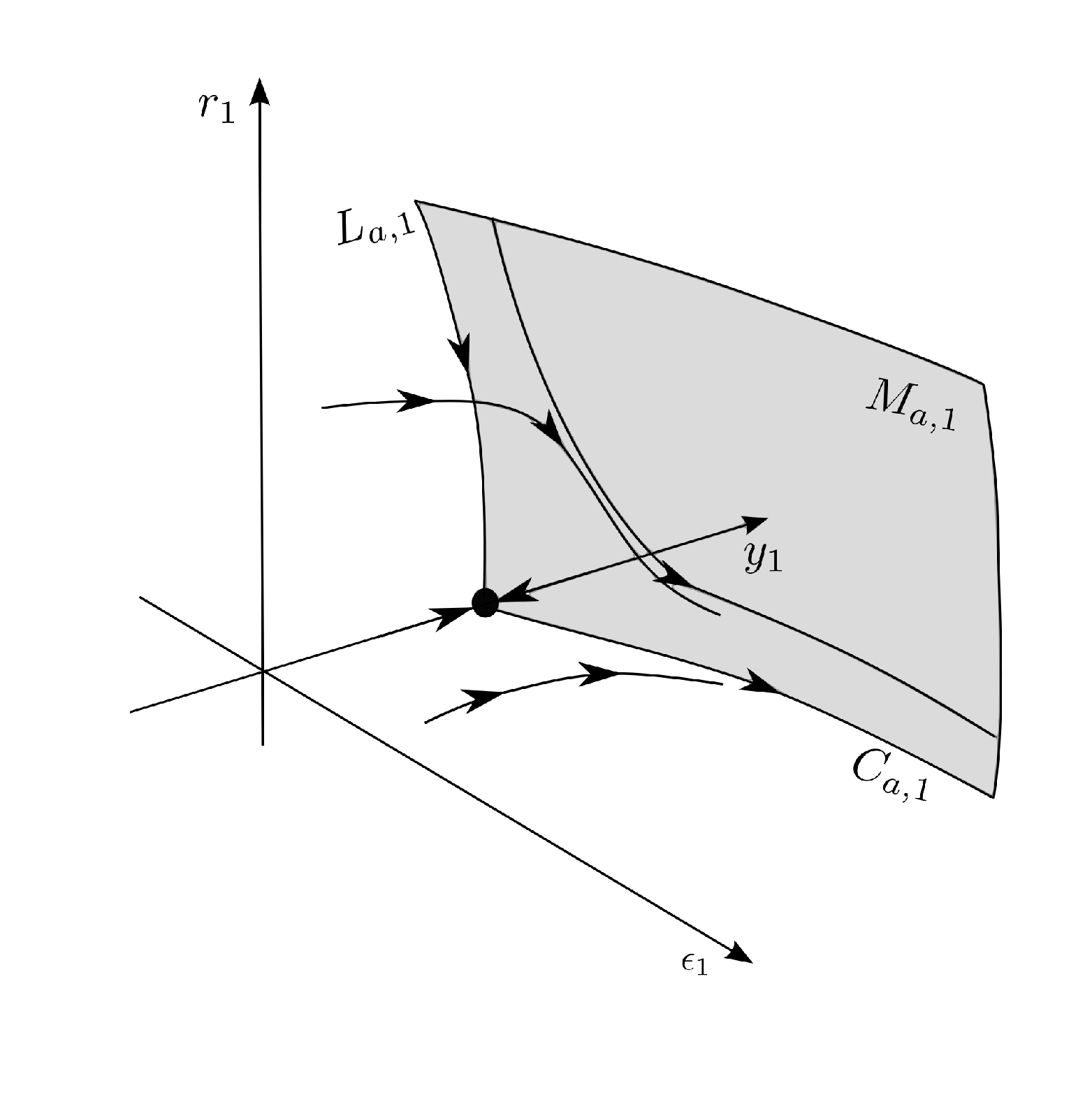}}
\subfigure[Invisible fold-line]{\includegraphics[width=.45\textwidth]{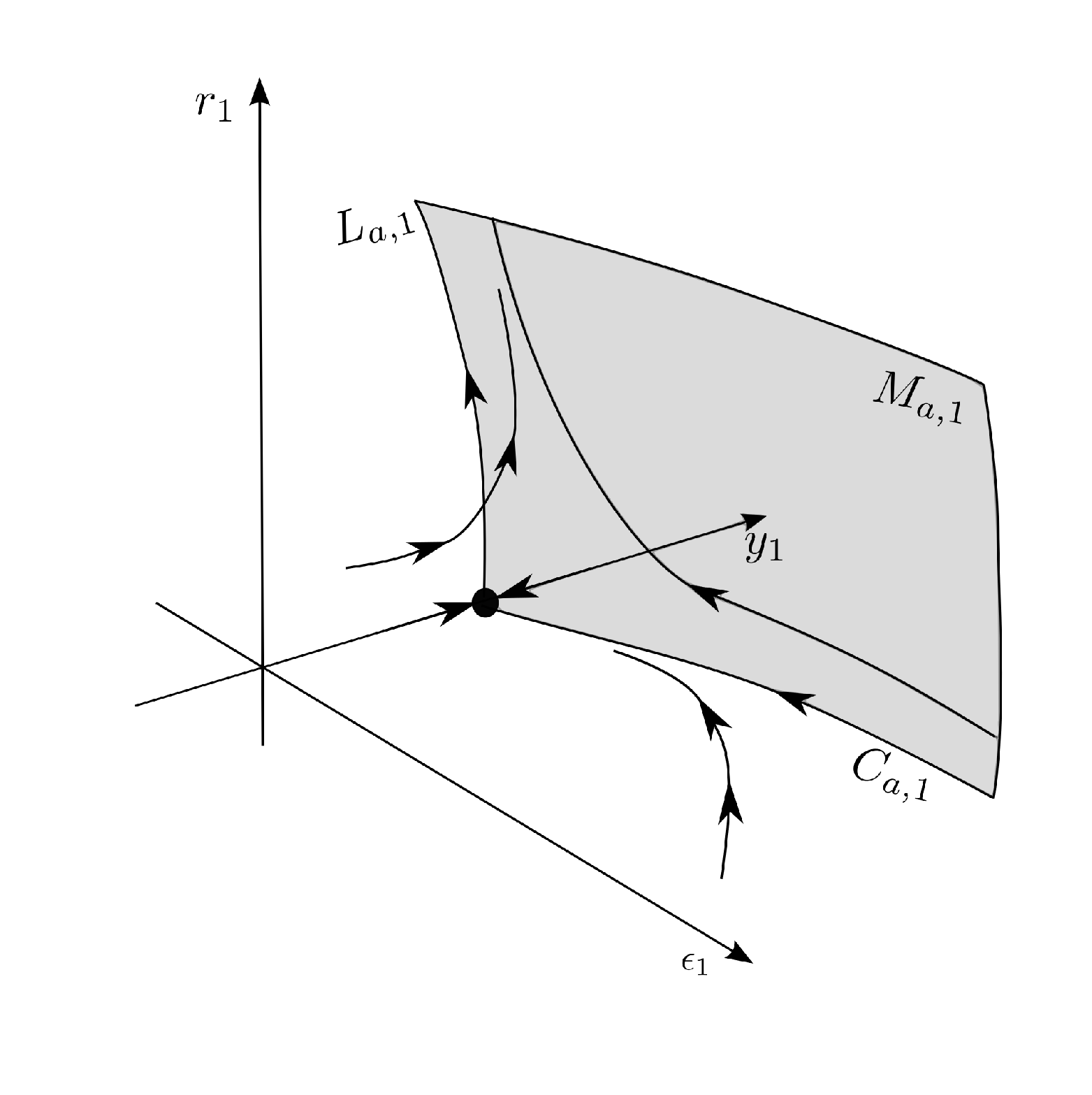}}
\end{center}
\caption{Illustration of the attracting center manifold $M_{a,1}$ using a projection onto $z=\text{const}.<0$. }
\figlab{centerManifoldFoldLines}
\end{figure}

The section $\Lambda_2^{in}$ corresponds to $$\Lambda_1^{in}:\,\epsilon_1=\rho^{a_3/a_1},$$ in chart $\kappa_1$. The manifold $C_{a,1}$ intersects $\Lambda_1^{in}$ in 
\begin{align}
 C_{a,1}\cap \Lambda_1^{in}:\,y_1 = \left(\frac{2(\vert \beta \vert+\mathcal O(z))}{\phi^{[k]}(\vert b\vert + \mathcal O(z))(-z) } \right)^{1/k} + \mathcal O(\rho),\,\epsilon_1 =\rho.\eqlab{Ca1Gamma}
\end{align}
The intersection of $M_{a,1}(\epsilon)$ with $\Lambda_1^{in}$ is $\mathcal O(r_1)$-close to $C_{1,a}\cap \Lambda_1^{in}$.
\subsection{Conclusions on the analysis of the fold-line $\tilde l^-$}
First we consider the visible case ($\beta>0$). Then our conclusions are very similar to those in \cite{krupa_extending_2001}. From the results obtained in the two charts, it can be concluded that the family of trajectories $\gamma_1(z)=\kappa_{12}(\gamma_2(z))$ for $x_2\ll 0$ is contained within the unique center manifold $C_{a,1}$. 
 This is where we need the constant $\rho$ sufficiently small, see \eqref{Sigma2inFoldLines}. Using this solution as guide, we can continue $M_{a,1}(\epsilon)$ into chart $\kappa_2$ as an invariant manifold $M_{a,2}(\epsilon)$ for $r_2$ sufficiently small. 
In particular, since $M_{a,2}(\epsilon)\cap \Lambda_2^{in}$ is $\mathcal O(r_2)$-close to $\gamma_2(z)\cap \Lambda_2^{in}$, and by the properties of $\Pi_2$, we have that $M_{a,2}(\epsilon)\cap \Lambda_2^{out}$ is $\mathcal O(r_2)$-close to $$\gamma_2(z)\cap \Lambda_2^{out}:\,(x_2,y_2,z,r_2) = (\eta,0,z,0).$$
\begin{theorem}
If $\beta>0$ (the visible fold), then $M_{a,2}(\epsilon)$ intersects $y=0$ (corresponding to $y=-1$ in our original coordinates \eqref{yTranslation}) in a curve:
\begin{align}
 x&=\epsilon^{a_1/a_3} (\eta+\mathcal O(\epsilon^{1/a_3}))=\mathcal O(\epsilon^{k/(2k-1)}(-\phi^{[k]}z)^{-1/(2k-1)}), \eqlab{MaIntersectionFoldLines}
\end{align}
with $\eta=\eta(z,k,\phi^{[k]})$ as in \eqref{etaExpr}. 
\end{theorem}
\begin{proof}
Here we have used \eqref{mu2FoldLines} to \textit{blow down}. In particular, we have used that $$r_2=\epsilon^{1/a_3}=\epsilon^{1/(2(2k-1))}.$$
\end{proof}

Using the attractiveness of $M_{a,1}(\epsilon)$, the potential non-uniqueness of $S_{a,\epsilon}$ only manifests itself in exponentially close curves, each of the form \eqref{MaIntersectionFoldLines}.

The invisible case ($\beta<0$) is easier. By the non-uniqueness of $C_{a,1}$, it can be concluded that a full neighborhood of the fold-line $(x_2,y_2,z,\epsilon_2)=(0,0,z,0)$ in chart $\kappa_2$ is contained within $W^s(S_{a,\epsilon})$.



 
\section{The two-fold singularity}\seclab{twoFold}
We now move on to the two-fold singularity, the main focus of our paper. The remainder of the paper will aim to prove the following theorem:
\begin{theorem}\thmlab{mainRes}
 There exists an $\epsilon_0>0$ so that the following statements hold true. 
\begin{itemize}
 \item Every locally unique primary singular canard $\gamma_0$ persists the regularizing perturbation in the sense that it gives rise to a locally unique canard $\gamma_\epsilon$ of \eqref{slowFastEquations1} for $\epsilon\le \epsilon_0$ with the following property: Let $\mathcal V$ be a neighborhood of the two-fold singularity. Then $\gamma_{\epsilon}\cap \mathcal V$ is $\mathcal O(\sqrt{\epsilon})$-close to $\gamma_0 \cap \mathcal V$.
 \item In the cases where there exists a sector of primary singular canards contained within $\Sigma_{sl}$, then there exists a locally unique canard of the regularized system $\gamma_\epsilon$ for $\epsilon\le \epsilon_0$ provided a non-resonance condition
 \begin{align*}
  \lambda_{+}^{-1}\lambda_{-} \notin \mathbb N,
 \end{align*}
 holds, where $\lambda_{\pm}$ are given by \eqref{lambdapm}.  The projection of such a $\gamma_\epsilon$ onto the $(x,z)$-plane is $\mathcal O(\sqrt{\epsilon})$-close to tangency with $v_{+}$ at the two-fold singularity. 
 \end{itemize}
 \end{theorem}
 We directly deduce the following important corollary:
\begin{corollary}
The regularized system only fully retains the features of the singular canards in the piecewise smooth system in the cases when the sliding region does not include a full sector of singular canards.
\end{corollary}
 \begin{remark}
 According to \lemmaref{uniqueness} the locally unique primary singular canards, considered in the first bullet-point of the theorem, is either strong primary singular canards in the case (N) or primary singular canards in case (S).
\end{remark}

To prove the theorem it is useful to introduce the following affine transformation 
\begin{align}
w=\frac{1-\phi(y)}{1+\phi(y)}\in (0,\infty) \quad \text{for} \quad y\in (-1,1),\eqlab{wTransformation}
\end{align}
into \eqref{slowFastEquations1} to obtain:
\begin{eqnarray}
\dot x &=&\epsilon ( \vert \beta\vert^{-1} c+\text{sign}(\beta)w+\mathcal O(\epsilon + x + z)),\eqlab{slowFastEquations2}\\
\dot w &=&f(w)(-(\vert b\vert+\mathcal O(x+z)) {z} + (\vert \beta\vert +\mathcal O(x+z))xw+\mathcal O(\epsilon)),\nonumber\\
 \dot z&=&\epsilon ( \text{sign}(b)+\vert b\vert^{-1} \gamma w+\mathcal O(\epsilon + x + z)),\nonumber
\end{eqnarray}
where $$f(w)=\frac12 (1+w)^2\phi'(\phi^{-1}((1-w)/({1+w})))>0,$$ for $y\in (-1,1)$. We have here also multiplied the vector-field by $\frac12 (1+w)>0$.
Then the critical manifold $S_a$ takes the following form
\begin{eqnarray}
%
S_a:\,w =\frac{(\vert b \vert +\mathcal O(x+z)) z}{(\vert \beta \vert +\mathcal O(x+z))x},\,(x,0,z)\in \Sigma_{sl}^-.
  \eqlab{wSlowManifold}
\end{eqnarray}
Moreover, the (interior of the) non-hyperbolic line $\tilde p$ \eqref{tildep} is transformed into
\begin{align}
 \tilde p:\,x=0,\,z=0,\,\epsilon=0,\,w\in (0,\infty). \eqlab{tildep2}
\end{align}
We will continue to refer to these objects as $S_a$ and $\tilde p$, respectively.

\subsection{Blow-up transformation}
We blow up the transformed non-hyperbolic invariant line $\tilde p$ \eqref{tildep2} (see \propref{nonHyperbolicLines}) using the following blow up transformation
\begin{eqnarray*}
 \Gamma:\,x=r\overline x,\,z=r\overline z,\,\epsilon = r^2\overline \epsilon,
\end{eqnarray*}
with $(\overline x,\overline z,\overline \epsilon)\in S^2$. Note that fast variable $w$ does not transform. As in the case of the fold, we consider the following charts 
$$\kappa_1:\,\overline x=-1,\quad \text{and}\quad \kappa_2:\,\overline \epsilon=1.$$ 
However the corresponding local blow up transformations 
\begin{align}
 \mu_1:\,\,x=-r_1,\,z=r_1z_1,\,\epsilon = r_1^2\epsilon_1,\eqlab{chartK2}
 \end{align}
 and
 \begin{align}
 \mu_2:\,x=r_2x_2,\,z=r_2z_2,\,\epsilon = r_2^2.\eqlab{chartK1}
 \end{align}
 are different. We will continue to use the same notation as we did in section \secref{foldLines}, even though there will be different expressions for $S_{a,1}$, $M_{a,1}$ and $C_{a,1}$ (see below). We believe it is useful to duplicate the notation because it stresses the standardization of the method, emphasizes the similarity of the arguments and leads to related geometric objects.

The change of coordinates between the different charts is given by
\begin{eqnarray*}
\kappa_{12}&:(x_2,w,z_2,r_2)\mapsto (r_1,w,z_1,\epsilon_1)=&(-r_2 x_2,w,\,-x_2^{-1} z_2,\,x_2^{-2}),\\
\kappa_{21}&:(r_1,w,z_1,\epsilon_1)\mapsto (x_2,w,z_2,r_2) =&(-1/\sqrt{\epsilon_1},\,w,\,z_1/\sqrt{\epsilon_1},\,r_1 \sqrt{\epsilon_1}),
\end{eqnarray*}
defined for $x_2<0$. 
\subsection{Dynamics in chart $\kappa_1$}
Inserting \eqref{chartK2} into \eqref{slowFastEquations2} gives the following equations
\begin{eqnarray}
 \dot r_1 &=&-r_1\epsilon_1 F_1(r_1,w,z_1,\epsilon_1),\eqlab{eqnChartK2}\\
 \dot w &=&f(w) \left(-\vert  b\vert z_1-\vert \beta \vert w+\mathcal O(r_1)\right),\nonumber\\
 \dot z_1 &=&\epsilon_1 \left(\text{sign}(b)+\vert b\vert^{-1}\gamma w +\mathcal O(r_1)+F_1(r_1,w,z_1,\epsilon_1)z_1\right),\nonumber\\
 \dot \epsilon_1&=&2 F(r_1,w,z_1,\epsilon_1)\epsilon_1^2,\nonumber
\end{eqnarray}
where
\begin{eqnarray*}
 F_1(r_1,w,z_1,\epsilon_1) = \vert \beta\vert^{-1} c+\text{sign}(\beta) w+\mathcal O(r_1).
\end{eqnarray*}
and we have rescaled time by a factor of $r_1$. 
%

As for the fold, in chart $\kappa_1$,  $\epsilon_1=0$ and $r_1=0$ are two invariant spaces. Their intersection is a line of equilibria $L_{a,1}$ determined by
\begin{align*}
 L_{a,1}:\,(r_1,w,z_1,\epsilon_1)=(0,-\vert \beta \vert^{-1} {\vert b\vert}z_1,z_1,0),\,z_1<0.
\end{align*}
Within $r_1=0$ there exists a $2D$ center manifold $C_{a,1}$ of $L_{a,1}$ which can be written as a graph over $(z_1,\epsilon_1)$: 
\begin{align*}
 C_{a,1}:\,(r_1,w,z_1,\epsilon_1) = (0,-\vert \beta \vert^{-1} {\vert b\vert}z_1+\mathcal O(\epsilon_1),z_1,\epsilon_1),\,z_1<0,\,\epsilon_1\ge 0.
\end{align*}
Within $C_{a,1}$ there exists invariant lines $l_{a,1,\pm}$ given by:
 \begin{eqnarray}
  l_{a,1,\pm}:\,(r_1,w,z_1,\epsilon_1)=(0,-\vert \beta \vert^{-1} {\vert b\vert}\chi_{\pm},\chi_{\pm},\epsilon_1),\,\epsilon_1\ge 0.\eqlab{la2pm}
 \end{eqnarray}
%
 The pair $(w,z_1)=(-\vert \beta \vert^{-1} {\vert b\vert} \chi_{\pm },\chi_{\pm})$ is obtained as a solution of $(\dot w,\dot z_1)=(0,0)$ for $r_1=0$ where $\chi_{\pm}$ are given in \eqref{chipm}.
We only consider $\chi_{\pm}<0$ since this corresponds to the stable sliding region $\Sigma_{sl}^-$. Then also $w=-\vert \beta \vert^{-1} {\vert b\vert} \chi_{\pm }>0$ as it should be according to \eqref{wTransformation}.
%

The space $\epsilon_1=0$ contains a $2D$-manifold $S_{a,1}$ of equilibria given by
\begin{align*}
 S_{a,1}:\,(r_1,w,z_1,\epsilon_1)=(r_1,-\vert \beta \vert^{-1} {\vert b\vert} z_1+\mathcal O(r_1),z_1,0),\,r_1\ge 0,\,z_1<0.
\end{align*}
The invariant manifold $S_{a,1}$ corresponds to the critical manifold $S_{a}$ \eqref{wSlowManifold}. The manifolds $C_{a,1}$ and $S_{a,1}$ are both contained within a $3D$ attracting center manifold $M_{a,1}$ of $L_{a,1}$, which can be written as a graph $w=m(r_1,z_1,\epsilon_1)$ over $(r_1,z_1,\epsilon_1)$:
\begin{align}
M_{a,1}:\, w=m(r_1,z_1,\epsilon_1)=-\vert \beta \vert^{-1} {\vert b\vert}z_1+\mathcal O(r_1+\epsilon_1)
\eqlab{mCenterManifold}
\end{align}
where the function $m$ satisfies the following condition:
\begin{align}
 m(0,\chi_{\pm},\epsilon_1) = -\vert \beta \vert^{-1} {\vert b\vert} \chi_{\pm},\eqlab{mCenterManifoldProperty}
\end{align}
since $M_{a,1}$ also contains the invariant lines $l_{a,1,\pm}$.  
The center manifold $M_{a,1}$ is foliated by invariant sub-manifolds given by $\epsilon=r_1^2\epsilon_1=\text{const}.$ We denote such an invariant sub-manifold by $M_{a,1}(\epsilon)$ in the chart $\kappa_1$ and $\overline M(\epsilon)$ in the blow up space. ${M}_{a,1}(\epsilon)$ corresponds to the slow manifold $S_{a,\epsilon}$ where this is defined by Fenichel's theory. They are potentially distant only by an amount $\mathcal O(e^{-c/\epsilon})$ but this non-uniqueness plays no role in the following. 

As with the analysis of the fold lines in section \secref{foldLines}, the uniqueness/non-uniqueness of $C_{a,1}$ as a center manifold plays an important role here. It depends on the direction of the flow on $C_{a,1}$. 

We would like to connect the flow of the sliding vector-field to the flow on the center manifold. We therefore consider the reduced equations by inserting $w=m(r_1,z_1,\epsilon_1)$ into \eqref{mCenterManifold}. 
Upon division by $2\epsilon_1\vert \beta \vert^{-1}$ we finally obtain the reduced equations on $M_{a,1}$:
\begin{align}
 r_1' &=-r_1 G(r_1,z_1,\epsilon_1),\eqlab{centerManifoldK2ReducedEqns}\\
 z_1'&=-\left(\vert b\vert \text{sign}(\beta)z_1^2 - (c-\gamma) z_1-\vert \beta\vert \text{sign}(b)\right)+\epsilon_1H(r_1,z_1,\epsilon_1)+\mathcal O(r_1),\nonumber\\
 \epsilon_1'&=2\epsilon_1 G(r_1,z_1,\epsilon_1),\nonumber
\end{align}
where
\begin{align*}
G(r_1,z_1,\epsilon_1) =  -\vert b\vert \text{sign}(\beta)z_1+c+\mathcal O(r_1+\epsilon_1).
\end{align*}
Note that $H(0,\chi_{\pm},\epsilon_1)=0$ by \eqref{mCenterManifoldProperty}.
The points
\begin{align}
 p_{a,1,\pm}\equiv (r_1,z_1,\epsilon_1)= (0,\chi_{\pm},0),\eqlab{pa2pm}
\end{align}
are equilibria of these equations provided $\chi_{\pm}<0$. The value of $w$ at $p_{a,1,\pm}$ is $-\vert \beta \vert^{-1} {\vert b\vert} \chi_{\pm}$. In the following we shall also think of $p_{a,1,\pm}$ as $(r_1,w,z_1,\epsilon_1)=(0,-\vert \beta \vert^{-1} {\vert b\vert} \chi_{\pm},\chi_{\pm},0)\in L_{a,1}$. The lines $l_{a,1,\pm}$ \eqref{la2pm}
then emanate from $p_{a,1,\pm}$. The eigenvalues of the Jacobian matrix $$\partial_{(r_1,z_1,\epsilon_1)} \begin{pmatrix}
                                                                                                          r_1' \\
                                                                                                          z_1'\\
                                                                                                          \epsilon_1'
                                                                                                         \end{pmatrix}(p_{a,1,\pm}),$$
 are
 \begin{align}
  \mu_{1,\pm}=\lambda_{\pm},\,\mu_{2,\pm}=\mp \sqrt{(c-\gamma)^2+4b\beta},\mu_{3,\pm}=-2\lambda_{\pm}.\eqlab{EigenvaluesChartK2}
 \end{align}
  The eigenspace associated with the first eigenvalue $\mu_{1,\pm}$ is spanned by a vector $v_{1,\pm}$ contained within the $(r_1,z_1)$-plane, whereas the eigenspaces associated with the remaining two eigenvalues $\mu_{2,\pm}$ and $\mu_{3,\pm}$ are spanned by the vectors $$v_{2}=(0,1,0),$$ and $$v_3=(0,0,1),$$ respectively. The flow in the $\epsilon_1$-direction on $l_{a,1,\pm}$ (and $C_{a,1}$) is determined by the sign of $\mu_{3,\pm}$. The lines correspond to primary singular canards if $\mu_{3,\pm}>0$ and faux singular canards if $\mu_{3,\pm}<0$ (compare with \propref{EqSlidingVectorField}).

Tables 1 and 2 summarize the existence of $p_{a,1,\pm}$ as hyperbolic equilibria for the reduced equations \eqref{centerManifoldK2ReducedEqns} and the properties of their stable and unstable manifolds.


\begin{table}
\centering
\begin{tabular}{|c|c|c|c|c|c|c||}
 \hline 
 \hline
Type & Condition & Equilibrium & $(\mu_1,\mu_2,\mu_3)$ & Stable manifold & Unstable manifold \\
\hline
\hline
V & None & $p_{a,1,-}$ & $(-,+,+)$ & Primary singular canard & $C_{a,1}\supset l_{a,1,-}$ \\
\hline
I & None & $p_{a,1,+}$ & $(+,-,-)$ &  $C_{a,1}\supset l_{a,1,+}$ & Faux singular canard\\
\hline
VI & $b>0$ & $p_{a,1,+}$ & $(+,-,-)$ & $C_{a,1}\supset l_{a,1,+}$ & Singular faux canard\\
 & & $p_{a,1,-}$ & $(-,+,+)$ & Primary singular canard & $C_{a,1}\supset l_{a,1,-}$\\
 \hline
 \hline
\end{tabular}
\caption{The hyperbolic equilibria $p_{a,1,\pm}$ for case (S). Column 1: type of two-fold singularity [visible (V), invisible (I) and visible-invisible (VI)]. Column 2: conditions on parameters. Column 3: which equilibria exist. Column 4: the sign of the eigenvalues $\mu_{1\pm},\mu_{2\pm},\mu_{3\pm}$. Column 5: the stable manifolds of $p_{a,1,\pm}$ as a hyperbolic equilibrium of the reduced equations \eqref{centerManifoldK2ReducedEqns}. Column 6: as column 5, for the unstable manifolds. In visible-invisible case $p_{a,1,\pm}$ co-exist if and only if $b>0$. }\tablab{tblS}
\end{table}

\begin{table}
\centering
\begin{tabular}{|c|c|c|c|c|c|c|||}
 \hline 
 \hline
Type & Condition & Equilibrium & $(\mu_1,\mu_2,\mu_3)$ & Stable manifold & Unstable manifold \\
\hline
\hline
V & None & $p_{a,1,-}$ & $(-,+,+)$ & Strong singular canard & $C_{a,1}\supset l_{a,1,-}$ \\
\hline
I & None & $p_{a,1,+}$ & $(-,-,+)$ &  $S_{a,1}$: sector of & $l_{a,1,+}$\\
  &      &             &           & weak singular canards &  \\
\hline
VI & $b>0$ & $p_{a,1,+}$ & $(-,-,+)$ &  $S_{a,1}$: sector of & $l_{a,1,+}$\\
  &      &             &           & weak singular canards &  \\
    & & $p_{a,1,-}$ & $(-,+,+)$ & Strong singular canard & $C_{a,1}\supset l_{a,1,-}$ \\
 \hline
 \hline
\end{tabular}
\caption{The hyperbolic equilibria $p_{a,1,\pm}$ for case (N). Column headings as in \tabref{tblS}. All canards in case (N) are primary canards.}\tablab{tblN}
\end{table}

We conclude this section with a proposition (compare with \cite[Proposition 4.1]{szmolyan_canards_2001}) that summarizes the findings in chart $\kappa_1$:
\begin{proposition}\proplab{centerManifoldK2}
Within a small neighborhood of the invariant line $L_{a,1}$, the following statements hold true: There exists a $3D$ attracting center manifold $M_{a,1}$ of $L_{a,1}$ for Eqs. \eqref{eqnChartK2} that takes the following form:
\begin{eqnarray}
 w=-\vert \beta \vert^{-1} {\vert b\vert} z_1+\mathcal O(r_1+\epsilon_1).\eqlab{Ma2Graph}
\end{eqnarray}
$M_{a,1}$ is foliated by $M_{a,1}(\epsilon)$ corresponding to $M_{a,1}\cap \{\epsilon=r_1^2\epsilon_1=\text{const}\}.$ 
 The center manifold includes $S_{a,1}$ contained within $\epsilon_1=0$ as a manifold of equilibria and $C_{a,1}$ contained within $r_1=0$ as a center sub-manifold. The former corresponds to the critical manifold $S_a$.  The latter contains the invariant lines $l_{a,1,\pm}$ \eqref{la2pm} if $\chi_{\pm}<0$. The lines $l_{a,1,\pm}$ emanate from $p_{a,1,\pm}$ \eqref{pa2pm} which appear as hyperbolic equilibria of the reduced, de-singularized equations \eqref{centerManifoldK2ReducedEqns}. The manifold $C_{a,1}$ is: 
 \begin{itemize}
  \item Case (S): unique near $p_{a,1,-}$ and non-unique near $p_{2,a,\pm}$ (when these exist);
    \item Case (N): unique near $p_{a,1,\pm}$ (when these exist);
    \end{itemize}
as a center manifold of $L_{a,1}$ within $r_1=0$.


%
\end{proposition}

It follows from this proposition that every primary (faux) singular canard lies within the unique (non-unique) part of the center manifold $C_{a,1}$ (see \figref{centerManifoldCa2}). Note the similarity between \figref{centerManifoldCa2FoldLine} for the fold and \figref{centerManifoldCa2} for the two-fold. We illustrate the dynamics within $M_{a,1}$ \eqref{Ma2Graph} for case (S) in \figref{centerManifoldMa2S} and for case (N) in \figref{centerManifoldMa2N}. Here $S_{a,1}$ and $C_{a,1}$ are identified as invariant sub-manifolds. In particular the motion on $S_{a,1}$ is compared to the analysis of the piecewise smooth systems. 
\begin{remark}\remlab{noname}
The singular canards described in \propref{singularCanards} for the piecewise smooth system are identified within $S_{a,1}$ as trajectories asymptotic to $p_{a,1,\pm}$, since (a) $S_{a,1}=M_{a,1}\cap \{\epsilon_1=0\}$ is $S_a$ written in chart $\kappa_1$ and (b) the slow flow on $S_a$ coincides with the sliding vector field (cf. \thmref{criticalManifold}). See \figref{centerManifoldMa2S} and \figref{centerManifoldMa2N} and recall also section \secref{connection}. 
 An important consequence of \propref{centerManifoldK2} is the fact that we can continue singular canards (even a whole sector of singular canards) within $S_{a,1}$ into chart $\kappa_2$ using a single trajectory $l_{a,1,\pm}$. This is the reason why we obtain only one canard for $\epsilon$ sufficiently small from a whole sector of weak singular canards in \thmref{mainRes}. 
\end{remark}

\begin{figure}[h!] 
\begin{center}
\subfigure[]{\includegraphics[width=.45\textwidth]{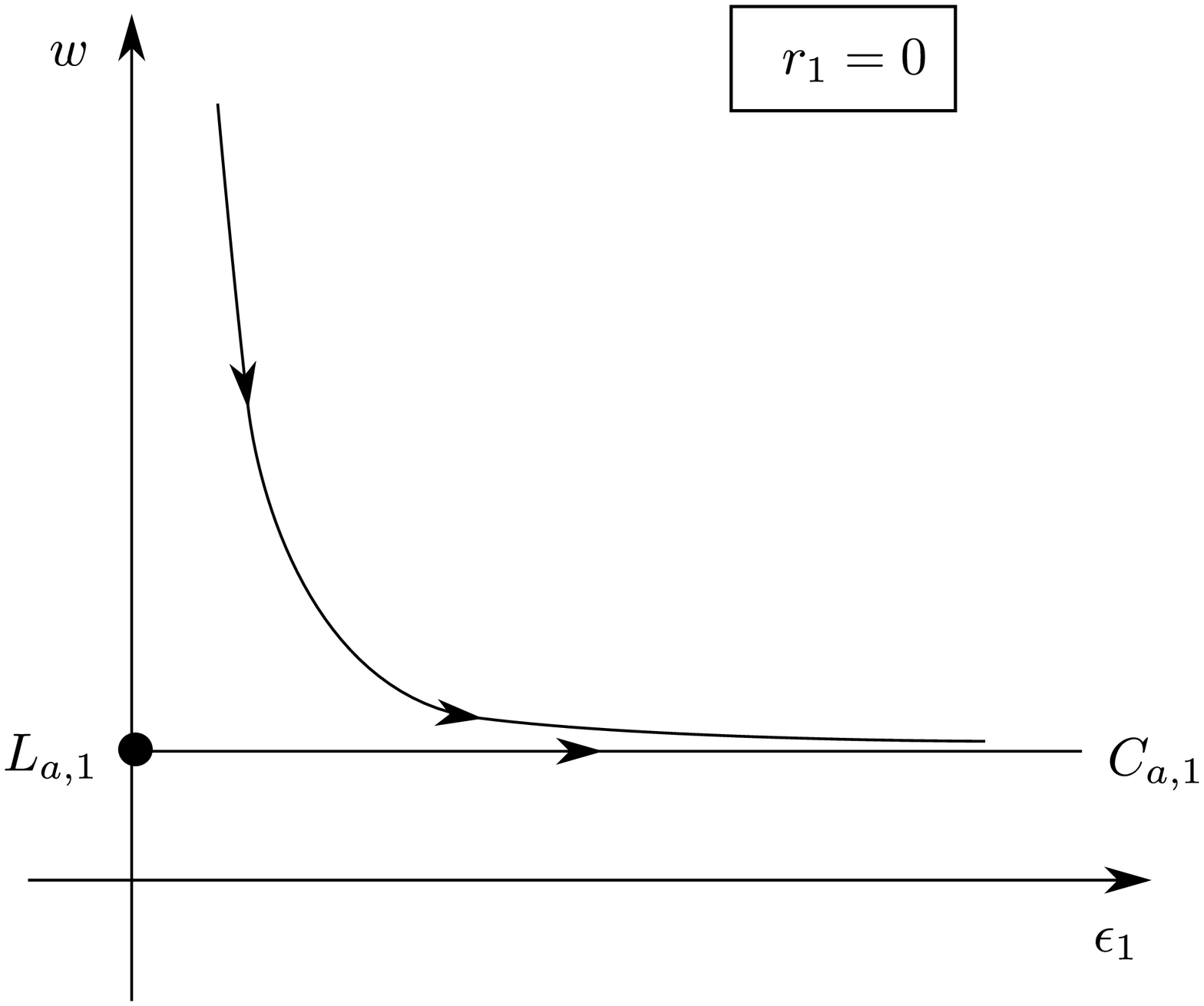}}
\subfigure[]{\includegraphics[width=.45\textwidth]{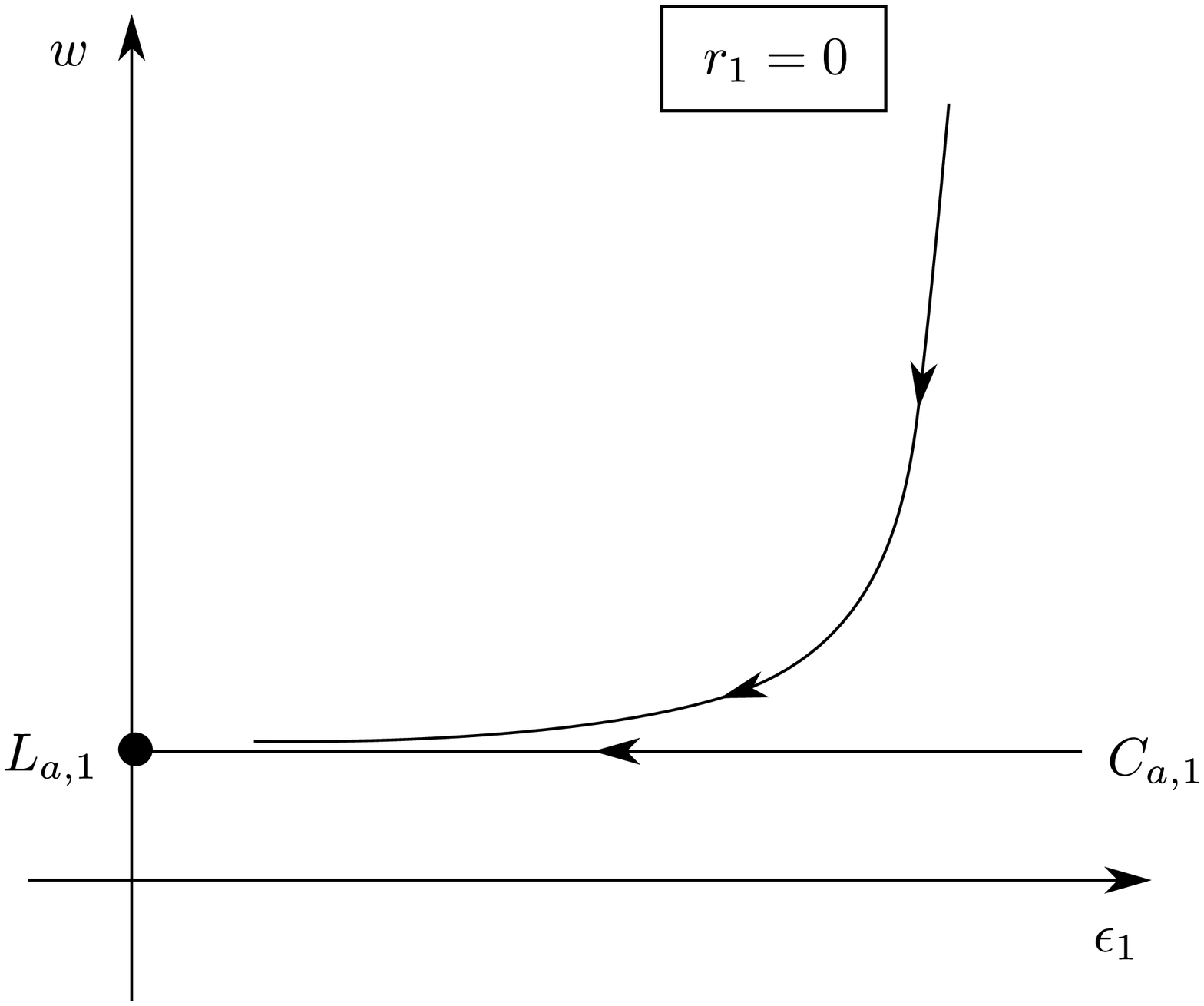}}
\end{center}
\caption{Illustration of the attracting center manifold $C_{a,1}$ within $r_1=0$ using a projection onto the $(\epsilon_1,w)$-plane. (a) The motion on $C_{a,1}$ is away from $L_{a,1}$ (which is a point in this projection) and $C_{a,1}$ is unique. This situation corresponds to a primary singular canard. (b) The motion on $C_{a,1}$ is towards $L_{a,1}$ and $C_{a,1}$ is non-unique. This situation corresponds to a faux singular canard. }
\figlab{centerManifoldCa2}
\end{figure}
\begin{figure}[p!] 
\begin{center}
\subfigure[Visible]{\includegraphics[width=.54\textwidth]{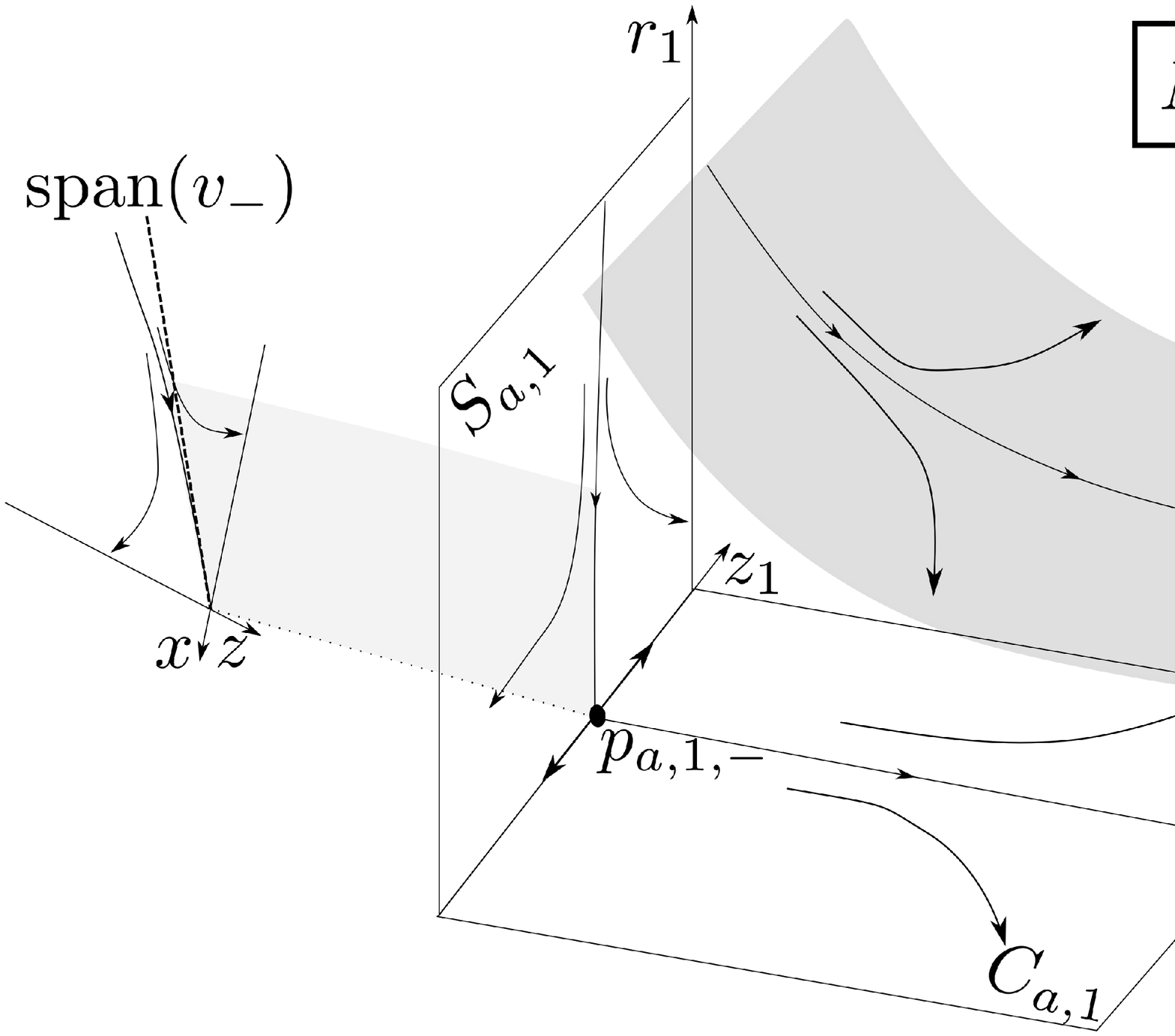}}
\subfigure[Invisible]{\includegraphics[width=.54\textwidth]{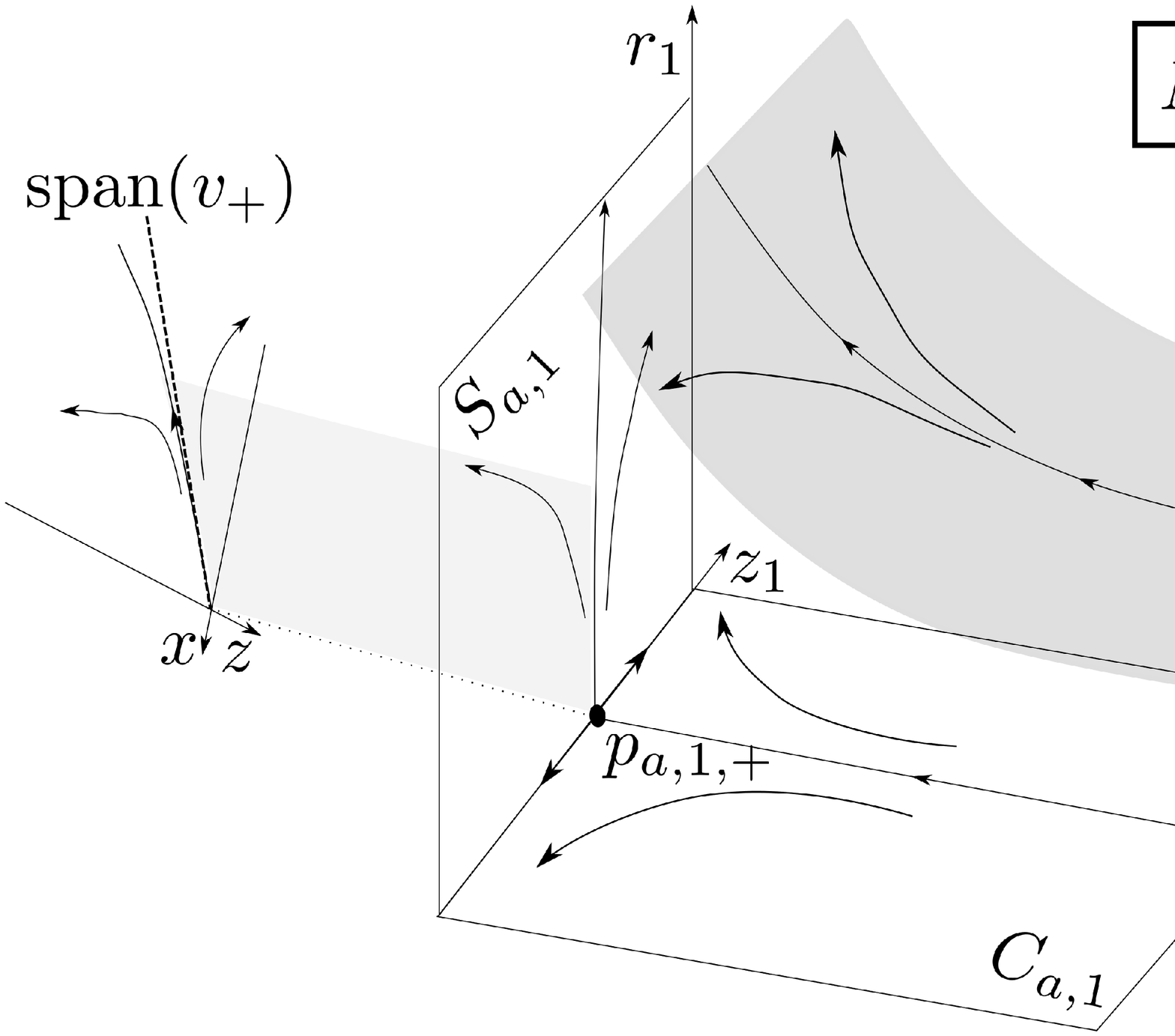}}
\subfigure[Visible-invisible]{\includegraphics[width=.54\textwidth]{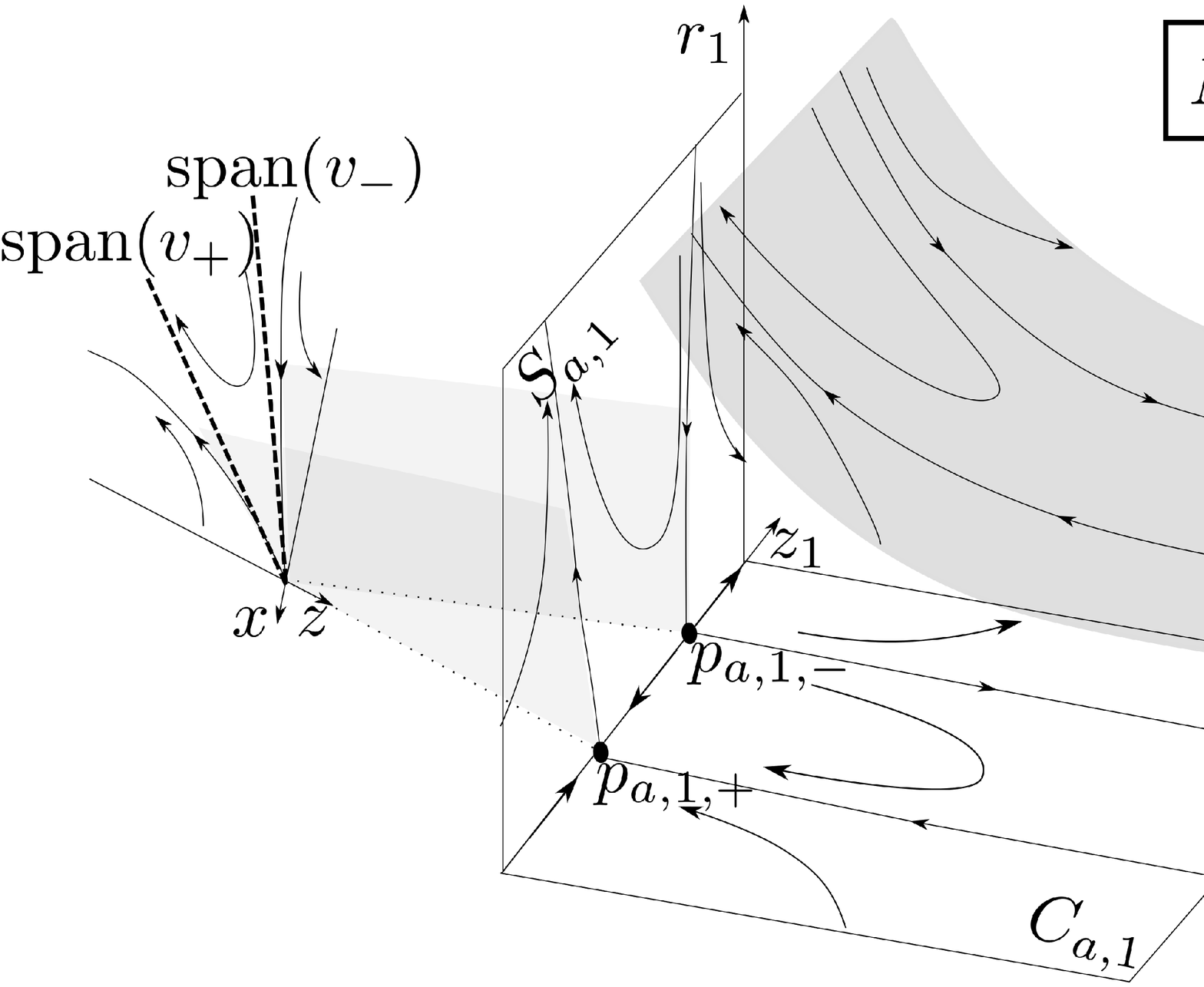}}
\end{center}
\caption{Case (S): Illustrations of the de-singularized dynamics within the attracting center manifold $M_{a,1}$ and the corresponding sliding motions within the different piecewise smooth systems (shown on the left of each illustration). The curved shaded regions within $(r_1,z_1,\epsilon_1)$-space illustrate the invariant sub-manifolds with $\epsilon=\text{const}.$}
\figlab{centerManifoldMa2S}
\end{figure}
\begin{figure}[p!] 
\begin{center}
\subfigure[Visible]{\includegraphics[width=.54\textwidth]{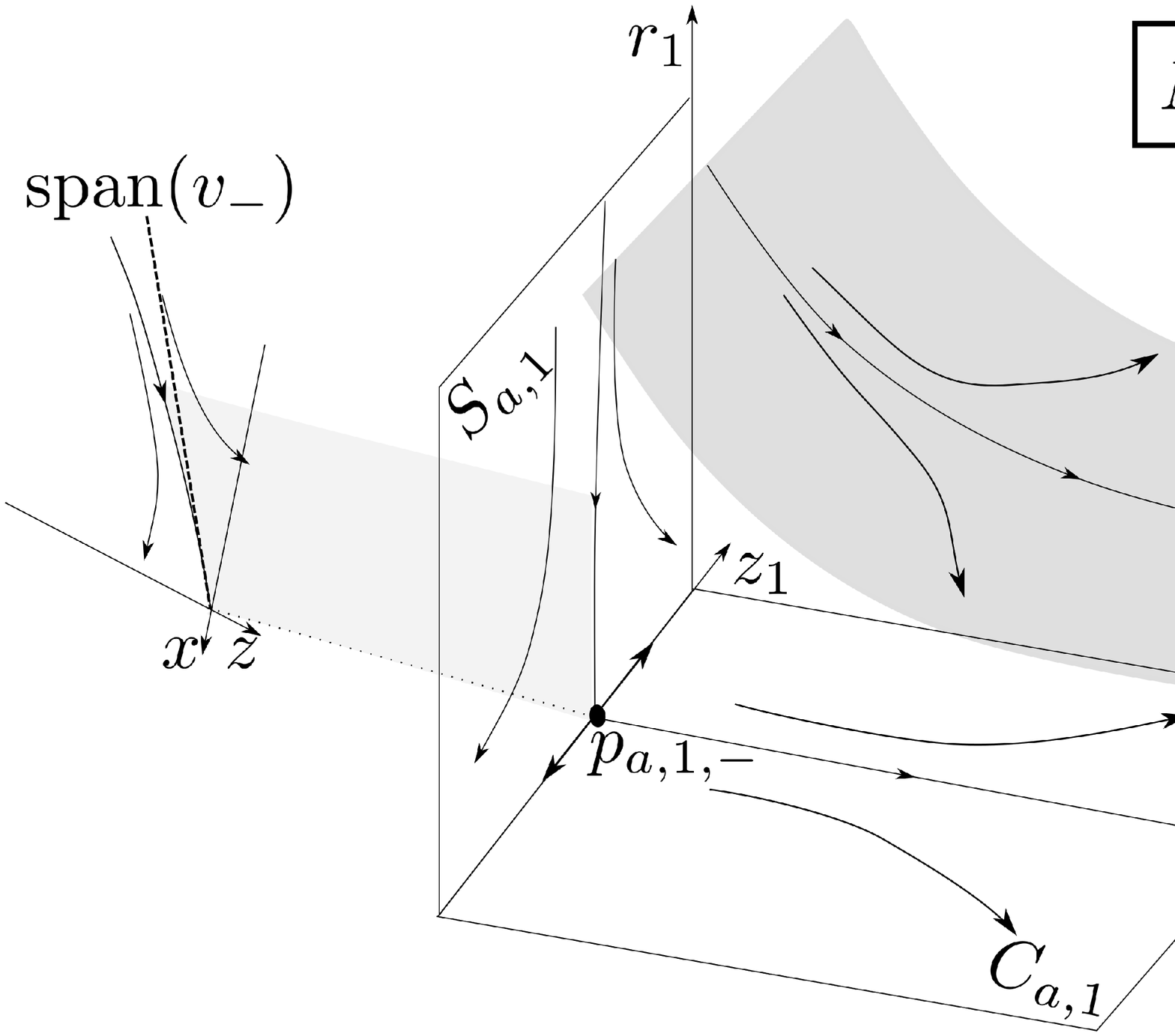}}
\subfigure[Invisible]{\includegraphics[width=.54\textwidth]{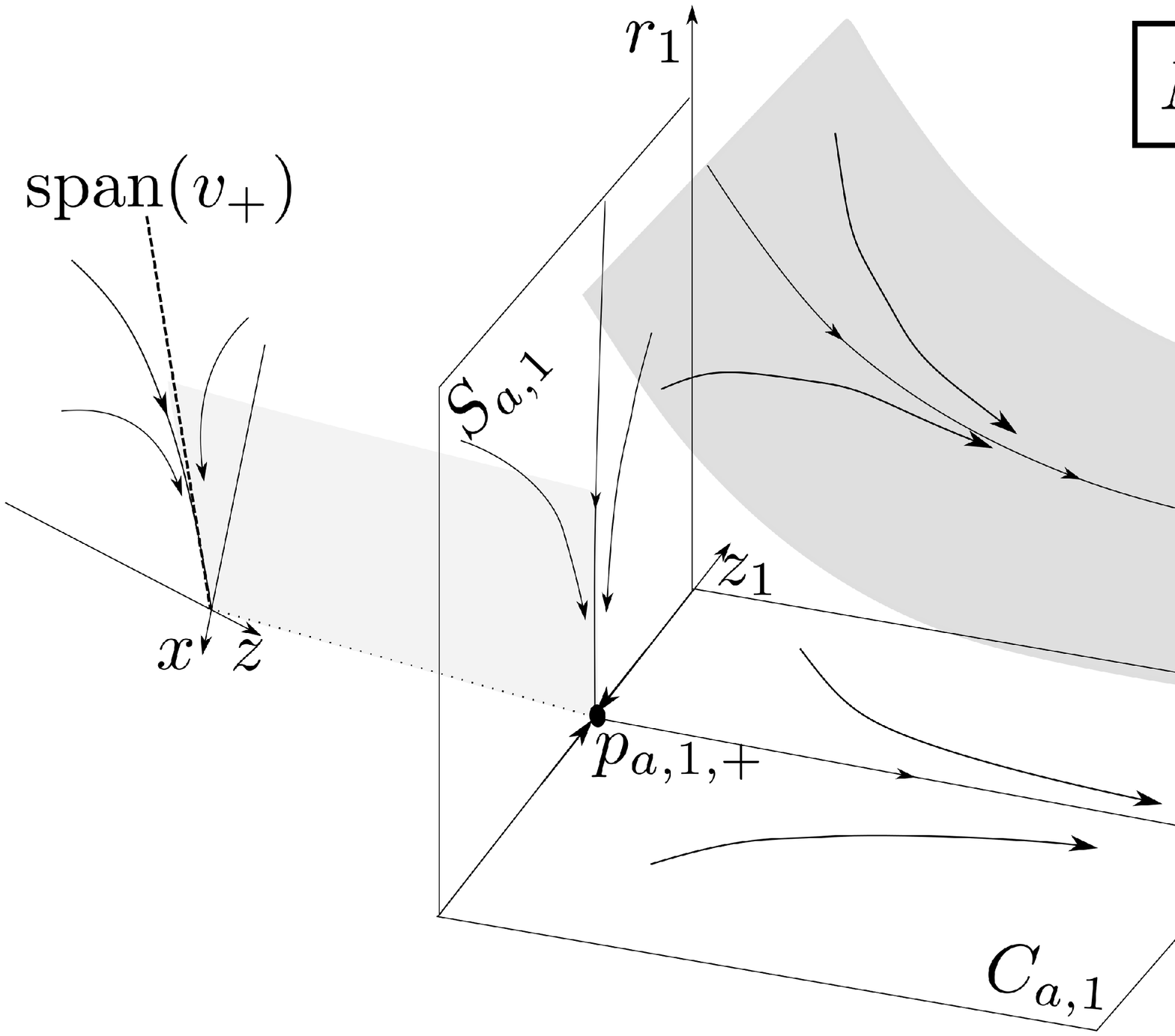}}
\subfigure[Visible-invisible]{\includegraphics[width=.54\textwidth]{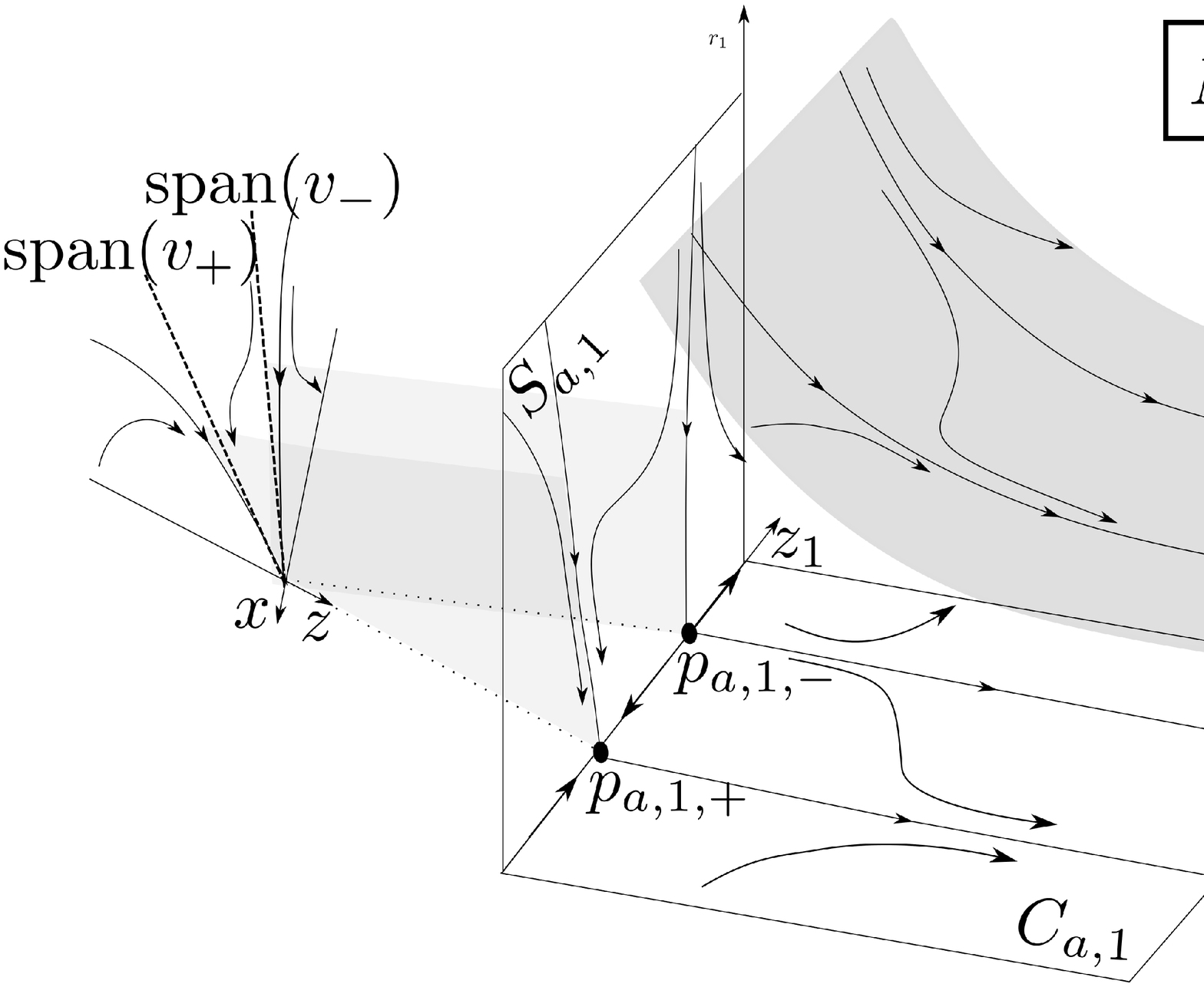}}
\end{center}
\caption{As \figref{centerManifoldMa2S} for Case (N).}
\figlab{centerManifoldMa2N}
\end{figure}

\subsection{Dynamics in chart $\kappa_2$}
In this chart we obtain the following equations
\begin{eqnarray}
\dot x_2 &=&\vert \beta\vert^{-1} c+\text{sign}(\beta)w+O(r_2),\eqlab{ss}\\
\dot w &=&f(w)\left(-\vert b\vert {z}_2 +\vert \beta\vert x_2w +O(r_2)\right),\nonumber\\
 \dot z_2&=&\text{sign}(b)+\vert b\vert^{-1} \gamma w+O(r_2),\nonumber\\
 \dot r_2&=&0. \nonumber
\end{eqnarray}
where we have re-scaled time by $r_2$. Then we have the following:
\begin{lemma}\label{gamma1lemma}
 Suppose $\chi_{\pm}<0$. Then there exists an invariant line :
\begin{align}
 l_{2,\pm}: \,(x_2,w,z_2,r_2)=(x_2,-\vert \beta \vert^{-1} {\vert b\vert} \chi_{\pm},-\chi_{\pm}x_2,0),\,x_2\in \mathbb R,\eqlab{l2pmEq}
 \end{align}
which coincides with the image of $l_{a,1,\pm}$ under $\kappa_{21}$ where this is defined. 
The motion on $l_{2,\pm}$ is determined by 
\begin{align*}
 \dot x_2&=-\vert \beta \vert^{-1} \lambda_{\pm},\\
 \dot w&=0.
\end{align*}
\end{lemma}
\begin{proof}
 For the motion on $l_{2,\pm}$ we use \eqsref{eqnchipm}{ss} to obtain
 \begin{align*}
  \dot x_2 &=-\vert \beta\vert^{-1} (-c+\text{sign}(\beta)\vert b\vert \chi_{\pm})=-\vert \beta\vert^{-1} \lambda_{\pm}.
 \end{align*}
\end{proof}

The invariant line $l_{2,\pm}$ intersects $x_2=0$ in the point 
\begin{align*}
q_{2,\pm} = (0,-\vert \beta \vert^{-1} {\vert b\vert} \chi_{\pm},0,0).
\end{align*}
The role of the invariant line $l_{2,\pm}$ is to carry the center manifold $M_{a,2}=\kappa_{21}(M_{a,1})$ up until $q_{2,\pm}$. In particular, $l_{2,\pm}$ carries the sub-manifold $C_{a,2}=\kappa_{21}(C_{a,1})$ up until $q_{2,\pm}$. 

The truncation of our equations has a time-reversible symmetry: $$(x,w,z,t)\mapsto (-x,w,-z,-t).$$ So we can deduce the existence of a repelling center manifold $\overline M_{r}(\epsilon)$ that is an extension of $S_{r,\epsilon}$ close to the singularity. In particular, the invariant line $l_{2,\pm}$ will be forward asymptotic to the reflection of $p_{a,1,\pm}$. The reflection of $p_{a,1,\pm}$ will be contained in a reflection of $C_{a,1}$ which we denote by $C_{r,2}$ in chart $\kappa_2$. The center manifold $C_{r,2}$ is obtained by applying a symmetry and it is therefore unique if and only if $C_{a,2}$ is unique. 

In the next section we will investigate the transverse intersection of the tangent spaces $T_{q_{2,\pm}} C_{a,2}$ and $T_{q_{2,\pm}} C_{r,2}$. We will apply regular perturbation theory in chart $\kappa_2$ to conclude that $M_{a,2}(\epsilon)$ and $M_{r,2}(\epsilon)$ are $r_2$-close ($r_2=\sqrt{\epsilon}$) to $C_{a,2}$ and $C_{r,2}$ respectively. The transverse intersection of the tangent spaces $T_{q_{2,\pm}} C_{a,2}$ and $T_{q_{2,\pm}} C_{r,2}$ will therefore imply the transverse intersection of $\overline M_{a}(\epsilon)$ with $\overline M_{r}(\epsilon)$ (or simply $S_{a,\epsilon}$ with $S_{r,\epsilon}$) for sufficiently small $\epsilon$, and hence provide the existence of canards and so prove our main result, \thmref{mainRes}.

\subsection{The persistence of canards}
First we consider the persistence of faux singular canards. 
\begin{proposition}\proplab{existenceFauxCanards}
There exists an $\epsilon_0$ so that, for every $\epsilon\le \epsilon_0$, every  faux singular canard implies the existence of a two parameter family of faux canard solutions.
\end{proposition}
\begin{proof}
 The proof is identical to that of \cite[Proposition 4.3]{szmolyan_canards_2001}. It is due to the fact that the invariant line $l_{2,\pm}$ is contained within the non-unique parts of $C_{a,2}$ and $C_{r,2}$. 
\end{proof}

Next, we focus on the persistence of primary singular canards. From \propref{EqSlidingVectorField} these correspond to equilibria of \eqref{slidingVectorFieldNewTime} with $\lambda_{\pm}<0$. From \propref{centerManifoldK2} we can conclude that the invariant line $l_{2,\pm}$ is contained within the unique parts of $C_{a,2}$ and $C_{r,2}$. To analyse the transverse intersection of the tangent spaces $T_{q_{2,\pm}}C_{a,2}$ and $T_{q_{2,\pm}} C_{r,2}$ we have to analyse the variational equations about the invariant line $l_{2,\pm}$. To do this we apply the transformation $(x_2,w,z_2)\mapsto (\tilde x_2,\tilde w,\tilde z_2)=(x_2,w+\vert \beta\vert^{-1} \vert b\vert \chi_{\pm},z_2+\chi_{\pm}x_2)$. This moves $l_{2,\pm}$ to the $\tilde x_2$-axis. We obtain the following equations
\begin{align*}
 \dot{\tilde x}_2 &=-\vert \beta \vert^{-1} \lambda_{\pm} +\text{sign}(\beta)\tilde w,\\
 \dot{\tilde w}&=f(w)(-\vert b\vert \tilde z_2+\vert \beta \vert \tilde x_2 \tilde w),\\
 \dot{\tilde z}_2&=-\vert b\vert^{-1} \lambda_{\mp} \tilde w,
\end{align*}
with $f(w)=f(\tilde w-\vert \beta\vert^{-1} \vert b\vert \chi_{\pm})$. To obtain the expression for $\dot{\tilde z}_2$ we have used \eqsref{lambdapm}{chipm} to conclude that
\begin{align*}
 -\chi_{\pm}\text{sign}(\beta)\vert b\vert - \gamma = \lambda_{\mp}.
\end{align*}
In these variables, the dynamics on $l_{2,\pm}$ is given by 
\begin{align*}
 (\tilde x_2,\tilde w,\tilde z_2) = (-\vert \beta \vert^{-1} \lambda_{\pm}t,0,0). 
\end{align*}
We then take the variations about this solution to obtain:
 \begin{align}
  \frac{d u}{dx_2} &= -{\lambda_{\pm}}^{-1}{f(w)\vert \beta\vert} \left(-\vert b\vert v+\vert \beta\vert x_2 u\right),\eqlab{uvVariationalEqs}\\
  \frac{d v}{dx_2}&= \vert \beta \vert \vert b\vert^{-1} \lambda_{\pm}^{-1} \lambda_{\mp} u,\nonumber
 \end{align}
having replaced time by $x_2$. Here $(u,v)=(\delta \tilde w,\delta \tilde z_2) =(\delta w,\delta z_2)$. The time-reversible symmetry then becomes an invariance of these equations with respect to $(u,v,x_2)\mapsto (u,-v,-x_2)$. 

We now follow the reasoning in \cite{szmolyan_canards_2001}. In particular, we will make use of the following lemma from \cite[Proposition 4.2]{szmolyan_canards_2001}:
\begin{lemma}\lemmalab{transversality}
The tangent spaces $T_{q_{2,\pm}} C_{a,2}$ and $T_{q_{2,\pm}} C_{r,2}$ are transverse if and only if there exists no non-zero solution of \eqref{uvVariationalEqs} which has algebraic growth for $x_2\rightarrow \pm \infty$.
\end{lemma}
\begin{proof}
 The main part of the proof in \cite{szmolyan_canards_2001} just restates the claim of the theorem. We therefore include our own proof here. There exists a $c$ sufficiently large, so that $\kappa_{21}(l_{2,\pm}(x_2))\subset C_{a,2}$ for $x_2<-c$ and $\kappa_{21}(l_{2,\pm}(x_2))\subset C_{r,2}$ for $x_2>c$ with $C_{a,2}$ and $C_{r,2}$ the center manifolds, described in chart $\kappa_2$, that are unique in the case of primary singular canards. Variations within $T_{q_{2,\pm}}C_{a,2}$ and $T_{q_{2,\pm}}C_{r,2}$ will therefore be characterized by algebraic growth properties in the past ($x_2\rightarrow -\infty$) and in the future ($x_2\rightarrow \infty$), respectively. Variations normal to $T_{q_{2,\pm}} C_{a,2}$ and $T_{q_{2,\pm}}C_{r,2}$ are, on the other hand, characterized by exponential growth in the past ($x_2\rightarrow -\infty$) and in the future ($x_2\rightarrow \infty$), respectively. The statement of the Lemma therefore follows. 
\end{proof}

We will write \eqref{uvVariationalEqs} as a Weber equation. To do this we first write it as a second order ODE:
\begin{align*}
 \frac{d^2 v}{dx_2^2} - \nu_{\pm} x_2 \frac{dv}{dx_2}+\nu_{\pm}\xi_{\pm} v=0.
\end{align*}
where
\begin{align}
\nu_{\pm} &=-\lambda_{\pm}^{-1}f(w)\vert \beta\vert^2,\nonumber\\
\xi_{\pm} &=  \lambda_{\pm}^{-1} \lambda_{\mp}.\eqlab{xipm}
\end{align}
Note that $\nu_{\pm}>0$ since we have assumed $\lambda_{\pm}<0$. 
We then write $x_2=\nu_{\pm}^{-1/2} \overline x_2$ and obtain the Weber equation:
\begin{align}
 \frac{d^2 v}{d\overline{x}_2^2} - \overline x_2 \frac{dv}{d\overline{x}_2}+\xi_{\pm} v=0.\eqlab{weberEqns}
\end{align}


\begin{remark}\remlab{comparisonMu}
 {The study of the persistence of weak canards in folded nodes in smooth slow-fast systems leads to the consideration of algebraic solutions of a similar Weber equation \cite[Eq. (2.24)]{wechselberger_existence_2005}. However, in our work the coefficient of $v$ in \eqref{weberEqns} is the ratio of the eigenvalues $\xi_{\pm}$ whereas in \cite[Eq. (2.24)]{wechselberger_existence_2005} this coefficient is $\xi_{\pm}-1$.} 
\end{remark}

\begin{lemma}
 If $\xi_{\pm}\in \mathbb N$ then the Hermite polynomial $H_{\xi_{\pm}}(\overline{x}_2/\sqrt{2})$ is a polynomial solution of \eqref{weberEqns} \cite{AbramowitzStegun1964}. This solution has $\xi_{\pm}$ zeros. If $n<\xi_{\pm}<n+1$ with $n\in \mathbb N$ then there exists two linearly independent solutions $v_1=v_1(\overline x_2)$ and $v_2=v_1(-\overline x_2)$ which grow exponentially in the future $\overline x_2\rightarrow \infty$ and in the past $\overline x_2\rightarrow -\infty$ respectively. Furthermore, $v_1$ and $v_2$ possess $n$ zeros.
\end{lemma}
\begin{remark}
 The zeros of $v_1$ correspond to the number of twists (cf. \cite[Lemma 4.4]{szmolyan_canards_2001}) of $T_{l_{2,\pm}}C_{2,a}$ along $l_{2,\pm}$. By considering the rotation angle $\theta$ defined by $v/v'=\tan \theta$, it can be seen \cite{szmolyan_canards_2001} that one twist corresponds to one rotation of $180^\circ$.  In slow-fast theory this is a mechanism for generating small oscillations in mixed-mode oscillations \cite{brons-krupa-wechselberger2006:mixed-mode-oscil}.
\end{remark}

Following \lemmaref{transversality} we therefore conclude the following:
\begin{proposition}\proplab{tangencies}
$T_{q_{2,\pm}} C_{a,2}$ intersects $T_{q_{2,\pm}}C_{r,2}$ transversally if and only if $\xi_{\pm}=\lambda_{\pm}^{-1}\lambda_{\mp}\notin \mathbb N$. 
\end{proposition}

From this follows that $M_{a,2}(\epsilon)$ and $M_{r,2}(\epsilon)$ intersect transversally in a $\mathcal O(\sqrt{\epsilon})$-neighborhood of $q_{2,\pm}$ if $\xi_{\pm}\notin \mathbb N$. Here we have used the fact that $r_2=\sqrt{\epsilon}$. As in \cite{szmolyan_canards_2001} the condition $\xi_{\pm}\in \mathbb N$ can only be realised in case (N) since $\xi_{\pm}<0$ in case (S). In case (N) where $\lambda_-<\lambda_+<0$, we have $\xi_{\pm}>0$ and $\xi_+<1$ and $\xi_->1$ 
(\propref{EqSlidingVectorField}). Hence only $\xi_-$ can be a natural number. This corresponds to a weak singular canard.

\subsection{Conclusions on the analysis of the two-fold}
The identification of $M_{a}(\epsilon)$ as the continuation of the slow manifold $S_{a,\epsilon}$ close to $\tilde p$ \eqref{tildep2} implies the existence of a canard close to a singular one. Strong singular canards always persist and we can trace their perturbed version backwards on $\overline M_a(\epsilon)$ using the $1D$ stable manifold of $p_{a,1,\pm}$ of \eqref{centerManifoldK2ReducedEqns} as a guide. In this case the stable manifold of $p_{a,1,\pm}$ (see \tabref{tblS} and \tabref{tblN}) coincides with the singular canard. The perturbed singular canard can similarly be traced forwards in $\overline M_r(\epsilon)$. This gives the first statement of our main theorem,  \thmref{mainRes}. A weak singular canard persists whenever $\xi_{-}\notin \mathbb N$. It is, however, as in \cite{szmolyan_canards_2001}, not possible to track these perturbed weak canards onto $S_{a,\epsilon}$ and $S_{r,\epsilon}$. This is to be expected since weak singular canards are non-unique. 

We believe that the analysis in \cite{wechselberger_existence_2005} on the existence of secondary canards near resonances should also apply to our system. The numerics performed in the following section support this claim.


%

\section{Numerics}\seclab{numerics}
In this section we present results from some numerical experiments 
to investigate the bifurcations of primary canards in the limit $\epsilon=0$. The blow up analysis allows us to consider the limit $\epsilon=0$ by continuing $S_{a}$ and $S_r$ using $C_{a,1}$ and $C_{r,1}$ (see \remref{noname}). To perform computations we have to fix a choice of regularization $\phi(y)$. We have based our computations on the following $C^1$-function
\begin{eqnarray*}
 \phi(y) = -\frac12 y^3 + \frac32 y\quad \text{for}\quad y\in (-1,1),\\
 \end{eqnarray*}
 with $\phi(y)=\pm 1$ for $y\gtrless \pm 1$. The degree of the smoothness appears to play little role in the bifurcations of primary canards in the limit $\epsilon=0$. 
 
 We wish to focus on the appearance of secondary canards, which occur close to resonances of a weak singular canard. Hence we consider case (N) and choose a visible-invisible two-fold singularity with the following parameters:
 \begin{align}
  c-\gamma &=\frac52,\quad  c+\gamma=\frac{3(\xi_-+1)}{2(\xi_--1)},\quad  b=1,\quad
  \beta =-1.\eqlab{parameters}
 \end{align}
This gives:
\begin{align}
 \lambda_+ &= -\frac{3}{2(\xi_--1)},\quad\lambda_- = \xi_-\lambda_+,\quad \chi_+=-2,\quad \chi_-=-\frac12.\eqlab{lambdaChiValues}
\end{align}
corresponding to \circled{5} in \figref{ParameterEqType1}. See also \tabref{tblN}. We take initial conditions on  $C_{a,2}=\kappa_{21}(C_{a,1})$ where $r_2=\sqrt{\epsilon}=0$, by setting $x_2=-1/\sqrt{\delta},\,z_2=z_1/\sqrt{\delta},\,w=-\vert \beta \vert^{-1}\vert b\vert z_1+\mathcal O(\delta)$, with fixed $\delta=0.01$. \figref{secondaryCanards} illustrates the intersections of $C_{a,2}\cap \{z_2=0\}$ ($-$) and $C_{r,2}\cap \{z_2=0\}$ ($--$) (compare with Fig. 13 of \cite{wechselberger_existence_2005}), in agreement with \propref{tangencies}. Here we also find bifurcations of \text{secondary canards} for odd $\xi_-$ from the weak canard intersecting $z_2=0$ at 
\begin{align}
 (x_2,y)&=(0,\phi^{-1}(-1/3))=(0,\sqrt{3}\sin(\arctan (2\sqrt{2})/3)-\cos(\arctan (2\sqrt{2})/3))\nonumber\\
 &\approx (0,-0.2261).                                                                                       \eqlab{weakCanardValue}                                                                                                                                                                                                                                                                                                                                                                                                                                                                                                                                                                                                                           \end{align}
 To obtain this expression we have used that $w=-\vert \beta \vert^{-1} \vert b \vert \chi_+=\frac12$ and inverted \eqref{wTransformation} for $y$.
  The secondary canards originate from the equilibrium $p_{a,-,1}$ within $L_{a,1}$ or from the fold line $\tilde l^-$ (see \figref{centerManifoldMa2N} (c)) and they are characterized by their rotational properties about the weak canards. \figref{secondaryCanardsZoom} shows a zoom of the diagrams in \figref{secondaryCanards} near the point \eqref{weakCanardValue}. The point \eqref{weakCanardValue} appears as a black dot in all of the diagrams in \figref{secondaryCanardsZoom}. 
  
  The first bifurcation is seen to occur at $\xi_-=3$, where the curve $C_{a,1}\cap \{z_2=0\}$ is tangent to the $y_2$-axis at the point \eqref{weakCanardValue}. The secondary canard, denoted by $l_{sc,2}^{(1)}$, that appears as a result of this bifurcation is shown for $\xi_->3$ in \figref{secondaryCanards} (d) and \figref{secondaryCanardsZoom} (d) as a new transversal intersection of $C_{a,2}\cap \{z_2=0\}$ with $C_{r,2}\cap \{z_2=0\}$. It rotates $360^\circ$ around the weak canard. This is shown in \figref{C12ZEq0Mu4_5wRotation} (a), (b) and (c) using a projection onto the $(z_2,x_2)$-plane. Initially the secondary canard $l_{sc,2}^{(1)}$ goes beneath the weak canard, then goes above it and finally beneath it again. This is further illustrated in \figref{C12ZEq0Mu6.5wRotationNewDetails} where we have projected the secondary canard $l_{sc,2}^{(1)}$ for $\xi_-=6.5$ onto the plane $(2x_2-z_2,y_2)$. The reason for considering this plane is that here the weak canard appears as a point at $(0,-0.2261)$ and the single rotation of $l_{sc,2}^{(1)}$ about the weak canard is clearly visible. 
  
  At $\xi_-=5$ there is another bifucation of the weak canard, again the curve $C_{a,1}\cap \{z_2=0\}$ is tangent to the $y_2$-axis at the point \eqref{weakCanardValue} corresponding to the weak canard. This introduces another secondary canard $l_{sc,2}^{(2)}$ for $\xi_->5$. This is seen in \figref{secondaryCanards} (f) as a new transversal intersection and it is also visible in the corresponding close-up in \figref{secondaryCanardsZoom} (f). The secondary canard $l_{sc,2}^{(2)}$ rotates $720^\circ$ around the singular canard, which is illustrated in \figref{C12ZEq0Mu4_5wRotation} (c) and \figref{C12ZEq0Mu6.5wRotationNewDetails}. Hence, in accordance with the theory of the reference \cite{wechselberger_existence_2005}, we have observed that bifurcations only occur at $\xi_-=2n+1$ and that each such bifurcation give rise to a secondary canard that is visible for $\xi_->2n+1$ and which rotates $n\times 360^\circ$ around the weak canard. 

 There is also a strong canard. Using that $w=-\vert \beta \vert^{-1}\vert b\vert \chi_- = 2$ (cf. e.g. \eqref{l2pmEq}) together with \eqref{wTransformation} we find that it intersects $z_2=0$ at 
 \begin{align*}
  (x_2,y)&=(0,\phi^{-1}(1/3))=(0,-\phi^{-1}(-1/3)) \approx (0,0.2261).
 \end{align*}
 It is visible in \figref{secondaryCanards} as a transverse intersection of $C_{a,2}$ with $C_{r,2}$ at this value. In agreement with \thmref{mainRes}, the strong canard never undergoes a bifurcation.
\begin{figure}[h!] 
\begin{center}
\subfigure[$\xi_-=2$]{\includegraphics[width=.45\textwidth]{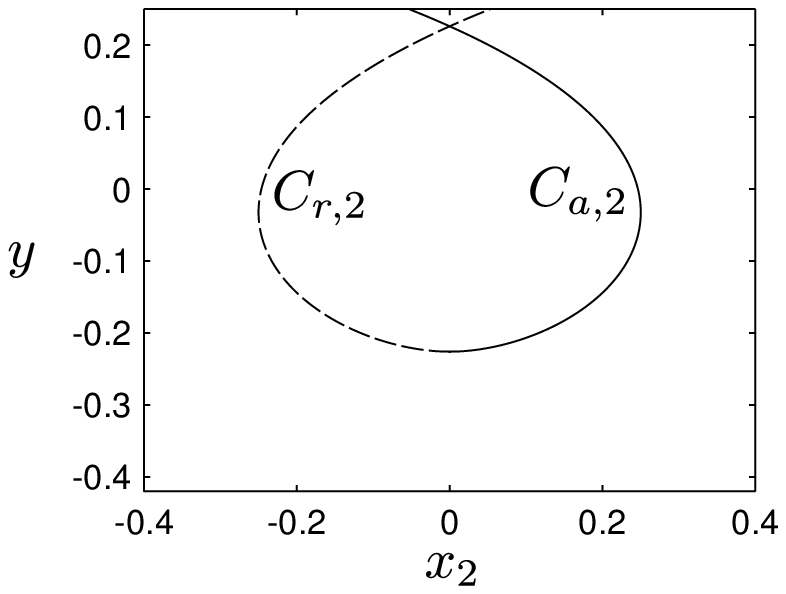}}
\subfigure[$\xi_-=2.5$]{\includegraphics[width=.45\textwidth]{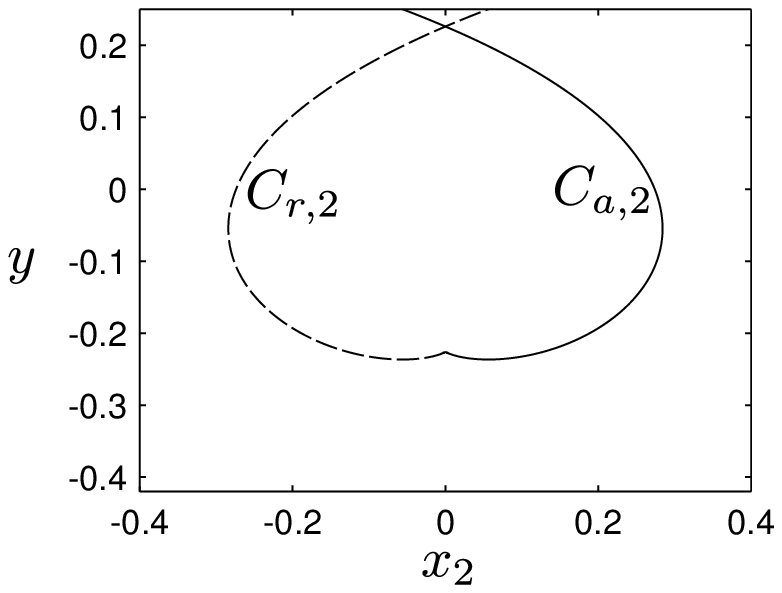}}
\subfigure[$\xi_-=3$]{\includegraphics[width=.45\textwidth]{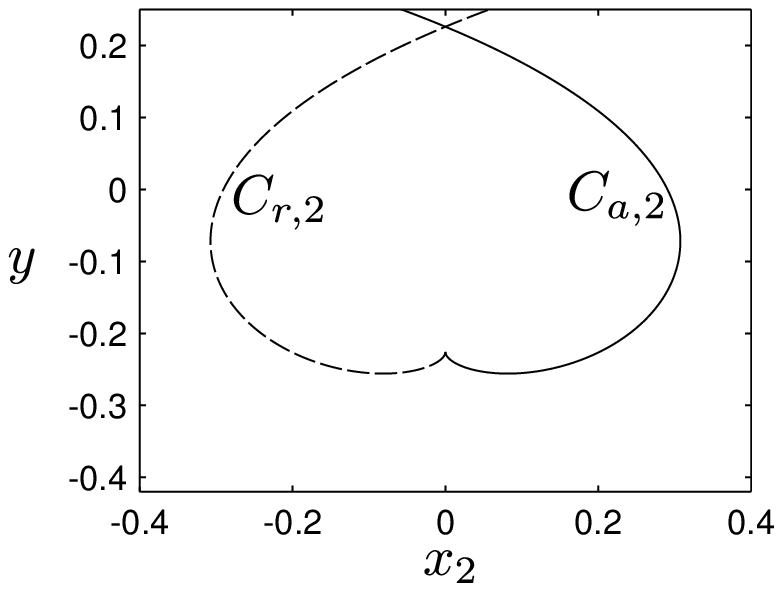}}
\subfigure[$\xi_-=3.5$]{\includegraphics[width=.45\textwidth]{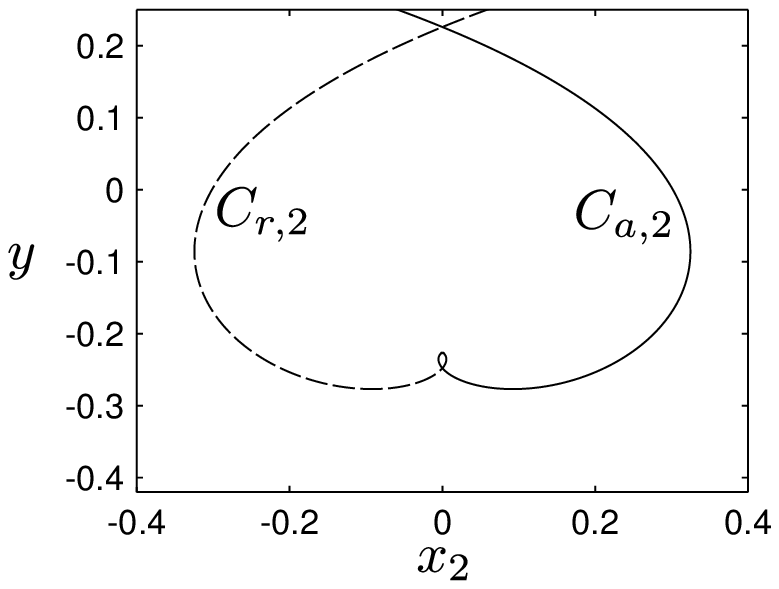}}
\subfigure[$\xi_-=5$]{\includegraphics[width=.45\textwidth]{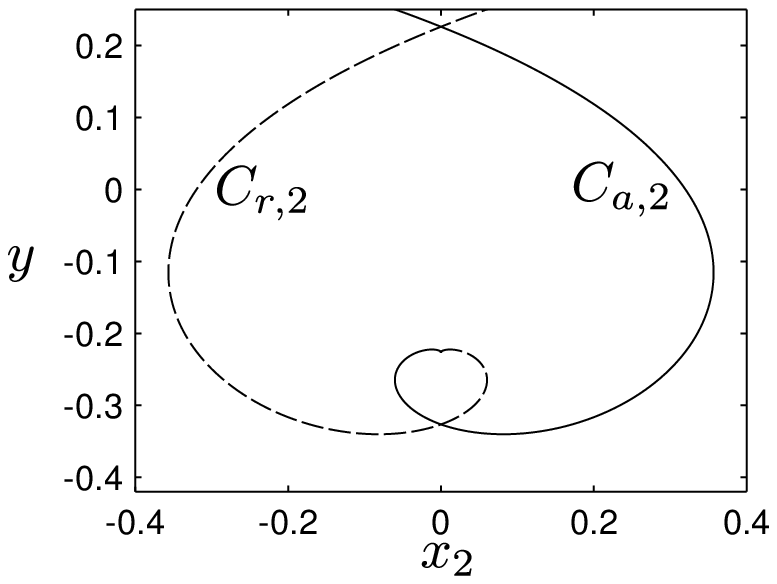}}
\subfigure[$\xi_-=6.5$]{\includegraphics[width=.45\textwidth]{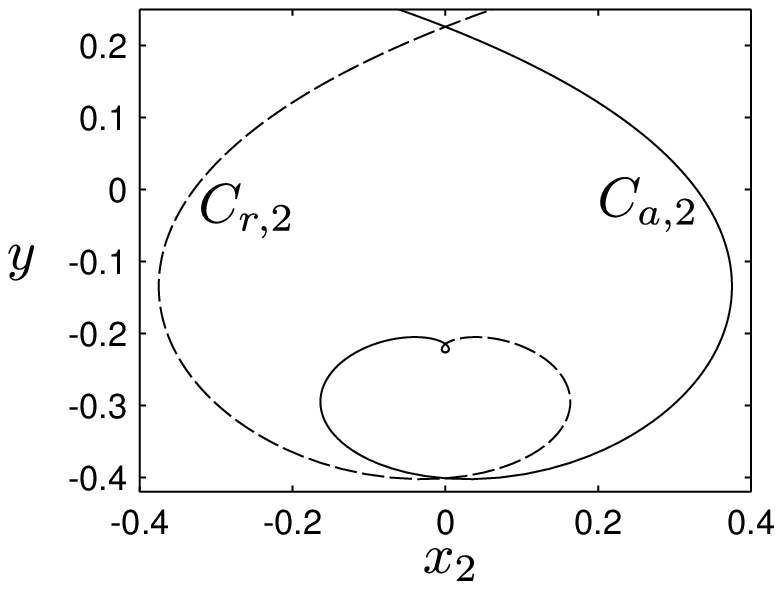}}
\end{center}
\caption{The intersections $C_{a,2}\cap \{z_2=0\}$ and $C_{r,2}\cap \{z_2=0\}$ for different values of $\xi_-$. An additional transversal intersection appears upon the passage of a bifurcation-value $\xi_-=2n+1$, $n\in \mathbb N$. The cases $\xi_-=2n$ do not introduce new bifurcations. The new intersections are magnified in \figref{secondaryCanardsZoom}.}
\figlab{secondaryCanards}
\end{figure}

\begin{figure}[h!] 
\begin{center}
\subfigure[$\xi_-=2$]{\includegraphics[width=.45\textwidth]{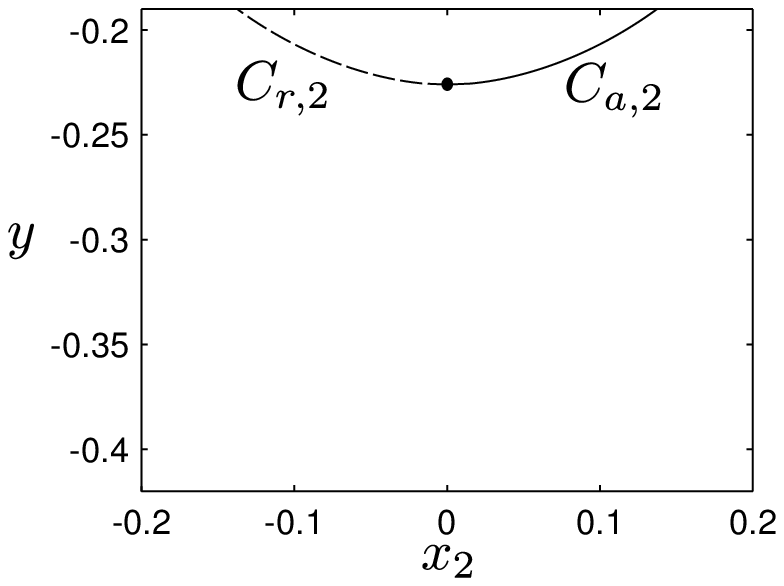}}
\subfigure[$\xi_-=2.5$]{\includegraphics[width=.45\textwidth]{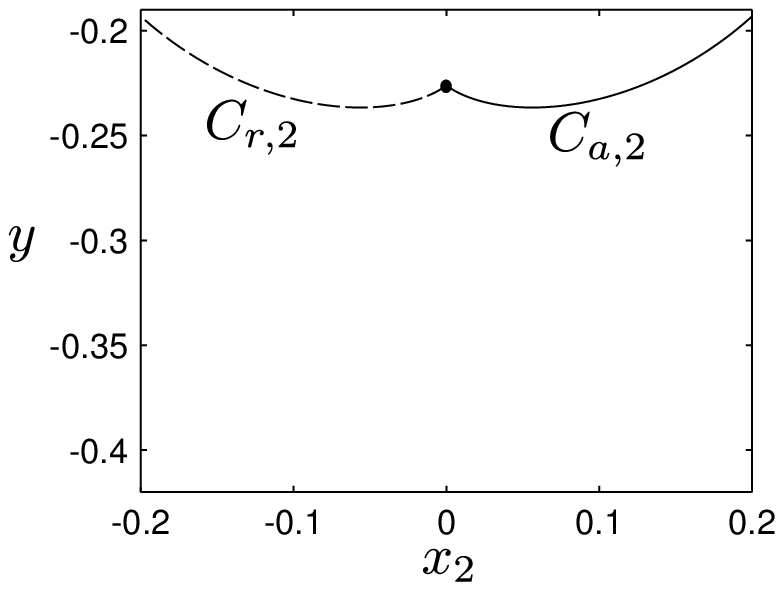}}
\subfigure[$\xi_-=3$]{\includegraphics[width=.45\textwidth]{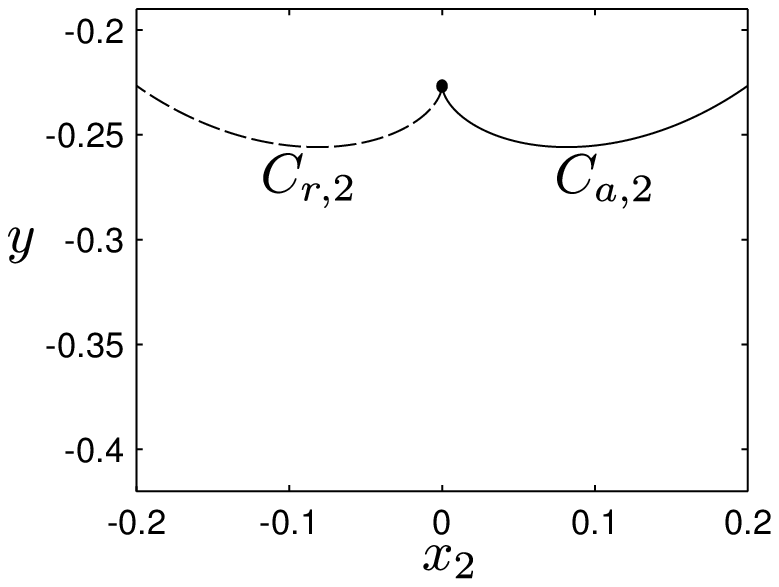}}
\subfigure[$\xi_-=3.5$]{\includegraphics[width=.45\textwidth]{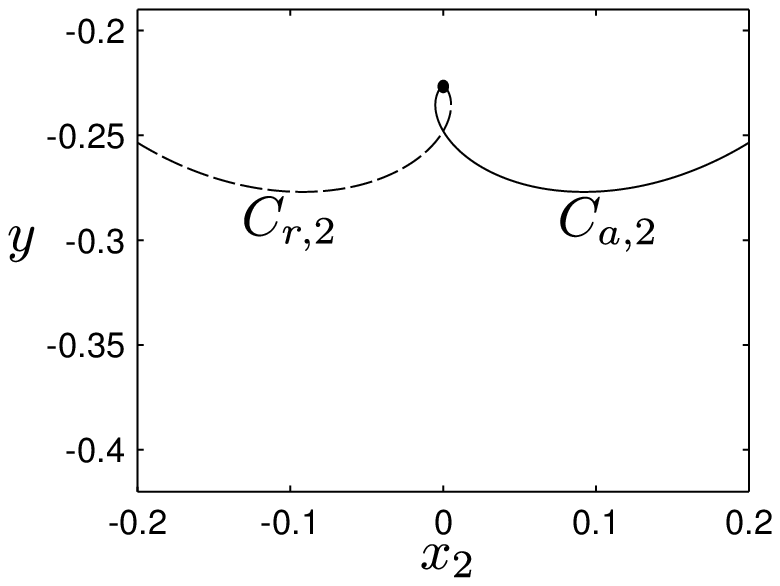}}
\subfigure[$\xi_-=5$]{\includegraphics[width=.45\textwidth]{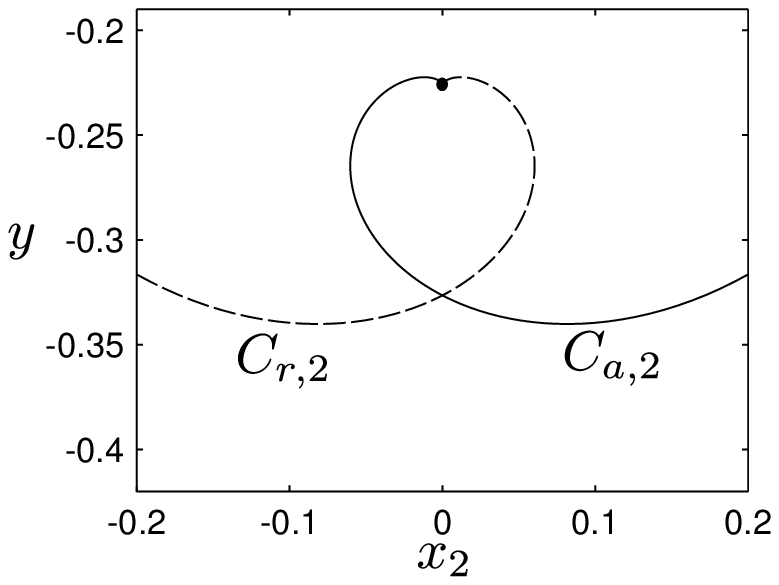}}
\subfigure[$\xi_-=6.5$]{\includegraphics[width=.45\textwidth]{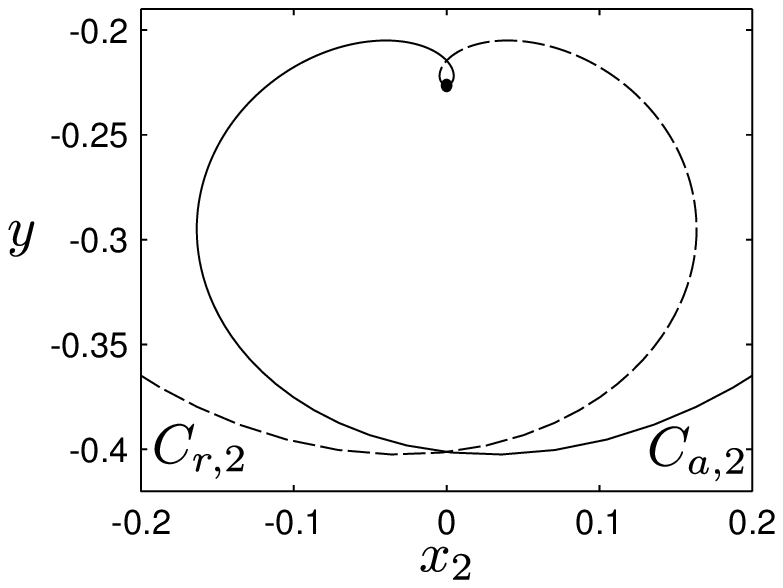}}
\end{center}
\caption{Zoom in on the vicinity of the weak canards (indicated by the black dots) that are visible in the figures in \figref{secondaryCanards}. The additional intersections are illustrated by $\circ$'s. }
\figlab{secondaryCanardsZoom}
\end{figure}
\begin{figure}[h!] 
\begin{center}
\subfigure[$\xi_-=3.5$]{\includegraphics[width=.45\textwidth]{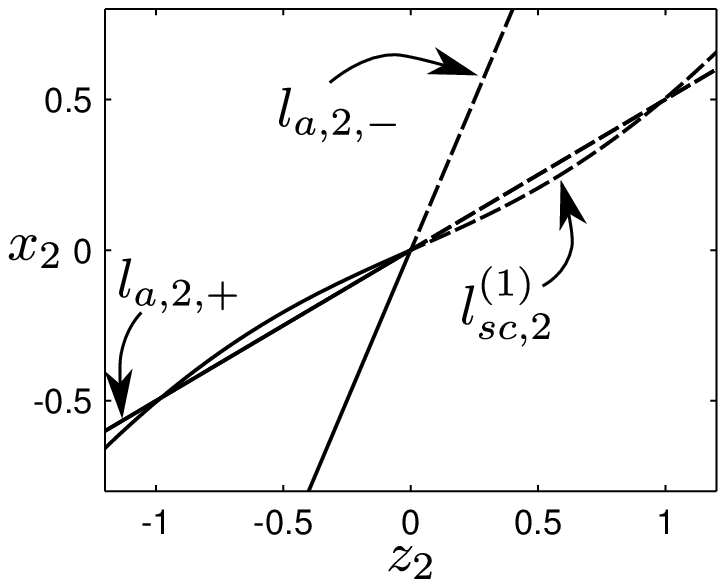}}
\subfigure[$\xi_-=4.5$]{\includegraphics[width=.45\textwidth]{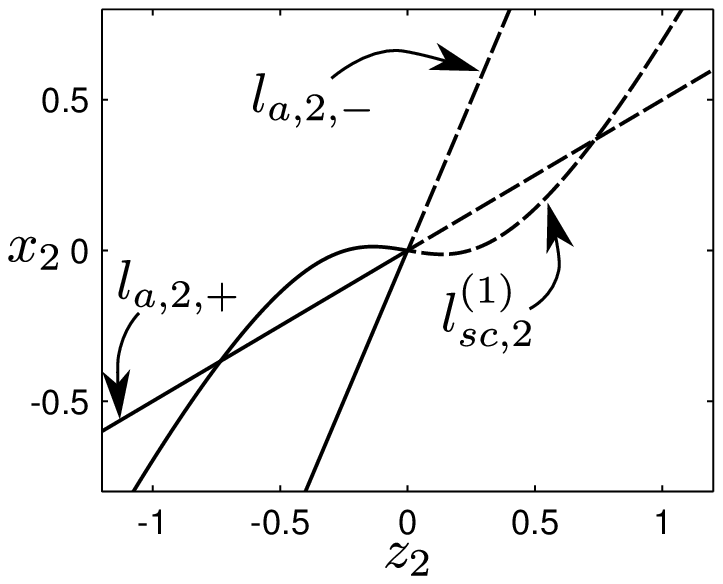}}
\subfigure[$\xi_-=6.5$]{\includegraphics[width=.45\textwidth]{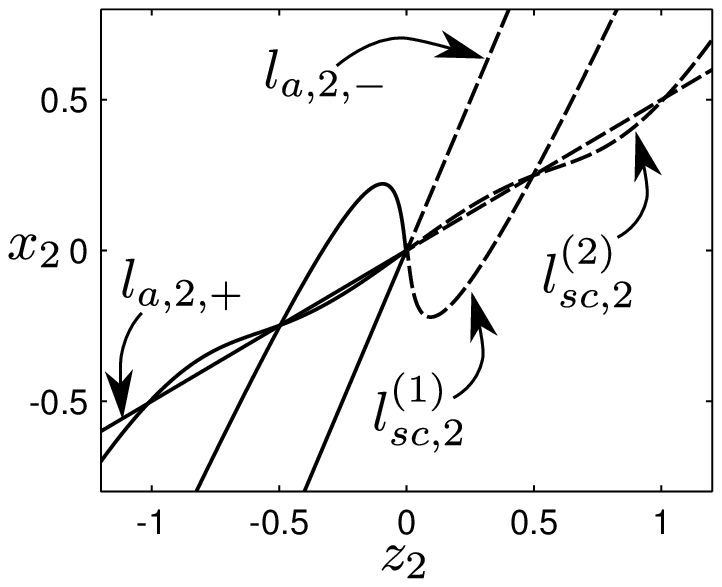}}
\end{center}
\caption{Illustration of the rotational properties of the bifurcating secondary canards. The secondary canards $l_{sc,2}^{(1)}$ and $l_{sc,2}^{(2)}$ are seen to rotate $360^{\circ}$ respectively $720^\circ$ about the weak canard $l_{a,2,+}$. }
\figlab{C12ZEq0Mu4_5wRotation}
\end{figure}

\begin{figure}[h!]
\begin{center}
 {\includegraphics[width=.45\textwidth]{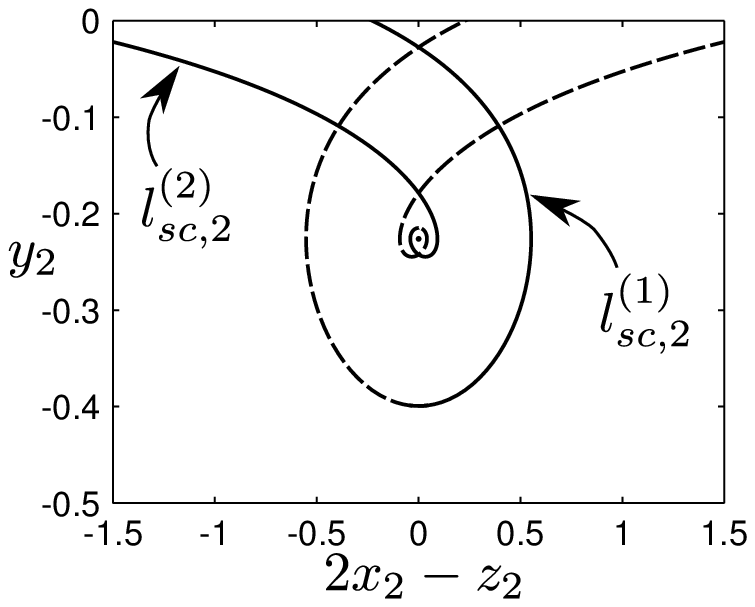}}
 \end{center}
 \caption{Illustration of the rotational properties of the bifurcating secondary canards for $\xi_-=6.5$ using a projection onto the plane described by $(2x_2-z_2,y_2)$.}
 \figlab{C12ZEq0Mu6.5wRotationNewDetails}
\end{figure}

\section{Future work}\seclab{future}
In $\mathbb R^2$, two-folds are co-dimension one \cite{filippov1988differential,kuznetsov_one-parameter_2003}. As in $\mathbb R^3$, they are associated with singular canards. But, unlike in $\mathbb R^3$, they are also connected to the existence and the non-existence of pseudo-equilibria (see e.g. Figs. 10 and 11 in \cite{kuznetsov_one-parameter_2003}). In \cite{krkri_2014} we consider the regularization of two-folds in $\mathbb R^2$.

\section{Conclusions}\seclab{conclusion}
In this paper, we have brought together three different areas of nonlinear dynamics; canards, piecewise smooth systems and blow up methods, each of which is currently attracting intense interest amongst researchers. The study of canards in smooth dynamical systems, begun with \cite{Benoit81}, continues to deliver surprises \cite{wechselberger_existence_2005}. Similarly, piecewise smooth systems \cite{Bernardo08, filippov1988differential,Utkin77}, which pose fundamental theoretical problems as well as being widely applicable, also continue to yield new results \cite{jeffrey_geometry_2011}. Finally, blow up methods, whose origins go back over a century \cite{Alvarez11}, have recently received a boost by their applications to geometric singular perturbation theory \cite{krupa_extending_2001}.

Recently, it was shown that canards are naturally linked to the two-fold singularities of piecewise smooth systems \cite{Desroches13,desroches_canards_2011,  Prohens13}. It is a natural question to ask what happens to these (singular) canards when the underlying piecewise smooth system is regularized. 
In this paper we have used the blow up method of Dumortier and Roussarie \cite{dumortier_1991,dumortier_1993,dumortier_1996} in the formulation of Krupa and Szmolyan \cite{krupa_extending_2001} to study the effect of regularization on canards of two-fold singularities in piecewise smooth dynamical systems \cite{filippov1988differential}. As many previous authors have done, we used the regularization approach of Sotomayor and Teixeira \cite{Sotomayor96}. But before examining the two-fold, we demonstrated the power of our approach by considering the simpler case of the fold line away from the two-fold. This is also an area that is the subject of current research, in its own right. There are no canards in this case, but traditional slow-fast theory still fails close to the fold line. This problem has been considered by Bonet and Seara \cite{reves_regularization_2014}, who used asymptotic methods to show that topological features of the associated piecewise smooth bifurcations are preserved under regularization. Using our approach, we recovered one of their main results in a simple manner. 
Then, for the two-fold, we showed that the regularized system only fully retains the features of the singular canards in the piecewise smooth system in the cases when the sliding region does not include a full sector of singular canards. In particular, we showed that every locally unique primary singular canard persists under regularization. For the case of a sector of primary singular canards, we showed that the regularized system contains a canard, provided a certain non-resonance condition holds true. Finally, we provided numerical evidence for the existence of secondary canards near resonances. 



\bibliography{refs}

\begin{thebibliography}{10}

\bibitem{AbramowitzStegun1964}
M.~Abramowitz and A.~Stegun.
\newblock {\em Handbook of Mathematical Functions}, volume~55.
\newblock National Bureau of Standards Applied Mathematics Series, Washington,
  DC, 1964.

\bibitem{Alvarez11}
M.~J. \'Alvarez, A.~Ferragut, and X.~Jarque.
\newblock A survey on the blow up technique.
\newblock {\em Int. J. Bif. Chaos}, 21:3103--3118, 2011.

\bibitem{Benoit81}
E.~Beno\^{i}t, J.~L. Callot, F.~Diener, and M.~Diener.
\newblock Chasse au canard.
\newblock {\em Collect. Math.}, 31-32:37--119, 1981.

\bibitem{brons-krupa-wechselberger2006:mixed-mode-oscil}
M.~Br{\o}ns, M.~Krupa, and M.~Wechselberger.
\newblock Mixed mode oscillations due to the generalized canard phenomenon.
\newblock In W.~Nagata and N.~Sri Namachchivaya, editors, {\em Bifurcation
  Theory and Spatio-Temporal Pattern Formation}, volume~49 of {\em Fields
  Institute Communications}, pages 39--64. American Mathematical Society, 2006.

\bibitem{Carmona08}
V.~Carmona, F.~Fern\'andez-S\'anchez, and A.~E. Teruel.
\newblock Existence of a reversible {T}-point heteroclinic cycle in a piecewise
  linear version of the {M}ichelson system.
\newblock {\em {SIAM} Journal on Applied Dynamical Systems}, 7:1032--1048,
  2008.

\bibitem{casey06}
R.~Casey, H.~{de Jong}, and J.-L. Gouz\'e.
\newblock Piecewise-linear models of genetic regulatory networks: {E}quilibria
  and their stability.
\newblock {\em J. Math. Biol.}, 52:27--56, 2006.

\bibitem{cortez2008}
J.~Cortez.
\newblock {Discontinuous dynamical systems}.
\newblock {\em {Control Systems, IEEE}}, {28}({3}):{36--73}, {2008}.

\bibitem{Desroches13}
M.~Desroches, E.~Freire, S.~J. Hogan, E.~Ponce, and P.~Thota.
\newblock Canards in piecewise-linear systems: explosions and super-explosions.
\newblock {\em Proc. Roy. Soc. Lond. A.}, 469:20120603, 2013.

\bibitem{desroches_canards_2011}
M.~Desroches and M.~R. Jeffrey.
\newblock Canards and curvature: nonsmooth approximation by pinching.
\newblock {\em Nonlinearity}, 24(5):1655--1682, May 2011.

\bibitem{Bernardo08}
M.~di~Bernardo, C.~J. Budd, A.~R. Champneys, and P.~Kowalczyk.
\newblock {\em Piecewise-smooth Dynamical Systems: Theory and Applications}.
\newblock Springer Verlag, 2008.

\bibitem{dumortier_1991}
F.~Dumortier.
\newblock Local study of planar vector fields: Singularities and their
  unfoldings.
\newblock In H.~W.~Broer et~al, editor, {\em Structures in Dynamics, Finite
  Dimensional Deterministic Studies}, volume~2, pages 161--241. Springer
  Netherlands, 1991.

\bibitem{dumortier_1993}
F.~Dumortier.
\newblock Techniques in the theory of local bifurcations: Blow-up, normal
  forms, nilpotent bifurcations, singular perturbations.
\newblock In D.~Schlomiuk, editor, {\em Bifurcations and Periodic Orbits of
  Vector Fields}, volume 408 of {\em NATO ASI Series}, pages 19--73. Springer
  Netherlands, 1993.

\bibitem{dumortier_1996}
F.~Dumortier and R.~Roussarie.
\newblock Canard cycles and center manifolds.
\newblock {\em Mem. Amer. Math. Soc.}, 121:1--96, 1996.

\bibitem{ISI:000174961100007}
T.~Engleder, P.~Vielsack, and K.~Schweizerhof.
\newblock {FE-regularization of non-smooth vibrations due to friction and
  impacts}.
\newblock {\em {Comp. Mech.}}, {28}({2}):{162--168}, {2002}.

\bibitem{fen1}
N.~Fenichel.
\newblock Persistence and smoothness of invariant manifolds for flows.
\newblock {\em Indiana University Mathematics Journal}, 21:193--226, 1971.

\bibitem{fen2}
N.~Fenichel.
\newblock Asymptotic stability with rate conditions.
\newblock {\em Indiana University Mathematics Journal}, 23:1109--1137, 1974.

\bibitem{fen3}
N.~Fenichel.
\newblock Geometric singular perturbation theory for ordinary differential
  equations.
\newblock {\em J. Diff. Eq.}, 31:53--98, 1979.

\bibitem{filippov1988differential}
A.F. Filippov.
\newblock {\em Differential Equations with Discontinuous Righthand Sides}.
\newblock Mathematics and its Applications. Kluwer Academic Publishers, 1988.

\bibitem{Guckenheimer97}
J.~Guckenheimer and P.~Holmes.
\newblock {\em Nonlinear Oscillations, Dynamical Systems and Bifurcations of
  Vector Fields}.
\newblock Springer Verlag, 5th edition, 1997.

\bibitem{jeffrey_geometry_2011}
M.~R. Jeffrey and S.~J. Hogan.
\newblock The geometry of generic sliding bifurcations.
\newblock {\em {SIAM} Review}, 53(3):505--525, January 2011.

\bibitem{jones_1995}
C.K.R.T. Jones.
\newblock {\em Geometric Singular Perturbation Theory, Lecture Notes in
  Mathematics, Dynamical Systems (Montecatini Terme)}.
\newblock Springer, Berlin, 1995.

\bibitem{kosiuk_scaling_2011}
I.~Kosiuk and P.~Szmolyan.
\newblock Scaling in singular perturbation problems: Blowing up a relaxation
  oscillator.
\newblock {\em {SIAM} Journal on Applied Dynamical Systems}, 10(4):1307--1343,
  January 2011.

\bibitem{krkri_2014}
K.~Uldall Kristiansen and S.~J. Hogan.
\newblock Regularizations of fold bifurcations in planar piecewise smooth
  systems.
\newblock {\em In preparation}, 2014.

\bibitem{krupa_extending_2001}
M.~Krupa and P.~Szmolyan.
\newblock Extending geometric singular perturbation theory to nonhyperbolic
  points - fold and canard points in two dimensions.
\newblock {\em {SIAM} Journal on Mathematical Analysis}, 33(2):286--314, 2001.

\bibitem{krupa_extending_2001_nonlinearity}
M.~Krupa and P.~Szmolyan.
\newblock Extending slow manifolds near transcritical and pitchfork
  singularities.
\newblock {\em Nonlinearity}, 14(6):1473--1491, 2001.

\bibitem{krupa_local_2010}
M.~Krupa and M.~Wechselberger.
\newblock Local analysis near a folded saddle-node singularity.
\newblock {\em Journal of Differential Equations}, 248(12):2841--2888, June
  2010.

\bibitem{2014arXiv1403.5658K}
C.~{Kuehn} and P.~{Szmolyan}.
\newblock {Multiscale geometry of the Olsen model and non-classical relaxation
  oscillations}.
\newblock {\em {arXiv} preprint {arXiv:1403.5658}}, 2014.

\bibitem{kuznetsov_one-parameter_2003}
Yu~A. Kuznetsov, S.~Rinaldi, and Alessandra Gragnani.
\newblock One-parameter bifurcations in planar filippov systems.
\newblock {\em International Journal of Bifurcation and chaos},
  13(08):2157--2188, 2003.

\bibitem{Llibre07}
J.~Llibre, P.~R. da~Silva, and M.~A. Teixeira.
\newblock Regularization of discontinuous vector fields on $\mathbb{R}^3$ via
  singular perturbation.
\newblock {\em J. Dyn. Diff. Eq.}, 3:309--331, 1997.

\bibitem{llibre_sliding_2008}
J.~Llibre, P.~R. da~Silva, and M.~A. Teixeira.
\newblock Sliding vector fields via slow-fast systems.
\newblock {\em Bulletin of the Belgian Mathematical Society-Simon Stevin},
  15(5):851--869, 2008.

\bibitem{Llibre97}
J.~Llibre and M.~A. Teixeira.
\newblock Regularization of discontinuous vector fields in dimension three.
\newblock {\em Discr. Cont. Dyn. Sys.}, 3:235--241, 1997.

\bibitem{MakarenkovLamb12}
O.~Makarenkov and J.~S.~W. Lamb.
\newblock Dynamics and bifurcation of nonsmooth systems: {A} survey.
\newblock {\em Physica D}, 241:1826--1844, 2012.

\bibitem{ISI:000246954500016}
S.~McNamara.
\newblock {Rigid and quasi-rigid theories of granular media}.
\newblock In {Eberhard, P}, editor, {\em {IUTAM Symposium on Multiscale
  Problems in Multibody System Contacts}}, volume~{1} of {\em {IUTAM
  Bookseries}}, pages {163--172}. {IUTAM}, {2007}.
\newblock {IUTAM Symposium on Multiscale Problems in Multibody System Contacts,
  Stuttgart, Germany, Feb 20-23, 2006}.

\bibitem{Michelson86}
D.~Michelson.
\newblock Steady solutions of the {K}uramoto-{S}ivashinsky equation.
\newblock {\em Physica D}, 19:89--111, 1986.

\bibitem{Prohens13}
R.~Prohens and A.~E. Teruel.
\newblock Canard trajectories in 3{D} piecewise linear systems.
\newblock {\em Discr. Cont. Dyn. Sys.}, 33:4595--4611, 2013.

\bibitem{reves_regularization_2014}
C.~B. Reves and T.~M. Seara.
\newblock Regularization of sliding global bifurcations derived from the local
  fold singularity of {F}ilippov systems.
\newblock {\em {arXiv} preprint {arXiv:1402.5237}}, 2014.

\bibitem{Sotomayor96}
J.~Sotomayor and M.~A. Teixeira.
\newblock Regularization of discontinuous vector fields.
\newblock In {\em Proceedings of the International Conference on Differential
  Equations, Lisboa}, pages 207--223, 1996.

\bibitem{ISI:000246954500026}
W.~Stamm and A.~Fidlin.
\newblock {Regularization of 2D frictional contacts for rigid body dynamics}.
\newblock In {Eberhard, P}, editor, {\em {IUTAM Symposium on Multiscale
  Problems in Multibody System Contacts}}, volume~{1} of {\em {IUTAM
  Bookseries}}, pages {291--300}. {IUTAM}, {2007}.
\newblock {IUTAM Symposium on Multiscale Problems in Multibody System Contacts,
  Stuttgart, Germany, Feb 20-23, 2006}.

\bibitem{szmolyan_canards_2001}
P.~Szmolyan and M.~Wechselberger.
\newblock Canards in $\mathbb{R}^3$.
\newblock {\em J. Diff. Eq.}, 177(2):419--453, December 2001.

\bibitem{teixeira_bifurcations_1993}
M.~A. Teixeira.
\newblock Generic bifurcations of sliding vector fields.
\newblock {\em J. Math. Anal. Appl.}, 176:436--457, 1993.

\bibitem{Utkin77}
V.~I. Utkin.
\newblock Variable structure systems with sliding modes.
\newblock {\em IEEE Trans. Automatic Control}, 22:212--222, 1977.

\bibitem{wechselberger_existence_2005}
M.~Wechselberger.
\newblock Existence and bifurcation of canards in $\mathbb{R}^3$ in the case of
  a folded node.
\newblock {\em {SIAM} Journal on Applied Dynamical Systems}, 4(1):101--139,
  January 2005.

\end{thebibliography}
\bibliographystyle{plain}

\newpage

 \end{document}